\documentclass[11pt,letterpaper,reqno]{amsart}
\usepackage{url}
\usepackage{amsthm}
\usepackage{amsmath}
\usepackage{amssymb}
\usepackage{amsfonts}

\usepackage{enumitem}
\usepackage{mathscinet}
\usepackage{lipsum}
\usepackage[english]{babel}
\usepackage[autostyle]{csquotes}
\usepackage{bbold}

\usepackage[english]{babel}

\usepackage[colorlinks=true,
    linkcolor=red,
    citecolor=blue]{hyperref}

\usepackage{graphicx}

\usepackage[utf8]{inputenc}
\usepackage[english]{babel}

\usepackage{color}
\usepackage{comment}

\newtheorem{thm}{Theorem}[section]
\newtheorem{definition}[thm]{Definition}
\newtheorem{theorem}[thm]{Theorem}
\newtheorem*{oseledec*}{Oseledec Theorem}
\newtheorem{cor}[thm]{Corollary}

\newtheorem{claim}[thm]{Claim}
\newtheorem{lemma}[thm]{Lemma}
\newtheorem{prop}[thm]{Proposition}
\theoremstyle{definition}
\newtheorem{remark}[thm]{Remark}

\makeatletter
\def\moverlay{\mathpalette\mov@rlay}
\def\mov@rlay#1#2{\leavevmode\vtop{%
   \baselineskip\z@skip \lineskiplimit-\maxdimen
   \ialign{\hfil$\m@th#1##$\hfil\cr#2\crcr}}}
\newcommand{\charfusion}[3][\mathord]{
    #1{\ifx#1\mathop\vphantom{#2}\fi
        \mathpalette\mov@rlay{#2\cr#3}
      }
    \ifx#1\mathop\expandafter\displaylimits\fi}
\makeatother

\let\ul\underline
\let\ol\overline
\let\wh\widehat
\let\wt\widetilde

\newcommand{\nocontentsline}[3]{}
\newcommand{\tocless}[2]{\bgroup\let\addcontentsline=\nocontentsline#1{#2}\egroup}

\usepackage{microtype}

\usepackage{wasysym}

\def\Jac{\ensuremath{\mathrm{Jac}}}

\def\Vol{\ensuremath{\mathrm{Vol}}}

\def\H{\ensuremath{\mathrm{H}}}

\def\L{\ensuremath{\widehat{\mathcal{L}}}}

\def\U{\ensuremath{\widehat{\mathcal{U}}}}

\def\Uu{\ensuremath{{\mathcal{U}}}}

\def\Ll{\ensuremath{{\mathcal{L}}}}

\def\B{\ensuremath{\widehat{\mathfrak{B}}}}

\def\HH{\ensuremath{\widehat{\mathfrak{H}}}}

\def\Hh{\ensuremath{\mathfrak{H}}}

\title[Thermodynamic Formalism Out of Equilibrium]{Thermodynamic Formalism Out of Equilibrium Part II: Semi-Ruelle Operator, Conformal Measures, and Effective Expansion and Quasi-Compactness}

\begin{document}

\author{S. Ben Ovadia}
\date{}

\begin{abstract}
We introduce a general machinery to study {\em thermodynamic formalism out of equilibrium}: The thermodynamics of a topological Markov shift (denoted by $\Sigma^-$) where the potential is given by a random walk on a compact metric space $X$ (and the randomness is driven by a Gibbs process). We introduce the {\em semi-Ruelle operator}, which acts on $C(\Sigma^-\times X)$. We construct conformal measures and harmonic functions for the semi-Ruelle operator. We present a few applications: (1) We provide a new proof to the POE variational principle (POE stands for the {\em pressure out of equilibrium} which is associated with the process), and we show that maximizing measures in the POE variational principle admit positive {\em entropy out of equilibrium}, and satisfy {\em semi-Gibbs estimates}. (2) In the setting where the random walk on the fiber $X$ is given by $C^{1+}$ diffeomorphisms (which are allowed to be very dissipative), and it satisfies the open condition of {\em effective expansion on average}, we show that the {\em averaged semi-Ruelle operator} is  quasi-compact when acting on a Sobolev function space. An application includes proving a spectral gap when assuming volume decay of correlations, and proving bounds on the dimension of stationary measure in terms of similarity dimension.
\end{abstract}

\newcommand{\Addresses}{{% additional braces for segregating \footnotesize
  \bigskip
  \footnotesize

  S.~Ben Ovadia, \textsc{Einstein Institute of Mathematics, The Hebrew University of Jerusalem, 91904 Jerusalem, Israel}. \\ \textit{E-mail address}: \texttt{Snir.BenOvadia@mail.huji.ac.il}
  
}}

\maketitle

\tableofcontents

\section{Introduction}

Classical thermodynamic formalism, pioneered by Ruelle and Bowen (see \cite{Ruelle67,B4,RuelleBookThermoDynFor}), provides a powerful framework for singling out invariant measures of importance for chaotic dynamical systems. Thermodynamic formalism allows studying the statistical properties of the system via methods involving transfer operator (called the Ruelle operator in certain contexts), variational principle, and entropy. These techniques have been extended to non-compact systems (see \cite{SarigTDFSymposium,SarigPR,SarigNR}), and to non-uniformly hyperbolic systems (see \cite{Y,Y99,CPZ,IT}).

In \cite{TDFOE_I}, we introduce a framework to study the {\em pressure out of equilibrium} (POE for short) which is associated with a random family of potentials a compact topological Markov shift, where the randomness is driven by a Gibbs process random walk on a compact metric space. While a variational principle for the POE was successfully established in that setting, a complete operator-theoretic approach analogous to Ruelle's classical transfer operator theory has remained elusive. 

The primary objective of this paper is to bridge this gap by developing a complementary theory of adapted operators for these models of random dynamics which lead to the study of thermodynamic formalism out of equilibrium. To this end, we introduce the {\em semi-Ruelle operator}, $\L_{\wh\phi}$, an operator-theoretic analogue designed for compact topological Markov shifts (TMSs) coupled with random fiber dynamics. Unlike the standard Ruelle operator, which quotients out symbolic stable leaves to study dynamics on the space of symbolic local unstable leaves, the semi-Ruelle operator only quotients out the stable leaves of the base dynamics (the TMS $\Sigma$) while leaving the fiber dynamics ($X$) unquotiented.

This structural asymmetry introduces a major technical hurdle: the semi-Ruelle operator does not generally contract a cone in standard H\"older spaces ($\mathrm{H\ddot{o}l}(\Sigma \times X)$) due to the lack of a uniform modulus of continuity in the fiber coordinate. This makes the construction of eigen-measures and eigen-functions non-trivial. In fact, in \textsection \ref{AppCount}, we describe an example where an eigen-function does not exist in any reasonable space! Even when we forfeit any regularity requirements for the eigen-function.

Nonetheless, we begin by constructing {\em conformal measures} (eigen-measures of the dual operator), which require fewer topological prerequisites. We continue to present necessary and sufficient conditions for the existence of {\em harmonic function} (eigen-functions corresponding to the maximal eigen-value), with mild regularity requirements. .

Beyond providing a functional-analytic framework, this operator-theoretic approach yields several significant applications to the thermodynamic formalism out of equilibrium:

\begin{enumerate}
	\item New proof of the POE variational principle (independent of the variational principle established in \cite{TDFOE_I}).
	\item The maximizing measures of the POE variational principle which we construct satisfy strong thermodynamic properties (such as positive {\em entropy out of equilibrium}, thus positive metric entropy, and semi-Gibbs estimates, which are analogous the Gibbs estimates for Gibbs states).
\end{enumerate} 

We then go beyond the topological setting, and continue to study the action of the semi-Ruelle operator when fibers are equipped with a {\em smooth measure}, $m$ (namely, a measure whose Jacobian under the fiber maps is log-H\"older continuous). In that setting (\textsection \ref{GibbsProcess}) we introduce an adapted Hilbert space, $\HH_\kappa(\wh m)$. While in general the semi-Ruelle operator does not preserve the regularity of functions in the fiber direction, $\HH_\kappa(\wh m)$ is a space of function which admit some bounded variation in average, in the fiber direction. We study a potential $\wh \phi$, associated with a random walk on $X$ driven by the Gibbs measure on $\Sigma$ of a H\"older potential $\psi$. We show that the associated semi-Ruelle operator $\L_{\wh \phi}$ acts on $\HH_\kappa(\wh m)$ with a bounded norm, and that it is the dual of the Koopman operator for the skew-product dynamics on $L^2(\wh m)$. We then prove an inequality for the POE of $\wh \phi$, and provide conditions for the existence of a smooth measure which is invariant to the skew-product dynamics.

Finally, an additional application of the tools which we introduce in this paper is in the setting of smooth random dynamics. When the fiber $X$ is a closed Riemannian manifold $M$ of dimension $d\geq 2$, with a normalized Riemannian volume measure $m$, and the fiber maps are $C^{1+}$ diffeomorphisms. We let $\mu$ be the Gibbs state of $\psi$, and study the random dynamics on $M$ driven by $\mu$. We emphasize that we do not make additional assumption, and the diffeomorphisms are allowed to be arbitrarily dissipative. In \textsection \ref{UEASect2} we introduce and study a condition called {\em effective expansion on average} (a $C^1$-open condition, weaker than co-expansion on average for conservative systems, see \textsection \ref{coExImpEffEx}). We show that under effective expansion on average, the {\em averaged semi-Ruelle operator} $\mathcal L=\int \L_{\wh\phi} d\mu$ (see \textsection \ref{ASROSect}) is quasi-compact when acting on a Sobolev space. In particular, this extends the results of \cite{DeWittDolgopyat2} to the dissipative setting using effective expansion on average (rather than coexpansion). In \cite{DeWittDolgopyat3} the authors study i.i.d. (near) conservative random dynamics, under an assumption of co-expansion on average, study statistical properties such as quenched and annealed exponential mixing  (relying on the quasi-compactness properties). In \textsection \ref{appliC2} we show that decay of correlations for the volume implies a spectral gap for the volume averaged semi-Ruelle operator, and prove bounds for the dimension of the stationary measure in terms of the {\em similarity dimension} (see \textsection \ref{UEASect2}). In an upcoming work, using the spectral gap of the averaged semi-Ruelle operator, we study further statistical properties in the dissipative setting (e.g. CLT, LLT).

\subsection*{Acknowledgements}
The author would like to thank Aaron Brown, Dmitry Dolgopyat, and Federico Rodriguez-Hertz for many helpful and patient discussions from which the author had learned a lot, and which have helped improve this manuscript.

\section{Basic Definitions}\label{POE}

\noindent\textbf{Setup:}
\begin{enumerate}
	\item Let $X$ be a compact metric space.
	\item Let $\Sigma$ be a two-sided compact topological Markov shift, endowed with the left-shift $T:\Sigma\to \Sigma$.
	\item Let $F:\Sigma\to C_\alpha(X,X)$ be a H\"older continuous map, denoted by $\omega\mapsto F_\omega$, where $F_\omega$ is a $\alpha$-H\"older homeomorphism of $X$.
	\item  Given $\widehat{\phi}\in \mathrm{H\ddot{o}l}(\Sigma\times X)$, we denote $\phi_t(\omega)=\widehat{\phi}(\omega,t)$. We view the family $\{\phi_t\}_{t\in X}\subseteq \mathrm{H\ddot{o}l}(\Sigma)$ as a equi-H\"older family of potentials:\\ $\sup_{t\in X}\|\phi_t\|_{\mathrm{H\ddot{o}l}}<\infty$, and $\phi_t$ depends in a H\"older continuous manner on $t\in X$.
 	\item We denote $F_\omega^k:=F_{T^{k-1}\omega}\circ \cdots \circ F_\omega$ for $k\geq1$, and $F_\omega^0=\mathrm{Id}$.
\end{enumerate}

\subsection{Pressure Out of Equilibrium}

\begin{definition}[Pressure out of equilibrium {\cite[Definition~2.1]{TDFOE_I}}]\label{DefOfPOE}
	We define the {\em pressure out of equilibrium}, or {\em POE} for short, as $$\widehat{P}(\widehat{\phi}):=\limsup_{n\to\infty}\frac{1}{n}\log\sup_{t\in X}\widehat{Z}_n(\widehat{\phi},t,a),$$ 
	where
	$$\widehat{Z}_n(\widehat{\phi},t,a):=\sum_{|\underline{w}|=n,w_{n-1}=a}e^{\sum_{k=0}^{n-1}\phi_{F^k_{\theta_{\underline{w}}}(t)}(T^k(\theta_{\underline{w}}))},$$
	$[a]\subset \Sigma$, and $\theta_{\underline{w}}\in[\underline{w}]$ maximize $\omega\mapsto \sum_{k=0}^{n-1}\phi_{F^k_{\omega}(t)}(T^k(\omega))$.
\end{definition}

\begin{remark}\label{PBounded}
	Note, $\wh P(\wh\phi)\leq h_\mathrm{top}(T)+\|\wh\phi\|_\infty$.
\end{remark}

\begin{definition}[$\phi$-holonomies {\cite[Definition~6.1]{TDFOE_I}}]\label{phiHolonomies}
    We say that $(\Sigma\times X,\widehat{F})$ admits {\em $\widehat{\phi}$-holonomies}, if $\exists C>0$ s.t. $$\sup_{n\geq0,t\in X}\sup_{\omega,\omega'\in \Sigma: \omega_i=\omega_i', i\in [0,n]}|\widehat{\phi}^{(n)}(\omega,t)-\widehat{\phi}^{(n)}(\omega',t)|\leq C.$$
\end{definition}

\begin{remark}\label{phiHolonomiesRmk}\text{ }
\begin{enumerate}

\item Note, since $\Sigma$ is compact and topologically transitive, the definition of the POE does not depend on $a$. 
Under an assumption of $\widehat{\phi}$-holonomies (see Definition \ref{phiHolonomies}), the definition of the POE also does not depend on the choice of $\theta_{\underline{w}}\in[\underline{w}]$. The condition of $\widehat{\phi}$-holonomies can be satisfied via an open condition of bunching for many interesting examples.
\item Note, $\widehat{Z}_n(\widehat{\phi},t,a)$ is not of the composite form $L_{\phi_n}\circ \cdots \circ L_{\phi_1}\mathbb{1}_{[a]}$! For each word of length $n$, the sequence of potentials in its corresponding weight depends on the word.	
    \item Stable and unstable holonomies (see \cite[Definition~8.8]{TDFOE_I}) together imply $\widehat{\phi}$-holonomies. One can verify this by adding and subtracting the Birkhoff average over the word $[\omega,\omega']$ (i.e., the Smale bracket). In particular, stable and unstable holonomies are an open condition for a large class of partially hyperbolic systems (see \\ \cite[Remark~8.9]{TDFOE_I}).
\end{enumerate}
\end{remark}

\begin{lemma}[{\cite[Lemma2.3]{TDFOE_I}}]\label{limsupIsLim} The following limit exists:
	$$\widehat{P}(\widehat{\phi})=\lim_{n\to\infty}\frac{1}{n}\log\sup_{t\in X}\widehat{Z}_n(\widehat{\phi},t,a)= \lim_{n\to\infty}\frac{1}{n}\log\sum_{a} \widehat{Z}_n(\widehat{\phi},t,a).$$
\end{lemma}

\section{Semi-Ruelle Operator}\label{EntOutOfEq}

\begin{definition}
 Set    $\Sigma^-:=\{(\omega_i)_{i\leq-1}:\omega\in\Sigma\}$. Elements of $\Sigma^-$ are often denoted by $\omega^-$. %Given $\omega^-\in \Sigma$, its associated {\em unstable leaf} is $$V^u_{\leq-1}(\omega^-):=\{\omega\in \Sigma: \forall i\leq -1, \ \omega_i=\omega^-_i\}.$$
     Let $T_R:\Sigma^-\to \Sigma^-$ be the right-shift.
\end{definition}

\begin{remark}\label{afterSinaiRuelle}
	In \cite[Lemma~B.1]{TDFOE_I} we show that if $(\Sigma\times X,\widehat{F})$ admits unstable holonoimes (\cite[Definition~8.8]{TDFOE_I}) and $\widehat{\phi}$-holonomies, then there exists a bi-H\"older map $\widehat{C}:\Sigma\times X\to\Sigma\times X$ which preserves fibers s.t. $\widetilde{F}:=\widehat{C}^{-1}\circ \widehat{F}\circ \widehat{C}$ satisfies $\widetilde{F}_\omega=\widetilde{F}_{(\omega_i)_{i\leq0}}$ for all $\omega\in \Sigma$. Note, $\wh F$ maps the $X$-fiber over $\omega$ to the $X$-fiber over $T \omega$. Then, $\wh F^{-1}$ maps the $X$-fiber over $\omega$ to the $X$-fiber over $T^{-1} \omega$, that is it is constant on element of the form of $V^u_{\leq -1}$, and the map over $\{\omega\}\times X$ is in fact $T^{-1}\times F_{T^{-1}\omega}^{-1}:\{\omega\}\times X\to \{T^{-1}\omega\}\times X$. 
Lemma \ref{SinaiB} below expands on this notion, and demonstrates how to reduce to the case where $\widehat\phi(\omega,t)= \widehat\phi((\omega_i)_{i\leq 0},t) $.
\end{remark}

\begin{lemma}\label{SinaiB} Let $\widetilde{\phi}:=\widehat{\phi}\circ \widehat{C}\in \mathrm{H\ddot{o}l}(\Sigma\times X)$. Then there exists $\widehat{A}\in  \mathrm{H\ddot{o}l}(\Sigma\times X)$ s.t. $\widetilde{\phi}^-:=\widetilde{\phi}-\widehat{A}\circ\widetilde{F}-\widehat{A}$ satisfies $\widetilde{\phi}^-(\omega,t)=\widetilde{\phi}^-((\omega_i)_{i\leq0},t)$ for all $\omega\in \Sigma$ and $t\in X$.
\end{lemma}
\begin{proof}
    For all $[a]\subseteq \Sigma $ fix $\omega_a\in[a]$. Given $\omega\in \Sigma$ set $\omega^*:=[\omega,\omega_{\omega_0}]$, (i.e. the Smale brackets). Define
$$\widehat{A}:=\sum_{k\geq0}\widetilde{\phi}\circ \widetilde{F}^{-k}(\omega,t)-\widetilde{\phi}\circ \widetilde{F}^{-k}(\omega^*,t).$$
We show that $\widehat{A}$ is well-defined and bounded. Indeed, since $(\omega_i)_{i\leq0}=(\omega^*_i)_{i\leq0}$, $s_k:= \widetilde{F}^{-k}(\omega,t)= \widetilde{F}^{-k}(\omega^*,t)$, and so $\widetilde{\phi}\circ \widetilde{F}^{-k}(\omega,t)-\widetilde{\phi}\circ \widetilde{F}^{-k}(\omega^*,t)=\widetilde{\phi}(T^{-k}\omega,s_k)-\widetilde{\phi}(T^{-k}\omega^*,s_k)$ is exponentially small uniformly in $\omega$ and $t$, as $\widetilde{\phi}$ is H\"older continuous. 

One can now check that $\widetilde{\phi}^-:=\widetilde{\phi}+\widehat{A}\circ\widetilde{F}-\widehat{A}$ satisfies $\widetilde{\phi}^-(\omega,t)=\widetilde{\phi}^-((\omega_i)_{i\leq0},t)$ for all $\omega\in \Sigma$ and $t\in X$ as desired. To see that $\widehat{A}$ is H\"older continuous, let $(\omega,t)$ and $(\omega',t')$ with $d(\omega,\omega'),d(t,t')\leq e^{-n}$. Decompose the series which defines $\widehat{A}$ into a sum from $k=0$ to $k=\epsilon n$, and the tail series. By choosing a respective $\epsilon>0$ (uniform), since the tail series is exponentially small and the sum is composed of H\"older continuous functions, we get that $|\widehat{A}(\omega,t)-\widehat{A}(\omega',t')|\leq C_\epsilon \theta_\epsilon^n$.
\end{proof}

\begin{remark}
	Given $\widetilde\phi^-$ from Lemma \ref{SinaiB}, assuming that $F_\omega=F_{(\omega_i)_{i\leq0}}$ (recall Remark \ref{afterSinaiRuelle}), $\wh\phi^-:=\widetilde\phi^- \circ \wh F^{-1}$ satisfies $\wh\phi^-(\omega,t)= \wh\phi^-((\omega_i)_{i\leq-1},t) $. 
\end{remark}

\begin{definition}
In the setting where $F_{\omega}= F_{(\omega_i)_{i\leq0}} $, the {\em inverse-cocycle} $(\omega,t)\mapsto (T^{-1}\omega, F_{T^{-1}\omega}^{-1}(t))$ is costant on elements of the form
$$V^u _{\leq-1}(\omega^-):=\Big\{\omega\in \Sigma: \omega_i=\omega^-_i, \ \forall i\leq-1\Big\},$$
where $\omega^-\in \Sigma^-$. We therefore can define 
$$F_{\omega^-}^{-1}:=\text{the common map }F_{T^{-1}\omega}^{-1}\text{ on } V^u _{\leq-1}(\omega^-)_{\leq-1}.$$
\end{definition}

\begin{definition}[Non-invertible dynamics]
	In the setting where $F_{\omega}= F_{(\omega_i)_{i\leq0}} $, the cocycle $\wh F:\Sigma\times X\to \Sigma\times X$ induces a non-invertible cocycle on $\Sigma^-\times X$. We denote the non-invertible cocycle by $\wh F_R:\Sigma^-\times X\to\Sigma^-\times X$,  which is given by $$\wh F_R(\omega^-,t):=(T_R\omega^-,F_{\omega^-}^{-1}(t)).$$
\end{definition}

\begin{definition}[Semi-Ruelle Operator]\label{defOfRuelle}
Given $\wh\phi\in C(\Sigma^-)$, and $\wh F:\Sigma\times X \to \Sigma \times X$ s.t. $F_\omega=F_{(\omega_i)_{i\leq0}}$,  we define the associated {\em semi-Ruelle operator}, $\L_{\wh\phi}:C(\Sigma^-\times X)\to C(\Sigma^-\times X) $ by
\begin{equation}\label{defOfRuelleEq}
	(\L_{\wh\phi}g)(\omega^-,t):=\sum_{\wh{F}_R(\wt\omega^-,\wt t)=(\omega^-,t)}e^{\wh\phi(\wt\omega^-,\wt t)}g(\wt\omega^-,\wt t).
\end{equation}
\end{definition}

\begin{remark}\text{ }
\begin{enumerate}
	\item Note, the semi-Ruelle operator is a positive, bounded, linear operator on $C(\Sigma^-\times X)$.
	\item The standard Ruelle operator considers the left-shift on the two-sided shift, and quotients by the (symbolic) stable leaves, in order to study the dynamics on the space of local unstable leaves. In our setting, we quotient by the stable leaves of only the base dynamics ($\Sigma$), while leaving the fiber dynamics unquotiented (which may admit stable and unstable leaves nonetheless). Thus, we consider this operator as ``semi" Ruelle.
\end{enumerate}	
\end{remark}

\begin{definition}
	Given $n\geq 1$, write $$\wh\phi^{(-n)}(\omega^-,t):=\sum_{k=0}^{n-1}\wh\phi\circ \wh F^{k}_R(\omega^-,t).$$
\end{definition}

\begin{lemma}[POE is log-inverse-spectral radius]\label{specRad}
	Assume that $(\Sigma\times X,\wh F)$ admits $\wh\phi$-holonomies, then
	$$\lim_{n\to\infty}\frac{1}{n}\log\|\L_{\wh\phi}1\|_\infty=\wh P(\wh\phi).$$
\end{lemma}
\begin{proof}
	First, by iterating \eqref{defOfRuelleEq}, for all $n\geq1$, for all $(\omega^-,t)\in \Sigma^-\times X$,
	$$(\L_{\wh\phi}^ng)(\omega^-,t):=\sum_{\wh{F}^{n}_R(\wt\omega^-,\wt t)=(\omega^-,t)}e^{\wh\phi^{(-n)}(\wt\omega^-,\wt t)}g(\wt\omega^-,\wt t).$$
In particular,
	\begin{equation*}\label{LphinIsZn}
		(\L_{\wh\phi}^n1)(\omega^-,t)=\sum_{\wh{F}^{n}_R(\wt\omega^-,\wt t)=(\omega^-,t)}e^{\wh\phi^{(-n)}(\wt\omega^-,\wt t)}.
	\end{equation*}
	Moreover, since $(\Sigma\times X,\wh F)$ admits $\wh\phi$-holonomies, we have 
	\begin{equation*}\label{LphinIsZn2}
		\|\L_{\wh\phi}^n1\|_\infty=\max_{\omega^-}\max_t\sum_{\wh{F}^{n}_R(\wt\omega^-,\wt t)=(\omega^-,t)}e^{\wh\phi^{(-n)}(\wt\omega^-,\wt t)}=\max_{\omega^-}C^{\pm 1}\wh Z_n(\wh\phi, \omega^-_{-1}).
	\end{equation*}
	By Lemma \ref{limsupIsLim}, we are done.
\end{proof}

\begin{remark}\label{forNiceSpaces}
From Definition \ref{defOfRuelle}, a natural question arises: On what ``nice" space does $\L_{\wh\phi}$ act? While the standard Ruelle operator of a H\"older continuous potential $L_\psi$ contracts a cone in $\mathrm{H\ddot{o}l}(\Sigma^-)$, $\L_{\wh\phi}$ does not necessarily contract a cone in $\mathrm{H\ddot{o}l}(\Sigma^-\times X)$.
\end{remark}

\section{Conformal Measures}\label{ConfExistSect}

Once defining the semi-Ruelle operator, we are interested in construction self-measures and self-functions which correspond to its spectral radius. To construct self-functions, one has to first define a suitable function space with sufficiently ``nice" properties (see \textsection \ref{harmFuncsExistSect}). A preliminary step in accomplishing that, is constructing first the self-measures of the semi-Ruelle operator, as we show here in \textsection \ref{ConfExistSect}.

\begin{definition}[Dual operator]
	We define $\L_{\wh\phi}^*:\mathbb{P}(\Sigma^-\times X)\to \mathbb{P}(\Sigma^-\times X) $ by
	$$(\L_{\wh\phi}^*\wh p)(\wh g):=\wh p(\L_{\wh\phi}\wh g),$$
	for all $\wh g\in C(\Sigma^-\times X)$.
\end{definition}

\begin{remark}\label{rmkSpecRad}
	Note, by Lemma \ref{specRad}, the spectral radius of $\L_{\wh\phi}$ on $(C(\Sigma^-\times X),\|\cdot\|_\infty)$ is $e^{-\wh  P(\wh \phi)}$. Indeed, by the Kakutani-Markov-Riesz representation theorem, $\mathbb{P}(\Sigma^-\times X)$ is the dual space of $C(\Sigma^-\times X)$.  
\end{remark}

\begin{lemma}\label{VPOutside}
Assume that $(\Sigma^-\times X,\wh F_R)$ admit $\wh\phi$-holonomies, then the semi-Ruelle operator diverges at its radius of convergence:
$$\|\sum_{k\geq0} \lambda^{-k}\L_{\wh\phi}^k1\|_\infty
\xrightarrow[\lambda\downarrow e^{\wh P(\wh\phi)}]{}\infty.$$
\end{lemma}
\begin{proof}
Assume for contradiction that $$\wh S_n:=\sum_{k=0}^n \lambda^{-\wh P(\wh\phi)k}\L_{\wh\phi}^k1$$ is an equi-bounded sequence of functions, say by a constant $M\geq0$. Since $\L_{\wh\phi}$ is a positive operator, then the sequence of functions is increasing. Therefore, the sequence converges point-wise to a function $\wh h$. Then we have for all $n\geq0$,
$$1\leq \wh S_n\leq \wh h\leq M.$$ 
By monotonicity and the positivity of $\wh\L_{\wh\phi}$, we have similarly
 $$\wh S_{n+1}-1=e^{-\wh P(\wh\phi)}\L_{\wh\phi}\wh S_n\xrightarrow[n\to\infty] {\text{point-wise}} e^{-\wh P(\wh\phi)}\L_{\wh\phi}\wh h.$$
Then it follows that 
$$ e^{-\wh P(\wh\phi)}\L_{\wh\phi}\wh h =\wh h -1\leq \wh h -\frac{\wh h}{M}=\wh h\cdot  (1-\frac{1}{M}),$$
and consequently, for all $N\in \mathbb{N}$,
$$ e^{-\wh P(\wh\phi)N}\L_{\wh\phi}^N 1\leq e^{-\wh P(\wh\phi)N}\L_{\wh\phi}^N\wh h \leq \wh h \cdot (1-\frac{1}{M})^N\leq M\cdot (1-\frac{1}{M})^N.$$
Therefore, $ \|e^{-\wh P(\wh\phi)N}\L_{\wh\phi}^N 1 \|_\infty\leq M\cdot (1-\frac{1}{M})^N $ for all $N\in \mathbb{N}$, which is a contradiction to the fact that $e^{-\wh P(\wh\phi)}$ is the spectral radius of $\L_{\wh\phi}$ (recall Lemma \ref{specRad}). Therefore, 
\begin{equation}\label{SnEqNow}
	\|\wh S_n\|_\infty\to\infty.
\end{equation}

Finally, if $\sup_{\lambda>e^{\wh P(\wh\phi)}}\|\sum_{k\geq0} \lambda^{-k}\L_{\wh\phi}^k1\|_\infty<\infty$, and so
\begin{align*}
	\sup_N \|\wh S_N\|_\infty=&\sup_N\|\sum_{k=0}^N e^{-k\wh P(\wh\phi)}\L_{\wh\phi}^k1\|_\infty=\sup_N\sup_{\lambda>e^{\wh P(\wh\phi)}}\|\sum_{k=0}^N \lambda^{-k}\L_{\wh\phi}^k1\|_\infty\\
	\leq &\sup_{\lambda>e^{\wh P(\wh\phi)}}\|\sum_{k\geq0} \lambda^{-k}\L_{\wh\phi}^k1\|_\infty <\infty,
\end{align*}
which is a contradiction to \eqref{SnEqNow}.
\end{proof}

\begin{remark}\label{naiveProof} Our goal next is to construct $\wh p\in \mathbb{P}(\Sigma^-\times X)$ s.t. $e^{-\wh P(\wh\phi)}\L_{\wh\phi}^*\wh p=\wh p$, and s.t. $\wh p(1)=1$. Naively, we may try the following sequence of measures: 
	$$\wh p_n:=\frac{\sum_{k=0}^{n-1}e^{-k\wh P(\wh\phi)}(\L_{\wh\phi}^*)^k}{\|\sum_{k=0}^{n-1}e^{-k\wh P(\wh\phi)}\L_{\wh\phi}^k1\|_\infty}\delta_{(\omega^{-,n},t_n)},$$
	where $(\omega^{-,n},t_n) \in \mathrm{argmax}\Big\{\sum_{k=0}^{n-1}e^{-k\wh P(\wh\phi) }\L_{\wh\phi}^k1\Big\}$. This is indeed a normalized sequence of measures, where the unit ball of $\mathbb{P}(\Sigma^-\times X)$ is compact in the weak-* topology, and one could hope to choose a limit point which is $e^{-\wh P(\wh\phi)}\L_{\wh\phi}^*$-invariant. However, there is a subtlety here:
	$$\| e^{-\wh P(\wh\phi)}\L_{\wh\phi}^*\wh p_n-\wh p_n\|_\mathrm{W}\leq \frac{1+ e^{-n\wh P(\wh\phi)}\|\L_{\wh\phi}1\|_\infty}{\|\sum_{k=0}^{n-1}e^{-k\wh P(\wh\phi)}\L_{\wh\phi}^k1\|_\infty},$$
	
where $\|\wh \nu\|_\mathrm{W}:=\sup\{\wh \nu(\wh g):\|\wh g\|_\mathrm{Lip}\leq 1\}$ is the Wasserstein norm (on the space of signed measures whose total mass is $0$).

	Even knowing that $\|\sum_{k=0}^{n-1}e^{-k\wh P(\wh\phi)}\L_{\wh\phi}^k1\|_\infty \to\infty$, it is not sufficient in order to conclude that $\frac{e^{-n\wh P(\wh\phi)}\|\L_{\wh\phi}1\|_\infty}{\|\sum_{k=0}^{n-1}e^{-k\wh P(\wh\phi)}\L_{\wh\phi}^k1\|_\infty}\to0$. We bypass this hurdle in Theorem \ref{ConfExist} below.
\end{remark}

\begin{definition}
	Given $\wh p\in \mathbb{P}(\Sigma^-\times X)$, we say that $\wh p$ is {\em $\wh\phi$-conformal} if $$e^{-\wh P(\wh\phi)}\L_{\wh\phi}^*\wh p=\wh p.$$
\end{definition}

\begin{remark}
	In more general settings, such as when allowing $\Sigma$ to be non-compact, one can also study $\wh\phi$-conformal measures which are allowed to be infinite, but still Radon measures. In the setting of topological Markov shifts (without a fiber), see for example \cite{SarigNR, SarigTDF}.
\end{remark}

\begin{theorem}[Existence of a conformal measure]\label{ConfExist}
	Assume that $(\Sigma\times X,\wh F)$ admits $\wh\phi$-holonomies, then there exists $\wh p\in \mathbb{P}(\Sigma^-\times X)$ s.t. 
	$$ e^{-\wh P(\wh\phi)}\L_{\wh\phi}^*\wh p=\wh p.$$
\end{theorem}
\begin{proof}
	First, by Lemma \ref{specRad}, the spectral radius of $\L_{\wh\phi}^*$ is $e^{-\wh P(\wh\phi)}$ (see Remark \ref{rmkSpecRad}). In addition, $\L_{\wh\phi}$ is a positive operator%, and so also is $\L_{\wh\phi}^*$
	. Note, as $\|\wh\phi\|_\infty<\infty$, $\wh P(\wh\phi)>-\infty$, and so $e^{-\wh P(\wh\phi)}>0$ (recall Remark \ref{PBounded}).

Then, given $\lambda>e^{\wh P(\wh\phi)}$ the resolvent
\begin{equation*}\label{formOfResolvent0}
R(\lambda,\L_{\wh\phi}):=\Big(\lambda\cdot \mathrm{Id}-\L_{\wh\phi}\Big)^{-1},
\end{equation*}
	is well-defined and is given by
\begin{equation}\label{formOfResolvent}
R(\lambda,\L_{\wh\phi})= \frac{1}{\lambda}\sum_{k\geq0}(\lambda^{-1}\L_{\wh\phi})^k.
\end{equation}

In addition, %for every fixed $(\omega^-,t)\in \Sigma^-\times X$, $$\sum_{k\geq0}\lambda^{-k}\L_{\wh\phi}^k1(\omega^-,t)$$ is analytic in $\lambda\notin \sigma(\L_{\wh\phi})$ where $\sigma(\L_{\wh\phi})$ denotes the spectrum of $\L_{\wh\phi} $. Then by the Vivanti-Pringsheim theorem (in the context of positive operators, \cite{VP}),
by Lemma \ref{VPOutside} and by \eqref{formOfResolvent},
\begin{align}\label{eqVPlambda}
\lambda\cdot\|R(\lambda,\L_{\wh\phi})1\|_\infty%=&\|R(\lambda,\L_{\wh\phi})\|_\mathrm{Op(C(\Sigma^-\times X),\|\cdot\|_\infty)}\nonumber\\
\geq  %\frac{1}{\lambda}
\|\sum_{k\geq0} \lambda^{-k}(\L_{\wh\phi}^k1)\|_\infty
\xrightarrow[\lambda\downarrow e^{\wh P(\wh\phi)}]{}\infty.
\end{align}
This allows us to bypass the obstacle of Remark \ref{naiveProof} by defining the following convergent infinite sum without a tail term,
\begin{equation}\label{thePLambdaDef}
    \wh p_\lambda:= \frac{\sum_{k\geq0}\lambda^{-k}(\L_{\wh\phi}^*)^k}{\|\sum_{k\geq0}\lambda^{-k}\L_{\wh\phi}^k1\|_\infty}\delta_{(\omega^{-,\lambda},t_\lambda)}\in \mathbb{
P}(\Sigma^-\times X),
\end{equation}
where  $(\omega^{-,\lambda},t_\lambda) \in \mathrm{argmax}\Big\{\sum_{k\geq0}\lambda^{-k}\L_{\wh\phi}^k1\Big\}$. It follows that,
 	\begin{equation}\label{eqVPlambda2}
\| \wh p_\lambda-\lambda\L_{\wh\phi}^*\wh p_\lambda\|_\mathrm{W}\leq \frac{1}{\|\sum_{k\geq0}\lambda^{-k}\L_{\wh\phi}^k1\|_\infty},
\end{equation}

where $\|\wh \nu\|_\mathrm{W}:=\sup\{\wh \nu(\wh g):\|\wh g\|_\mathrm{Lip}\leq 1\}$ is the Wasserstein norm (on the space of signed measures whose total mass is $0$). So by \eqref{eqVPlambda} and \eqref{eqVPlambda2},
\begin{equation}\label{eqVPlambda3}
	\| \wh p_\lambda-\lambda\L_{\wh\phi}^*\wh p_\lambda\|_\mathrm{W}\leq \frac{1}{\lambda\| R(\lambda,\L_{\wh\phi}) 1\|_\infty}\xrightarrow[\lambda\downarrow e^{\wh P(\wh \phi)}]{}0.
\end{equation}
 
Then, by choosing a sequence $\lambda_j\downarrow e^{\wh P(\wh \phi)} $ s.t. 
\begin{equation}\label{thePSeq}
 \wh p_{\lambda_j}\to \wh p\in \mathbb{P}(\Sigma^-\times X),   
\end{equation}
 we get for all $\wh g \in C(\Sigma^-\times X)$ with $\|\wh g\|_\mathrm{Lip}\leq 1$, by \eqref{eqVPlambda3},
\begin{align}\label{eqVPlambda4}
\Big( \wh p-e^{-\wh P(\wh\phi)}\L_{\wh\phi}^*\wh p\Big)(\wh g)=&\lim_j\Big( \lambda_j^{-1}\L_{\wh\phi}^*\wh p_{\lambda_j}-\wh p _{\lambda_j}\Big)(\wh g)\nonumber\\
\leq& \frac{1}{\lambda_j\| R(\lambda_j,\L_{\wh\phi}) 1\|_\infty}\xrightarrow[j\to\infty]{}0.
\end{align}
By applying \eqref{eqVPlambda4} to the constant functions $1$ and $-1$, we get that $e^{-\wh P(\wh\phi)}\L_{\wh\phi}^*\wh p$ is a probability measure. Thus, as the Wasserstein norm is a norm for signed measures whose total mass is $0$, $\wh p -e^{-\wh P(\wh\phi)}\L_{\wh\phi}^*\wh p $ is the zero measure, and so $e^{-\wh P(\wh\phi)}\L_{\wh\phi}^*\wh p=\wh p $, and we are done.
\end{proof}

\begin{prop}\label{keyPosProp}
Assume that $(\Sigma,T)$ is topologically transitive, then there exists $N\geq0$ s.t. for all $a$, for all $k\geq 0$,
 $$\wh p(\mathbb{1}_{[a]\times X}\circ \wh F_R^k)\geq e^{-N\|\wh \phi\|_\infty-N\wh P(\wh\phi)}.$$
\end{prop}
\begin{proof}
 Let $N$ be such that for any symbol $b$ there exists an admissible word $\ul u^b$ of length $N$ s.t. $u_{N-1}^b=a$ and $u_{0}^b=b$. Then, for all $(\omega^-,t)$ for all $k\geq 0$ for all $n\geq 0$,
\begin{align}\label{posLogEq-1}
    \L_{\wh \phi}^{n+N}(\mathbb{1}_{[a]\times X}\L_{\wh\phi}^k1)(\omega^-,t)=& \sum_{\overset{\wh F^{n+N}_R(\sigma^-,s)}{=(\omega^-,t)}}e^{\wh \phi^{(-n-N)}(\sigma^-,t)}(\mathbb{1}_{[a]\times X}\L_{\wh\phi}^k1)(\sigma^-,s)\\
    \geq&e^{-N\|\wh\phi\|_\infty}\sum_{\wh F^{n}_R(\sigma^-,s)=(\omega^-,t)}e^{\wh \phi^{(-n)}(\sigma^-,t)}\L_{\wh\phi}^k1(\sigma^-,s)\nonumber\\
    =& e^{-N\|\wh\phi\|_\infty}\L_{\wh \phi}^{n}(\L_{\wh\phi}^k1)(\omega^-,t)=e^{-N\|\wh\phi\|_\infty}\L_{\wh \phi}^{n+k}1(\omega^-,t).\nonumber
\end{align}
Next, we note that by $\wh\phi$-conformality,
$$\wh p(\mathbb{1}_{[a]\times X}\circ \wh F_R^k)=\int \mathbb{1}_{[a]\times X}\circ \wh F_R^kd\wh p=\int \mathbb{1}_{[a]\times X}e^{-k\wh P(\wh\phi)}\L_{\wh\phi}^k1d\wh p.$$
Then, by \eqref{thePSeq} and \eqref{thePLambdaDef}, together with \eqref{posLogEq-1},
\begin{align*}
  &\int \mathbb{1}_{[a]\times X}e^{-k\wh P(\wh\phi)}\L_{\wh\phi}^k1d\wh p\\
  =&\lim_j\frac{\sum_{n\geq0}\lambda_j^{-n}\L_{\wh \phi}^n( \mathbb{1}_{[a]\times X}e^{-k\wh P(\wh\phi)}\L_{\wh\phi}^k1)(\omega^-_{\lambda_j},t_{\lambda_j})}{\sum_{n\geq0}\lambda_j^{-n}\L_{\wh \phi}^n1(\omega^-_{\lambda_j},t_{\lambda_j})}\\
  =&e^{-N\wh P(\wh\phi)}\lim_j\frac{\sum_{n\geq0}\lambda_j^{-n}\L_{\wh \phi}^{n+N}( \mathbb{1}_{[a]\times X}e^{-k\wh P(\wh\phi)}\L_{\wh\phi}^k1)(\omega^-_{\lambda_j},t_{\lambda_j})}{\sum_{n\geq0}\lambda_j^{-n}\L_{\wh \phi}^n1(\omega^-_{\lambda_j},t_{\lambda_j})}\\
  \geq&e^{-N\wh P(\wh\phi)}\lim_j\frac{\sum_{n\geq0}\lambda_j^{-n-k}\L_{\wh \phi}^n( \mathbb{1}_{[a]\times X}\L_{\wh\phi}^k1)(\omega^-_{\lambda_j},t_{\lambda_j})}{\sum_{n\geq0}\lambda_j^{-n}\L_{\wh \phi}^n1(\omega^-_{\lambda_j},t_{\lambda_j})} \\
  \geq&e^{-N\wh P(\wh\phi)} e^{-N\|\wh \phi\|_\infty}\lim_j\frac{\sum_{n\geq0}\lambda_j^{-n-k}\L_{\wh \phi}^{n+k}1(\omega^-_{\lambda_j},t_{\lambda_j})}{\sum_{n\geq0}\lambda_j^{-n}\L_{\wh \phi}^n1(\omega^-_{\lambda_j},t_{\lambda_j})}=e^{-N\|\wh \phi\|_\infty-N\wh P(\wh\phi)}.
\end{align*}
\end{proof}

\begin{lemma}[Ergodic conformal measure]\label{ConfOfErg}
	Let $\wh p$ be a conformal measure, then almost every ergodic component of $\wh p$ is $\wh\phi$-conformal.
\end{lemma}
\begin{proof}
Let $\mathcal{I} = \Big\{ B \in \mathcal{B}( \Sigma^-\times X ): \wh F_R^{-1}B = B \Big\}$ be the $\sigma$-algebra of $\wh F_R$-invariant sets. Since the conformal measure $\wh p$ is generally not $\wh F_R$-invariant but is quasi-invariant, we decompose $\wh p$ w.r.t. $\mathcal{I}$. We the existence of conditional probabilities theorem, we can write
$$\wh p = \int  \mathbb{E}(\cdot |\mathcal{I})(\omega^-,t) d\wh p(\omega^-,t),$$
where $\mathbb{E}(\cdot |\mathcal{I})(\omega^-,t) :L^\infty(\wh p)\to \mathbb{R}$ is a well-defined bounded linear functional $\wh p$-a.e., hence it defines a measure $\wh p_{(\omega^-,t)}$ on $\Sigma^-\times X$. This defines the ergodic decomposition of $\wh p$, as for every invariant set $E\in\mathcal{I}$, $\mathbb{1}_E$ is $\mathcal{I}$-measurable, and so for $\wh p$-a.e. $(\omega^-,t)$,
$$\wh p_{(\omega^-,t)}(\mathbb{1}_{E} ):=\mathbb{E}(\mathbb{1}_{E} |\mathcal{I})(\omega^-,t)= \mathbb{1}_{E} (\omega^-,t)\in \{0,1\}.$$

Recall that $\wh p$ being a conformal measure means that for any $\wh g %,\wh h
\in \mathrm{Lip}(\Sigma^-\times X)$ we have,
$$\int %\wh h
 e^{-\wh P(\wh\phi)}\L_{\wh\phi}\wh g d\wh p= \int %\wh h\circ \wh F_R \cdot 
 \wh gd\wh p.$$
And so,
\begin{align*}
\int \wh p_{(\omega^-,t)}(%\wh h 
e^{-\wh p(\wh\phi)}\L_{\wh\phi}\wh g )d\wh p =&	\int %\wh h 
e^{-\wh p(\wh\phi)}\L_{\wh\phi}\wh g d\wh p=  \int %\wh h\circ \wh F_R \cdot 
\wh gd\wh p\\
=&\int \wh p_{(\omega^-,t)}( %\wh h\circ \wh F_R \cdot 
\wh g)d\wh p.
\end{align*}
Indeed,
$$\int \wh g d(e^{-\wh P(\wh \phi)}\L_{\wh\phi}^*\wh p _{(\omega^-,t)})=\int %\wh h e^{-\wh P(\wh\phi)}
e^{-\wh P(\wh \phi)}\L_{\wh\phi}\wh g d\wh p _{(\omega^-,t)}.$$
This holds for every $\wh g%,\wh h
\in \mathrm{Lip}(\Sigma^-\times X)$, and so by the uniqueness of the conditional measures, we get that for $\wh p$-a.e. $(\omega^-,t)$,
 $$e^{-\wh P(\wh \phi)}\L_{\wh\phi}^*\wh p _{(\omega^-,t)} = \wh p _{(\omega^-,t)}.$$
\end{proof}

\section{Two-Sided Extension, and Semi-Gibbs Estimates}

\begin{lemma}\label{forHar}
	Assume that $(\Sigma\times X,\wh F)$ admits $\wh\phi$-holonomies. Then there exists $C>0$ s.t. for all $t\in X$ for all $d(\omega^-,\wt\omega^-)<1$ for all $k\geq0$,
$$\log \frac{(\L_{\wh \phi}^k1)(\omega^-,t)}{(\L_{\wh \phi}^k1)(\wt\omega^-,t)}\leq C.$$
\end{lemma}
\begin{proof}
	Let $C$ be constant given by the $\wh\phi$-holonomies. Then,  for all $t\in X$ for all $d(\omega^-,\wt\omega^-)<1$ for all $k\geq0$,
\begin{align*}
	\frac{(\L_{\wh \phi}^k1)(\omega^-,t)}{(\L_{\wh \phi}^k1)(\wt\omega^-,t)}=&\frac{\sum_{\overset{|\ul w|=k,}{w_0=a}}e^{\wh\phi^{(k)}(\omega^-\ul w,t)}}{\sum_{\overset{|\ul w|=k,}{w_0=a}}e^{\wh\phi^{(k)}(\wt\omega^-\ul w,t)}}= \frac{\sum_{\overset{|\ul w|=k,}{w_0=a}}e^{\wh\phi^{(k)}(\wt\omega^-\ul w,t)}e^{\wh\phi^{(k)}(\omega^-\ul w,t)-\wh\phi^{(k)}(\wt\omega^-\ul w,t)}}{\sum_{\overset{|\ul w|=k,}{w_0=a}}e^{\wh\phi^{(k)}(\wt\omega^-\ul w,t)}}\\
	= &\frac{\sum_{\overset{|\ul w|=k,}{w_0=a}}e^{\wh\phi^{(k)}(\wt\omega^-\ul w,t)}e^{\pm C}}{\sum_{\overset{|\ul w|=k,}{w_0=a}}e^{\wh\phi^{(k)}(\wt\omega^-\ul w,t)}}=e^{\pm C}.
\end{align*}	
\end{proof}

\begin{lemma}\label{divSumPess} Let $\wh\phi$ be a $\wh\phi$-conformal measure, then
	$$\|\sum_{k\geq0 }\lambda^{-k}\L_{\wh\phi}^k1\|_{L^\infty(\wh p)}\xrightarrow[\lambda\downarrow e^{\wh P(\wh\phi)}]{}\infty.$$
\end{lemma}
\begin{proof}
	If $\sum_{k\geq0 }\lambda^{-k}\L_{\wh\phi}^k1$ were bounded uniformly by $M$, then 
	$$M\geq \int \sum_{k\geq0 }\lambda^{-k}\L_{\wh\phi}^k1d\wh p = \sum_{k\geq0 }(\lambda^{-1}e^{\wh P(\phi)})^{k} \xrightarrow[\lambda\downarrow e^{\wh P(\wh\phi)}]{}\infty ,$$
	a contradiction.
\end{proof}

\begin{theorem}[Two-sided extension]\label{TSExtension}
		Assume that $(\Sigma,T)$ is topologically mixing and that $(\Sigma\times X,\wh F)$ admits $\wh\phi$-holonomies, and let $\wh p$ be a $\wh\phi$-conformal measure. Then, $\wh p$ extends to a probability measures $\wh p^\pm$ on $\Sigma\times X$.
\end{theorem}
\begin{proof}
	Let $\wh p=\int \wh p_{\omega^-} dp(\omega^-)$ be the disintegration of $\wh p$ w.r.t. the partition $\{\{\omega^-\}\times X\}_{\omega^-\in \Sigma^-}$, where $p$ is the projection of $\wh p$ to $\Sigma^-$. Write $\wt \phi:=\wh \phi-\wh P(\wh\phi)$ and note that $\L_{\wt \phi}^*p=p$ and $\wh P(\wt \phi)=0$. %We continue to construct $\wh\rho$ which is $\wt\phi$-harmonic. 
	Recall that 
$$V^u_{\leq -1}(\omega^-):=\{\omega\in \Sigma: \omega^-_i=\omega_i \ \forall i\leq-1\}.$$

\textbf{Step 1:} Finding $h\in L^\infty (p)$ s.t. $h$ is essentially bounded away from zero and infinity, and such that $h\cdot p$ is $T_R$-in Given $n\geq0$, we define
$$h_n(\omega^-):=\int \frac{1}{n}\sum_{k\leq n}\L_{\wt\phi}^k1 d\wh p_{\omega^-}.$$
We wish to prove that $\log h_n\in L^\infty(p)$ for all $n\geq0$. Indeed, for all $a$, when restricted to $[a]$, by Lemma \ref{forHar}, the variation of $h_n$ is bounded by $e^C$. We wish to prove that there exists $C_0>0$ s.t. for all $a$, there exists $\omega^-\in[a]$ s.t. $h_n(\omega^-)=C_0^{\pm1}$. Indeed, by the $\phi$-conformality of $\wh p$,
\begin{align}\label{forLogBound1}
	\frac{1}{ p([a])}\int_{[a]} h_n dp=&\frac{1}{\wh p([a]\times X)}\int_{[a]\times X}  \frac{1}{n}\sum_{k\leq n}\L_{\wt\phi}^k1 d\wh p\\
	=& \frac{1}{\wh p([a]\times X)}\int \frac{1}{n}\sum_{k\leq n}\mathbb{1}_{[a]\times X}\circ \wh F_R^k d\wh p.\nonumber 
\end{align}
By Proposition \ref{keyPosProp}, the r.h.s. of \eqref{forLogBound1} is bounded from below and from above by $e^{\pm C_0}$, respectively, where 
$$C_0:=N\cdot (\|\wh\phi\|_\infty+\wh P(\wh \phi)),$$ for some independent $N\in \mathbb{N}$ give by the topological mixing of $(\Sigma,T)$. Therefore,
$$\mathrm{ess-sup}_{p|_{[a]}} h_n\geq e^{-C_0}\text{ and }\mathrm{ess-inf}_{p|_{[a]}} h_n\leq e^{C_0}.$$
Thus, $\mathrm{ess-sup}_p|\log h_n|\leq C_0+C$. Since $L^\infty(p)=(L^1(p))^*$, its unit ball is compact in the weak-* topology. Therefore there exists a converging sub-sequence $h_{n_j}\xrightarrow[\text{weak-*}]{}h\in $. It is easy to check that $h$ is essentially bounded from above and from below by $e^{C+C_0}$ and  $e^{-(C+C_0)}$ respectively. 

We now verify that $h\cdot p$ is $T_R$-invariant. Let $g\in C(\Sigma^-)$ with $\|g\|_\infty\leq1$. We can also think of $g$ as a function on $\Sigma^-\times X$, in that case $g\circ T_R= g\circ \wh F_R $. Then,
\begin{align*}
	\int g h dp=&\lim_j \int g \int \frac{1}{n_j}\sum_{k\leq n_j}\L_{\wt\phi}^k1 d\wh p_{\omega^-} dp= \lim_j \frac{1}{n_j}\sum_{k\leq n_j}\int \int g \L_{\wt\phi}^k1 d\wh p_{\omega^-} dp\\
	=& \lim_j \frac{1}{n_j}\sum_{k\leq n_j}\int  g \L_{\wt\phi}^k1 d\wh p= \lim_j \frac{1}{n_j}\sum_{k\leq n_j}\int  g\circ \wh F_R^k d\wh p\\
	=& \lim_j \frac{1}{n_j}\sum_{k\leq n_j}\int  g\circ \wh F_R^{k+1} d\wh p= \lim_j \frac{1}{n_j}\sum_{k\leq n_j}\int \int g \circ \wh F_R\L_{\wt\phi}^k1 d\wh p_{\omega^-} dp\\
	=& \lim_j \int g\circ T_R \int \frac{1}{n_j}\sum_{k\leq n_j}\L_{\wt\phi}^k1 d\wh p_{\omega^-} dp= \int g\circ T_R h dp.
\end{align*}

\textbf{Step 2:} We extend $\wh p$ to $\Sigma\times X$% and disintegrate it
.

\textbf{Proof:} Extend $\nu=h\cdot p$ to $\nu^\pm$ which is a $T$-invariant probability on $\Sigma$. Then define,
$$\wh p^\pm =\int\frac{1}{h\circ \gamma(\omega)} \wh p_{\gamma(\omega)}d\nu^\pm(\omega),$$
where $\gamma:\Sigma\to \Sigma^-$ is the projection.% Finally, disintegrate $\wh p^\pm$ by the measurable partition $\{V^u_{\leq -1}(\omega^-)\times\{t\}\}_{(\omega^-,t)\in \Sigma^-\times X}$:
%$$\wh p^\pm=\int \wh p_{(\omega^-,t)}d\wh p.$$
\end{proof}

\begin{prop}[Semi-Gibbs estimates]\label{newSemiGibbsProp}
	In the setting of Theorem \ref{TSExtension}, let $\wh p^\pm$ be the extension of $\wh p$ to $\Sigma\times X$, and disintegrate $\wh p^\pm$ by the measurable partition $\{V^u_{\leq -1}(\omega^-)\times\{t\}\}_{(\omega^-,t)\in \Sigma^-\times X}$, $\wh p^\pm=\int \wh p_{(\omega^-,t)}d\wh p$.
Then, for $\wh p$-a.e. $(\omega^-,t)\in\Sigma^-\times X$, %constructing a measure $\wh p_{(\omega^-,t)}$ on $V^u_{\leq-1}(\omega^-)\times\{t\}$ s.t. 
for all $\ul w=(w_0,\ldots,w_{k-1})$ with $w_0=\omega^-_{-1}$, %for all $\ell\geq k$,
\begin{equation*}\label{theTheTheProperty}
	\wh p_{(\omega^-,t)}([\ul w]\times X)=%C_{\wh p}^{\pm 1}\frac{e^{\wt \phi^{(k)} (\theta_{\ul w},t)}%(\L_{\wt\phi}^{\ell-k}1)(\omega^-,t)
%}{(\L_{\wt\phi}^%\ell
%k1)(\omega^-,t)},
 \frac{e^{\wt \phi^{(-k)} (\omega^-\cdot \ul w,(F_{\omega^-\cdot\ul w}^{-k})^{-1}(t))}}{(\L_{\wt\phi}^k1) (\omega^-,t)}.
\end{equation*} 
\end{prop}
\begin{proof} Given a word $\ul w=(w_0,\ldots w_{k-1})$ with $w_0=\omega^-_{-1}$, for $\wh p$-a.e. $(\omega^-,t)$, 
\begin{align*}
\wh p_{(\omega^-,t)}([^\cdot\ul w] \times X)=&\lim_n\frac{\wh p^\pm(([^\cdot\ul w] \times X)\cap ([\omega^-]_{-n}\times B_X(t,\delta_n)))}{\wh p^\pm([\omega^-]_{-n}\times B_X(t,\delta_n))}\\
=&\lim_n\frac{\int \mathbb{1}_{[\ul w ^\cdot]\times X}\circ \wh F^k\cdot \mathbb{1}_{[\omega^-]_{-n}\times B_X(t,\delta_n)}d\wh p^\pm}{\int \mathbb{1}_{[\omega^-]_{-n}\times B_X(t,\delta_n)}d\wh p^\pm}\\
=& \lim_n\frac{\int \wh J_k \cdot\mathbb{1}_{[\ul w ^\cdot] \times X}\cdot \mathbb{1}_{[\omega^-]_{-n}\times B_X(t,\delta_n)} \circ \wh F^{-k} d\wh p^\pm}{\int \wh J_k \cdot \mathbb{1}_{[\omega^-]_{-n}\times B_X(t,\delta_n)} \circ \wh F^{-k} d\wh p^\pm}\\
=& \lim_n\frac{\int \wh J_k \cdot\mathbb{1}_{[\ul w ^\cdot] \times X}\cdot \mathbb{1}_{[\omega^-]_{-n}\times B_X(t,\delta_n)} \circ \wh F^{k}_R d\wh p}{\int \wh J_k \cdot \mathbb{1}_{[\omega^-]_{-n}\times B_X(t,\delta_n)} \circ \wh F^{k}_R d\wh p}\\
=& \lim_n\frac{\int \mathbb{1}_{[\ul w ^\cdot] \times X}\cdot \mathbb{1}_{[\omega^-]_{-n}\times B_X(t,\delta_n)} \circ \wh F^{k}_R d\wh p}{\int  \mathbb{1}_{[\omega^-]_{-n}\times B_X(t,\delta_n)} \circ \wh F^{k}_R d\wh p}\\
=& \lim_n\frac{\int \L_{\wt\phi}^k\mathbb{1}_{[\ul w ^\cdot] \times X}\cdot \mathbb{1}_{[\omega^-]_{-n}\times B_X(t,\delta_n)}  d\wh p}{\int \L_{\wt\phi}^k1\cdot \mathbb{1}_{[\omega^-]_{-n}\times B_X(t,\delta_n)} d\wh p}\\
=&\frac{\L_{\wt\phi}^k\mathbb{1}_{[\ul w ^\cdot] \times X}}{\L_{\wt\phi}^k1} (\omega^-,t)= \frac{e^{\wt \phi^{(-k)} (\omega^-\cdot \ul w,(F_{\omega^-\cdot\ul w}^{-k})^{-1}(t))}}{(\L_{\wt\phi}^k1) (\omega^-,t)},
\end{align*}
where $[\omega^-]_{-n}=[\omega^-_{-n},\ldots,\omega^-_{-1}]$ and $\delta_n\downarrow 0$. %The statement of this step follows with the constant given by the $\wh\phi$-holonomies.
\end{proof}

\begin{remark}
Note, in the classical case of a Ruelle operator on a compact TMS, if $L_\phi1=1$, $P(\phi)=0$, then the sequence $\{\log L_\phi^k1\}_{k\geq0}$ is equi-bounded.
\end{remark}

\section{Harmonic Functions}\label{harmFuncsExistSect}

In Theorem \ref{ConfExist} we prove the existence of a $\wh\phi$-conformal measure- that is an eigen-measure of the semi-Ruelle operator whose eigen-values is the reciprocal of the spectral radius. We now discuss the more subtle question existence of an eigen-function as well (see Definition \ref{harmFuncDef} below). In general, without knowing that the operator is compact, one cannot expect the point spectrum of the dual to coincide with point spectrum of the operator. %We begin by introducing the notion of {\em mild holonomies} which we additionally require for the existence of a harmonic function in \textsection \ref{harmExistSec}.

We begin by constructing an $\wh F_R$-invariant measure, which is equivalent to $\wh p$. Note, if we consider a limit point $\wh \nu$ of the sequence 
\begin{equation}\label{weakStarLimEq}
	(\frac{1}{n}\sum_{k=0}^{n-1}e^{-k\wh P(\wh\phi)}\L_{\wh\phi}^k1)\cdot \wh p
\end{equation}
where $\wh p$ is a $\wh\phi$-conformal measure, then $\wh\nu$ is an $\wh F_R$-invariant measure. Indeed, for all $\wh g\in C(\Sigma^-\times X)$,
\begin{align}\label{invP}
	\wh \nu(\wh g\circ \wh F_R)=&\lim_j \frac{1}{n_j}\sum_{k=0}^{n-1}\int \wh g\circ \wh F_R e^{-k\wh P(\wh\phi)}\L_{\wh\phi}^k1 d\wh p\\
	=& \lim_j \frac{1}{n_j}\sum_{k=0}^{n-1}\int\wh g e^{-(k-1)\wh P(\wh\phi)}\L_{\wh\phi}^{k-1}1 d\wh p =\wh \nu(\wh g).\nonumber
\end{align}
However, one cannot expect $\frac{1}{n}\sum_{k=0}^{n-1}e^{-k\wh P(\wh\phi)}\L_{\wh\phi}^k1 $ to converge in some strong function norm (e.g. $C^\alpha$ or $C^0$), see for example Remark \ref{noContHarm}. Therefore, in general, it could %naively
 happen that any weak-* limit point of \eqref{weakStarLimEq} is a singular measure to $\wh p$ (see e.g. \textsection \ref{AppCount} for such an example). %Nonetheless, by solving a cohomological equation,  we construct an invariant measure which is equivalent to $\wh p$, and whose Radon-Nikodym derivative is essentially bounded from below and from above. This Radon-Nikodym derivative allows us then to define a $\wh\phi$-harmonic function which is defined $\wh p$-a.e. (recall that Remark \ref{noContHarm} implies that we cannot expect much more).

\begin{definition}[Harmonic functions]\label{harmFuncDef}
	Let $\wh g:\Sigma^-\times X\to \mathbb{R}$ be a non-negative measurable function. We say that $\wh g$ is {\em $\wh\phi$-harmonic} if 
	$$e^{-\wh P(\wh \phi)}\L_{\wh\phi}\wh g=\wh g.$$
\end{definition}

\begin{remark} In more general settings, such as when allowing $\Sigma$ to be non-compact, one can also study $\wh\phi$-harmonic functions which are allowed to be unbounded, but are still bounded on compact sets which are fiber-saturated. In the setting of topological Markov shifts (without a fiber), see for example \cite{SarigNR, SarigTDF}.
\end{remark}

\begin{definition}[Space of measured leaves]\label{MeasuredLeaves}
    We define the following {\em space of measured leaves}:
$$\mathcal{S}:=\Big\{(\omega^-,t,\nu)\in \Sigma^-\times X\times \mathbb{P}(\Sigma\times X):\mathrm{Supp}(\nu)\subseteq V^u_{\leq-1}(\omega^-)\times\{t\}\Big\}.$$
\end{definition}

\begin{remark}\label{MeasuredLeavesRmk}
    Note that $\mathcal{S}$ is compact. This fact becomes important for us in the proof of Theorem \ref{harmExist}. 
\end{remark}

\begin{theorem}[Existence of a harmonic function]\label{harmExist}
	Assume that $(\Sigma,T)$ is topologically mixing and that $(\Sigma\times X,\wh F)$ admits $\wh\phi$-holonomies, and let $\wh p$ be a $\wh\phi$-conformal measure. Then there exists $\wh \rho>0$ s.t. $\log\wh\rho\in L^\infty(\wh p)$ and s.t. 
	$$ e^{-\wh P(\wh\phi)}\L_{\wh\phi}\wh \rho=\wh \rho,$$
if and only if $$\sup_k \|\log\L_{\wt\phi}^k1\|_{L^\infty (\wh p)}<\infty,$$
where $\wt \phi:=\wh \phi-\wh P(\wh \phi)$.
\end{theorem}
\begin{proof} \text{ }

\textbf{$\Rightarrow$:}  Assume that there exists $\wh \rho>0$ s.t. $\|\log\wh\rho\|_{L^\infty(\wh p)}\leq C$ and s.t. $e^{-\wh P(\wh\phi)}\L_{\wh\phi}\wh \rho=\wh \rho$. Let $k\geq 0$, and let $[\ul w]$ be a cylinder. Then, by \eqref{invP},
\begin{align*}
	\int \mathbb{1}_{[\ul w]} \L_{\wt \phi}^k1d\wh p=& \int \mathbb{1}_{[\ul w]}\circ \wh F_R^k d\wh p=e^{\pm C} \int \mathbb{1}_{[\ul w]}\circ \wh F_R^k d(\wh\rho\cdot\wh p)= e^{\pm C} \int \mathbb{1}_{[\ul w]} d(\wh\rho\cdot\wh p)\\
	=& e^{\pm 2C} \int \mathbb{1}_{[\ul w]} d\wh p.
\end{align*}
Therefore, $\sup_k\|\log\L_{\wt \phi}^k1 \|_{L^\infty(\wh p)}\leq 2 C$.

\medskip
\textbf{$\Leftarrow$:} Given $\epsilon>0$, set
$$\wh \rho_\epsilon:=\frac{\sum_{k\geq 0}e^{-\epsilon k}\L_{\wt \phi}^k1}{\|\sum_{k\geq 0}e^{-\epsilon k}\L_{\wt \phi}^k1 \|_{L^\infty(\wh p)}}.$$
Note, 
$$e^{-2C'}\leq\rho_\epsilon \leq e^{ 2C'},$$
where $C':=\sup_k\|\log\L_{\wt \phi}^k1 \|_{L^\infty(\wh p)} $. Therefore, since $L^\infty(\wh p)=(L^1(\wh p))^*$, we get that there exists a subsequence $\epsilon_n\to 0$ s.t. $\wh \rho_{\epsilon_n} \xrightarrow[n\to\infty]{\text{weak-*}}\wh\rho\in L^\infty(\wh p) $, with $\|\log \wh \rho\|_{L^\infty(\wh p)}\leq 2C'$.

\end{proof}

\begin{remark}
	Note, without the assumption of Theorem \ref{harmExist}, it is not even guaranteed to have an harmonic function in $L^1(\wh p)$, see \textsection \ref{AppCount}
\end{remark}

\section{Variational Principle, Semi-Gibbs Estimates, %Equi-Conformality, 
and Invariant Measures}\label{varPrinceSect}

Before describing the variational principle corresponding to the POE, we need to introduce an additional notion of {\em entropy out of equilibrium}.

\subsection{Entropy Out of Equilibrium}

\begin{definition}
   We define the following partitions:
   \begin{enumerate}
    \item $\mathcal{C}:=\{[a]: [a]\subseteq \Sigma\}$,
       \item $\widehat{\mathcal{C}}:=\{[a]\times X: [a]\subseteq \Sigma\}$,
       \item For $q\in \mathbb{N}$, $\mathcal{C}_q:=\bigvee_{i=0}^{q-1}T^{-i}\mathcal{C}$,
        \item For $q\in \mathbb{N}$, $\widehat{\mathcal{C}}_q:=\bigvee_{i=0}^{q-1}\widehat{F}^{-i}\widehat{\mathcal{C}}$.
   \end{enumerate}
\end{definition}

\begin{remark}
    Note, $\widehat{F}$ preserves the fiber structure, and so elements of $\widehat{C}_q$ are saturated by fibers.
\end{remark}

\begin{definition}[Metric entropy out of equilibrium {\cite[Definition~3.4]{TDFOE_I}}]
    Given an $\widehat{F}$-invariant Borel probability $\widehat{\eta}$ on $\Sigma^-\times X$, we define its {\em metric entropy out of equilibrium} by 
    \begin{equation}\label{EOEform}
        \widehat{h}_{\widehat{\eta}}(\widehat{F}):=\limsup_{n\to\infty} \int \frac{1}{n}\H_{\eta_{(\omega^-,t)}}\Big(\widehat{\mathcal{C}}_q\Big)d\tau(\omega^-,t),
    \end{equation}
    where $\widehat{\eta}=\int \eta_{(\omega^-,t)} d\tau(\omega^-, t)$ is the disintegration of $\widehat{\eta}$ into conditional measures by the measurable partition $\{V^u_{\leq -1}(\omega^-)\times\{t\}\}_{\omega^-\in \Sigma^-, t\in X}$.
\end{definition}

\begin{definition}[Metric pressure out of equilibrium {\cite[Definition~3.5]{TDFOE_I}}]
    Given an $\widehat{F}$-invariant Borel probability $\widehat{\eta}$ on $\Sigma\times X$, and a  potential $\widehat{\phi}:\Sigma\times X\to\mathbb{R}$, we define its {\em metric pressure out of equilibrium} by 
    $$\widehat{P}_{\widehat{\eta}}(\widehat{\phi}):=\widehat{h}_{\widehat{\eta}}(\widehat{F})+\int \widehat{\phi}d\widehat{\eta}.$$
\end{definition}

\begin{definition}[Local entropy out of equilibrium {\cite[Definition~3.6]{TDFOE_I}}]
   Given an $\widehat{F}$-invariant Borel probability $\widehat{\eta}$, the {\em local entropy out of equilibrium} of $\widehat{\eta}$ is defined for $\tau$-a.e. $(\omega,t)$ as,
   \begin{equation}\label{locEntForm}\widehat{h}_{\widehat{\eta}}((\omega,t)):=\limsup_{n\to\infty}\frac{-1}{n}\log\eta_{(\omega^-,t)}([\omega]_n),
   \end{equation}
   where $[\omega]_n:=[\omega_0,\ldots,\omega_{n-1}]$,  $\omega^-=(\omega_i)_{i\leq -1}$, and  $\widehat{\eta}=\int \eta_{(\omega^-,t)} d\tau(t)$ is the disintegration of $\widehat{\eta}$ into conditional measures by the measurable partition $\{V^u_{\leq -1}(\omega^-)\times\{t\}\}_{\omega^-\in \Sigma^-, t\in X}$.
\end{definition}

\begin{lemma}[{\cite[Lemma~3.7]{TDFOE_I}}]\label{strongUEnt}
    Assume that $F_\omega=F_{(\omega_i)_{i\leq0}}$, and let $\widehat{\eta}$ be an ergodic $\widehat{F}$-invariant Borel probability. Then the limit superior in \eqref{EOEform} is in fact a limit, and for $\widehat{\eta}$-a.e. $(\omega,t)$,
    $$
        \widehat{h}_{\widehat{\eta}}(\widehat{F})=\lim_{n\to\infty}\frac{-1}{n}\log\eta_{(\omega^-,t)}([\omega]_n),
    $$
     $\widehat{\eta}=\int \eta_{(\omega^-,t)} d\tau(t)$ is the disintegration of $\widehat{\eta}$ into conditional measures by the measurable partition $\{V^u_{\leq -1}(\omega^-)\times\{t\}\}_{\omega^-\in \Sigma^-, t\in X}$.
\end{lemma}

\begin{remark}\label{importRmk} Under the assumptions of Lemma \ref{strongUEnt} (i.e. $F_\omega=F_{(\omega_i)_{i\leq0}}$), the entropy out of equilibrium is in fact the conditional entropy corresponding the increasing measurable partition of strong unstable leaves (i.e. the unstable direction of the symbolic part).
\end{remark}

\begin{definition}[Entropy out of equilibrium {\cite[Definition~3.9]{TDFOE_I}}]
We define the {\em entropy out of equilibrium}:    $$\wh h (\wh F):=\wh P(0).$$
\end{definition}

\subsection{Invariant Measures from Conformal Measures and Harmonic Functions of the Semi-Ruelle Operator}

In this section we describe how the existence of both harmonic functions and conformal measures implies the existence of an associated $\wh F$-invariant measure.

\begin{lemma}[Adapted invariant measure]\label{InvExists}
Assume that $F_\omega=F_{(\omega_i)_{i\leq0}}$ and let $\wh \phi\in\mathrm{H\ddot{o}l}(\Sigma^-\times X)$. Let $\wh p$ be a $\wh \phi$-conformal measure, %which projects to $\tau\in \mathbb{P}(X)$. 
and let $\wh\rho\in L^1(\wh p)$ be a positive $\wh\phi$-harmonic function s.t. %$\log\wh\rho\in L^\infty(\wh p)$ and 
$\int \wh\rho d\wh p=1$. Then
$$\wh\nu^-:=\wh\rho\cdot \wh p$$
defines an $\wh F_R$-invariant probability measure which extends uniquely to an $\wh F$-invariant probability measure on $\Sigma\times X$, $\wh\nu$. That is,
$$\wh\nu^-=\wh\nu\circ \wh\pi^{-1},$$
where $\wh\pi:\Sigma\times X\to \Sigma^-\times X$ is the projection.
\end{lemma}
\begin{proof}
By definition, $\wh\nu^-$ is indeed a probability measure on $\Sigma^-\times X$. To see that it is $\wh F_R$-invariant, let $\wh g\in C(\Sigma^-\times X)$. Then, by the conformality and harmonicity,
\begin{align*}
	\int \wh g\circ \wh F_R d\wh\nu^-=& \int \wh\rho \wh g\circ \wh F_R d\wh p=\int e^{-\wh P(\wh\phi)}\L_{\wh\phi}(\wh \rho\cdot \wh g\circ \wh F_R)d\wh p\\
	=& \int \wh g\cdot e^{-\wh P(\wh\phi)}\L_{\wh\phi}(\wh \rho)d\wh p= \int \wh g\cdot \wh \rho d\wh p=\int \wh g d\wh \nu^-.
\end{align*} 

To see that $\wh\nu^-$ extends uniquely to an $\wh F$-invariant measure, we define the following unique extension. We consider the following collection of open sets which generate the topology of $\Sigma^-\times X$: $$\{T^\ell[[\underline w]]\times B: |\ul w|=n, \ell\in \mathbb{Z}, B=B(t_B,r_B)\subseteq X\}.$$
Given $B, \ul w$, and $\ell$, set 
$$\wh \nu(T^\ell[[\underline w]]\times B):=\wh \nu^-(\wh F^{-\ell}[T^\ell[[\underline w]]\times B])= \wh \nu^-([\underline w]\times \Pi_X\circ\wh F^{-\ell}[ \Sigma\times B]),$$
where $\Pi_X:\Sigma\times X\to X$ is the projection. Indeed, the projection of an open set is open, and so measurable. We leave it to the reader to verify that this defines the unique extension of $\wh \nu^-$ into an $\wh F$-invariant measure on $\Sigma\times X$. 
\end{proof}

\subsection{Variational Principle}\label{VarPrinceSubSect}

We decompose the proof of the variational principle into the lower bound and the upper bound%, where we require weaker assumptions for the lower bound (as in \textsection \ref{ConfExistSect}), while the upper bound requires the additional assumptions of \textsection \ref{harmExistSec}
.

\begin{prop}[POE lower bound]\label{POElowerBound}
	Assume that $(\Sigma\times X,\wh F)$ admits $\wh\phi$-holonomies, then for any $\wh F$-invariant Borel probability measure $\wh \eta$,
	$$\wh P_{\wh\eta}(\wh\phi)\leq \wh P(\wh \phi).$$
\end{prop}
\begin{proof}
	The metric pressure out of equilibrium is affine w.r.t. the ergodic decomposition, hence it is enough to prove the lower bound for ergodic measure. Let $\wh\eta$ be an ergodic $\wh F$-invariant Borel probability measure. Recall the information function
	$$\wh I_{\wh \eta}(\omega,t):=-\log \eta_{(\omega^-,t)}([\omega_0]\times X),$$
	where $\wh\eta=\int \eta_{(\omega^-,t)}d\wh\eta^- $ is the disintegration of $\wh\eta$ by \\ $\{V^u_{\leq -1}(\omega^-)\times \{t\}\}_{(\omega^-,t)\in \Sigma^-\times X}$, $V^u_{\leq -1}(\omega^-):=\{\omega\in \Sigma: \omega_i=\omega^-_i, \forall i\leq-1\}$, and $\wh\eta^-$ is the projection of $\wh\eta$ to $\Sigma^-\times X$. Then, by the point-wise ergodic theorem and by the dominated convergence theorem (as $\wh I_{\wh \eta} \in L^1(\wh\eta)$),
\begin{align*}
	\wh P_{\wh\eta}(\wh\phi)=&\int\lim_{n\to\infty}\frac{1}{n}\sum_{k\leq n}(\wh I_{\wh\eta}+\wh\phi)\circ \wh F^k d\wh\eta\\
	=& \lim_{n\to\infty}\int\frac{1}{n}\sum_{k\leq n}(\wh I_{\wh\eta}+\wh\phi)\circ \wh F^k d\wh\eta.
\end{align*}
Indeed, by the $\wh\phi$-holonomies, 
\begin{align}\label{forLowBd}
\wh P(\wh \phi)\geq& \limsup_n\frac{1}{n}\log \Big((\L_{\wh\phi}^*)^n\wh\eta\Big) 1= \limsup_n\frac{1}{n}\log \int \L_{\wh\phi}^n1 d\wh\eta\nonumber\\
=& \limsup\frac{1}{n}\log\int \Big(\sum_{\overset{|\underline{w}|=n,}{\omega^-_{-1}\to w_0}}\frac{e^{\widehat{\phi}^{(n)}}}{\eta_{(\omega^-,t)}([\underline{w}])}\cdot \eta_{(\omega^-,t)}([\underline{w}]) \Big)d \wh\eta\nonumber\\
%=& \limsup\frac{1}{n}\log\int \Big(\sum_{\overset{|\underline{w}|=n,}{\omega^-_{-1}\to w_0}}\frac{e^{\widehat{\phi}^{(n)}(\omega_{\ul w,\omega^-%,t
%},t)}}{\eta_{(\omega^-,t)}([\underline{w}])}\cdot \eta_{(\omega^-,t)}([\underline{w}]) \Big)d \wh\eta\\
\geq& \limsup\int \Big(\sum_{\overset{|\underline{w}|=n,}{\omega^-_{-1}\to w_0}} \frac{1}{n}\log\frac{e^{\widehat{\phi}^{(n)}(\omega_{\ul w,\omega^-%,t
},t)}}{\eta_{(\omega^-,t)}([\underline{w}])}\cdot \eta_{(\omega^-,t)}([\underline{w}]) \Big)d \wh\eta\nonumber\\
=& \limsup_{n\to\infty}\int\frac{1}{n}\sum_{k\leq n}(\wh I_{\wh\eta}+\wh\phi)\circ \wh F^k d\wh\eta=\wh P_{\wh\eta}(\wh\phi),
\end{align}
where $\omega_{\ul w, \omega^-}$ %maximizes $\wh\phi^{(n)}$ on 
is an element of $(V^u_{\leq -1}(\omega^{-})\cap[\ul w])\times \{t\}$ and the second inequality is due to Jensen. 
\end{proof}

\begin{theorem}[POE upper bound]\label{POEupperBound}
	Assume that $(\Sigma\times X,\wh F)$ admits %mild unstable 
	$\wh\phi$-holonomies and that $F_\omega=F_{(\omega_i)_{i\leq0}}$%, and that $\wh\phi\in \mathrm{H\ddot{o}l}(\Sigma^-\times X)$
	. Let $\wh p$ be a $\wh\phi$-conformal measure, and assume that there exists $\wh\rho\in L^\infty(\wh p)$ with $\|\log \wh\rho\|_{L^\infty(\wh p)}<\infty$. Then, $$\wh P(\wh\phi)=%\max\Big\{
	\wh P_{\wh \rho\cdot \wh p}(\wh\phi)%:\wh\eta\text{ is }\wh F\text{-invariant}\Big\}
	.$$
\end{theorem}
\begin{proof} %The lower bound is given by Proposition \ref{POElowerBound}. 
We show that the measure $\wh\nu:=\wh\rho\cdot\wh p$ (which is $\wh F_R$-invariant by Lemma \ref{InvExists}) satisfies $\wh P_{\wh\nu}(\wh\phi) = \wh P(\wh\phi) $. In addition, by Lemma \ref{ConfOfErg} we may assume that $\wh\nu$ is ergodic ($\wh p$ being ergodic implies that $\wh \rho\cdot \wh p$ is ergodic).

Let $\wh\nu=\int \wh\nu_{(\omega^-,t)}d\wh\nu^- $ be the disintegration of $\wh\nu$ by \\ $\{V^u_{\leq -1}(\omega^-)\times \{t\}\}_{(\omega^-,t)\in \Sigma^-\times X}$, $V^u_{\leq -1}(\omega^-):=\{\omega\in \Sigma: \omega_i=\omega^-_i, \forall i\leq-1\}$, where $\wh\nu^-$ is the projection of $\wh\eta$ to $\Sigma^-\times X$. Recall that $\wh\nu=\wh\rho\cdot \wh p$, with $\wh p$ being $\phi$-conformal (Lemma \ref{InvExists}). Moreover, by Theorem \ref{harmExist}, $\log \wh\rho\in %\B^{\gamma,\infty}_\tau
L^\infty(\wh p)$%  where $\tau$ is the projection of $\wh\nu$ to $X$
.

Given $\omega\in \Sigma$ and $n\in \mathbb{N}$, write $[\omega]_n:=[\omega_0,\ldots, \omega_{n-1}]$. Let $C_{\wh\rho}=\|\log\wh\rho\|_{%\gamma, \infty, \tau
L^\infty(\wh p)}$. Finally, let $\wh p^\pm := \frac{1}{\wh\rho}\cdot \wh\nu$, and let $\wh p^\pm=\int \wh p_{(\omega^-,t)}d\wh p $ be the disintegration of $\wh p^\pm$ by $\{V^u_{\leq -1}(\omega^-)\times \{t\}\}_{(\omega^-,t)\in \Sigma^-\times X}$.

Then, for $\wh\nu$-a.e. $(\omega,t)$,
\begin{align}\label{newEntLocForm}
\wh \nu_{(\omega^-,t)}([\omega]_n\times X)=e^{\pm C_{\wh\rho}} \wh p_{(\omega^-,t)}([\omega]_n\times X).
\end{align}
By Proposition \ref{newSemiGibbsProp} and Theorem \ref{harmExist}, $C_{\wh\rho}'>0$ s.t. for $\wh\nu$-a.e. $(\omega,t)$,
\begin{comment}
Let $\wh g\in \mathrm{Lip}(\Sigma\times X)$. Write $h_{\wh g}(\omega^-,t):= \wh p_{(\omega^-,t)}(\wh g)$, and so,
$$\int \wh g\circ \wh Fd \wh\nu =\int \wh g\wh \nu =\int h_{\wh g}d\wh \nu .$$
 Then by the $\wh\phi$-conformality of $\wh p$,
\begin{align*}
	\int \wh gd\wh p^\pm =&\int h_{\wh g}d\wh p= \int e^{-\wh P(\wh \phi)} \L_{\wh\phi} h_{\wh g}d\wh p\\
	=&\int e^{-\wh P(\wh \phi)} \sum_{\wh F_R(\wt\omega^-,\wt t)=(\omega^-,t)}e^{\wh \phi(\wt\omega^-,\wt t)} p_{(\omega^-,t)}(\wh g) d\wh p. 
\end{align*}
Therefore, by the uniqueness of conditional measures and the \\ Riesz–Markov–Kakutani representation theorem, 
for $\wh p$-a.e. $(\omega^-,t)$,
$$\wh p_{(\omega^-,t)}\circ \wh F^{-k}= e^{-\wh P(\wh \phi)} \sum_{\wh F_R(\wt\omega^-,\wt t)=(\omega^-,t)}e^{\wh \phi(\wt\omega^-,\wt t)} \wh p_{(\omega^-,t)}.$$
Substituting this back in \eqref{newEntLocForm}, we get 
\end{comment} 
\begin{equation}\label{GibbsEstsEq}
	\wh \nu_{(\omega^-,t)}([\omega]_n\times X)=e^{\pm C_{\wh\rho}'} e^{-n\wh P(\wh \phi)} e^{\wh \phi^{(n)}(\omega,t)}.
\end{equation}
By Lemma \ref{strongUEnt} and the point-wise ergodic theorem, we get that $\wh h_{\wh\nu}(\wh F)= \wh P(\wh \phi)-\int \wh\phi d\wh\nu$.
\end{proof}

\begin{remark}
	Indeed, in the setting of \textsection \ref{AppCount}, we have an example for which the POE variational principle is satisfied (as it satisfies the assumptions of \cite{TDFOE_I}), and indeed there exists a measure which obtains the maximum in the POE variational principle, but there is no harmonic function in the $L^1$ of no conformal measure. This highlights how the assumptions of Theorem \ref{POEupperBound} are necessary, as not always we can find a POE maximizing measure which is obtained as the product of a conformal measure times a harmonic function.
\end{remark}

\begin{cor}\label{rmkEOE}
    In the setting of Theorem \ref{POEupperBound}, we get that 
    \begin{equation*}
        \wh h (\wh F)=\max\Big\{\wh h_{\wh\eta}(\wh F): \wh\eta\text{ is }\wh F\text{-inv.} \Big\}= h_\mathrm{top}(T).
    \end{equation*} 
\end{cor}
\begin{proof}
	The first equality is due to the POE variational principle (Theorem \ref{POEupperBound}), the second inequality is as in \cite[Proposition~6.1]{TDFOE_I}.
\end{proof}

\section{Smooth Measures, Adapted Hilbert Spaces, Eigen-Measures, and the Semi-Ruelle Operator as the Dual of Random Dynamics}\label{GibbsProcess} 
In this section we wish to study a finer structure for the action of the semi-Ruelle operator, so we are able to apply tools from spectral theory. Namely, we wish to find a new space of functions on which the semi-Ruelle  operator acts, %we no longer wish to immediately construct a harmonic function within our space (as in \textsection \ref{harmExistSec}), without establishing further spectral properties; However, 
%once 
so we can later establish %ing those 
additional spectral properties on %the new space, we are able to construct a harmonic function with additional properties 
it (e.g. uniqueness of the harmonic function within our space, regularity along the $X$ direction, etc.). We require a few preliminary definitions.

\subsection{Adapted Hilbert Space}

\begin{definition}[Smooth measure]\label{defLogJac}
	We say that $(\Sigma\times X,\wh F)$ admits a {\em smooth measure} $m\in\mathbb{P}(X)$ with an exponent $\theta\in (0,1]$ if for every $\omega\in \Sigma$, $$m\circ F_{T^{-1}\omega} \sim m,$$
	and furthermore, there exists a representative $J_{\omega}\in \mathrm{H\ddot{o}l}_\theta(X)$ s.t.  
	$$\frac{d m\circ F_{T^{-1}\omega} }{dm}\circ F_{T^{-1}\omega}^{-1}=J_{\omega} \ m\text{-a.e.},$$
and $(\omega, t)\mapsto J_\omega(t)$ is in $\mathrm{H\ddot{o}l}_\theta(\Sigma\times X)$.
	
Moreover, if $F_\omega=F_{(\omega_i)_{i\leq0}}$, then $\omega\mapsto J_\omega$ is constant on elements of the form $V^u_{\leq-1}(\omega^-)$ as $F_{T^{-1}\omega}=(F_{\omega^-}^{-1})^{-1}$, and so we define $$\wh \varphi(\omega^-, t):= \log J_{\omega^-}(t)\in \mathrm{H\ddot{o}l}_\theta(\Sigma^-\times X).$$ We call $\wh\varphi$ {\em the associated log-Jacobian function}.

\end{definition}

\begin{remark}
	If $m$ is fully supported on $X$, then each element in the family $\{J_\omega\}_{\omega\in \Sigma}$ must be unique.
\end{remark}

\begin{definition}\label{properFormPot}
A potential $\psi\in \mathrm{H\ddot{o}l}(\Sigma^-)$ is called {\em in proper form} if 
$$L_\psi 1=1,$$
where $L_\psi:C(\Sigma^-)\to C(\Sigma^-)$ is the associated Ruelle operator.
\end{definition}

\begin{remark}\text{ }
\begin{enumerate}
	\item Every H\"older continuous potential $\wt\psi$ is cohomologous via a bounded H\"older coboundary to a potential in a proper form. This can be arranged easily as follows: Let $\mu$ be the $\wt\psi$-equilibrium measure, write $\psi(\omega^-):=\log \mu_{T_R\omega^-}(T^{-1}[V^u_{\leq-1}(\omega^-)])$, where $\{\mu_{\wt\omega^-} \}_{\wt\omega^-\in \Sigma^-}$ is the disintegration of $\mu$ by conditionals on symbolic unstable leaves.
	\item In particular, $\psi$ being in proper form implies that
\begin{enumerate}
	\item $\psi<0$, 
	\item $P(\psi)=0$. 
\end{enumerate}
\end{enumerate}
\end{remark}

\begin{lemma}[Bounded operator on a smooth measure $L^2$ space]\label{bddOnL2m}
Assume that $F_\omega=F_{(\omega_i)_{i\leq0}}$. Let $\psi$ be a potential in a proper form, and let $\mu$ be its associated Gibbs state. Let $m$ be a smooth measure on $X$, with its associated log-Jacobian function $\wh\varphi\in \mathrm{H\ddot{o}l}(\Sigma^-\times X)$. Let $\wh\phi:=\psi+\wh\varphi$. Write $$\wh m:=\mu\times m.$$
Then, $\L_{\wh\phi}$ defines a bounded linear operator on $L^2(\wh m)$. In addition, the Koopman operator is also bounded on $L^2(\wh m)$.
\end{lemma}
\begin{proof} We begin by showing that $\L_{\wh\phi}$ is bounded. 

\textbf{Step 1:} First, we argue that $L_\psi$ is a bounded linear operator on $L^1(\mu)$. We break the proof down into 3 steps, relying on the positivity of $L_\psi$.
\begin{enumerate}
\item[(a)] Given an open set $U$, $\mathbb{1}_U$ can be approximated from below as the point-wise limit of increasing continuous functions $g_n$. Then, by the monotone convergence theorem (as $L_\psi$ is positive), $\int L_\psi \mathbb{1}_U d\mu=\lim_n \int L_\psi g_n d\mu =\lim_n \int g_n d\mu=\int gd\mu$.
\item[(b)] Indeed, the measure $\mu'(A):=\int L_\psi \mathbb{1}_Ad\mu$ defines for all open sets $\mu'(U)=\mu(U)$. Thus $\mu'=\mu$, and so $\mu(A)=\int L_\psi \mathbb{1}_Ad\mu$ for all measurable sets.
\item[(c)] Lastly, every non-negative function $g\in L^1(\mu)$ can be approximated in $L^1(\mu)$, in an increasing point-wise limit, by simple functions $\{g_n\}_n$. Then, once again by the positivity of $L_\psi$ and the monotone convergence theorem,
$$\int L_\psi g d\mu= \int \lim L_\psi g_n d\mu = \lim \int L_\psi g_n d\mu = \lim \int  g_n d\mu =\int g d\mu.$$ 	
\end{enumerate}

\textbf{Step 2:} Next, notice that since $L_\psi1=1$, and so by Jensen's inequality, for $\wh g\in L^1(\wh m)$, for all $(\omega^-,t)\in \Sigma^-\times X$,
$$(\L_{\wh\phi}|\wh g|)^2(\omega^-,t)\leq e^{2\|\wh\varphi\|_\infty}\sum_{T_R\wt\omega^-=\omega^-}e^{\psi(\wt\omega^-)} |\wh g|^2 (\wt\omega^-,F_{\omega^-}(t)).$$
Then, for $G(\omega^-):=\int |\wh g|^2(\omega^-,F_{\omega^-}(t))dm(t)$, where $G\in L^1(\mu)$ as $G(\omega^-)=e^{\pm \|\wh\varphi\|_\infty}\int |\wh g|^2(\omega^-,t)dm(t)$ and $\int \int |\wh g|^2(\omega^-,t)dm(t) d\mu= \|\wh g\|_{L^2(\wh m)}^2<\infty$, 
\begin{align}\label{alsoCoolEq2}
	\int \int (\L_{\wh\phi}|\wh g|)^2(\omega^-,t)dmd\mu\leq &e^{2\|\wh\varphi\|_\infty}\int \sum_{T_R\wt\omega^-=\omega^-}e^{\psi(\wt\omega^-)} G(\wt\omega^-)d\mu\nonumber\\
	=& e^{2\|\wh\varphi\|_\infty}\int G d\mu= \int\int |\wh g|^2(\omega^-,F_{\omega^-}(t))dm(t)d\mu\nonumber\\
	=& e^{2\|\wh\varphi\|_\infty}\int\int \frac{d m\circ F_{\omega^-}^{-1}}{d m} |\wh g|^2(\omega^-,t)dm(t)d\mu\nonumber\\
	\leq& e^{3\|\wh\varphi\|_\infty}\int |\wh g|^2d\wh m\leq e^{3\|\wh\varphi\|_\infty}\|\wh g\|_{L^2(\wh m)}^2.
\end{align}

\textbf{Step 3:} We continue to show the boundedness of the Koopman operator:
\begin{align*}
	\int |\wh g|^2\circ \wh F_Rd\wh m=&\int \int |\wh g|^2(T_R\omega^-,F_{\omega^-}^{-1}(t))dmd\mu \\
	=&e^{\pm2 \|\wh\varphi\|_\infty}\int \int |\wh g|^2(T_R\omega^-,t)dmd\mu= e^{\pm2 \|\wh\varphi\|_\infty}\int |\wh g|^2d\wh m.
\end{align*} 
\end{proof}

\begin{theorem}[Adjoint operator]\label{coolThm2}
Assume that $F_\omega=F_{(\omega_i)_{i\leq0}}$. Let $\psi$ be a potential in a proper form, and let $\mu$ be its associated Gibbs state. Let $m$ be a smooth measure on $X$, with its associated log-Jacobian function $\wh\varphi\in \mathrm{H\ddot{o}l}(\Sigma^-\times X)$. Let $\wh\phi:=\psi+\wh\varphi$, and set $$\wh m:=\mu\times m.$$  
Then $\L_{\wh\phi}$ is the adjoint of the Koopman operator on $L^2(\wh m)$. Namely, for all $\wh g,\wh h\in L^2(\wh m) $,
	$$\int \wh h \circ \wh F_R\cdot \wh g d\wh m= \int  \wh h \L_{\wh\phi}\wh g d\wh m .$$
\end{theorem}
\begin{proof} First, we notice that since $C(\Sigma^-\times X:\mathbb{C})$ is dense in $L^2(\wh m)$, it is enough to prove the statement for any two $\wh g,\wh h\in C(\Sigma^-\times X:\mathbb{C})$ (recall Lemma \ref{bddOnL2m}).

Next, by Definition \ref{defLogJac}, given $(\omega^-,t)\in \Sigma^-\times X$, we have
$$e^{\wh\varphi(\omega^-,(F_{\omega^-}^{-1})^{-1}(t))}=\frac{dm\circ (F_{\omega^-}^{-1})^{-1}}{dm}(t).$$
So, for all $ g\in C(X)$,
$$\int e^{\wh\varphi(\omega^-,(F_{\omega^-}^{-1})^{-1}(t))} g(t)dm=\int g\circ F_{\omega^-}^{-1}dm.$$
Then we can continue to show that $\wh m$ is an eigen-measure of $\L_{\wh\phi}^*$. Let $\wh g\in C(\Sigma^-\times X:\mathbb{C})$, then
\begin{align*}
	&\int \L_{\wh\phi}\wh g d\wh m=\int \int  (\L_{\wh\phi}\wh g)(\omega^-,t)dm d\mu\\
	=& \int\sum_{T_R\wt\omega^-=\omega^-}e^{\psi(\wt\omega^-)} \int e^{\wh\varphi(\wt\omega^-,(F_{\wt\omega^-}^{-1})^{-1}(t))}\wh g(\wt\omega^-,(F_{\wt\omega^-}^{-1})^{-1}(t))dmd\mu\\
	=& \int\sum_{T_R\wt\omega^-=\omega^-}e^{\psi(\wt\omega^-)} \int \frac{dm\circ (F_{\wt\omega^-}^{-1})^{-1}}{dm}(t)\wh g(\wt\omega^-,(F_{\wt\omega^-}^{-1})^{-1}(t))dmd\mu\\
	=& \int\sum_{T_R\wt\omega^-=\omega^-}e^{\psi(\wt\omega^-)} \int \wh g(\wt\omega^-,t)dmd\mu=\int\int \sum_{T_R\wt\omega^-=\omega^-}e^{\psi(\wt\omega^-)}\wh g(\wt\omega^-,t)d\mu dm(t)\\
	=&\int \int\wh g(\omega^-,t)d\mu d m(t).
\end{align*}
Finally, given $\wh g,\wh h\in C(\Sigma^-\times X:\mathbb{C}) $, we have by the fact that $\wh m$ is an eiegn-measure of $\L_{\wh\phi}^*$,
$$\int \wh h \L_{\wh\phi}\wh g d\wh m= \int  \L_{\wh\phi}(\wh h\circ \wh F_R \cdot \wh g )d\wh m= \int\wh h\circ \wh F_R\cdot \wh gd\wh m.$$
\end{proof}

\begin{remark}\text{ }
\begin{enumerate}
	\item In Theorem \ref{coolThm2}, while we prove that $\L_{\wh\phi}^*\wh m= \wh m $, we refrained from using the terminology ``$\wh m$ is $\wh\phi$-conformal", as we have not shown that $\wh P(\wh\phi)=0$.
	\item Note that $\wh \phi$ can be read off of $\wh m=\mu\times m$, as the sum of the log-Jacobians of the two measure factors. 
\end{enumerate}
\end{remark}

Corollary \ref{smoothMes} is a direct corollary of Theorem \ref{coolThm2}.
\begin{cor}[Smooth eigen-measure]\label{smoothMes}
     	In the setting of Theorem \ref{coolThm2}, $\wh m$ is an eigen-measure of $\L_{\wh\phi}$: $$\L_{\wh\phi}^*\wh m=\wh m.$$
In particular, for all $\wh g\in L^2(\wh m) $,
	$$\int  \wh g d\wh m= \int  \L_{\wh\phi}\wh g d\wh m .$$
\end{cor}

\begin{cor}[Pressure bound]\label{PBound}
	Under the assumptions of Theorem \ref{coolThm2}, assuming also $\wh\phi$-holonomies,
	$$\wh P(\wh\phi)\geq0 .$$
\end{cor}
\begin{proof} Rceall Corollary \ref{smoothMes}. Then for all $n\geq0$,
	$$\frac{1}{n}\log \|\L_{\wh\phi}^n1\|_\infty\geq\frac{1}{n}\log \int \L_{\wh\phi}^n1d\wh m=0.$$
\end{proof}

\subsection{Bounded Operator on Adapted Hilbert Space}

In this section we introduce a Hilbert space whose inner-product as a natural Hilbert space on which the semi-Ruelle operator acts. We show that it preserves the space, and admits a bounded norm. 

\begin{definition}[Adapted Hilbert space]\label{defOfHil2}
 Given %$\wh p\in \mathbb{P}(\Sigma^-\times X)$, and $\theta\in (0,1]$. Then
 a Gibbs state $\mu$ in a proper form on $\Sigma^-$ and a smooth measure $m$ on $M$, the {\em associated inner-product} is formally defined for any two  functions $\wh g,\wh h\in L^2(\wh m)$, $\kappa>0$:
\begin{align*}
\langle \wh g,\wh h\rangle &:= \int \wh g \overline{\wh h}d\wh m\\
&+\int \int\int \frac{(\wh g(\omega^-,t)-\wh g(\omega^-,s))\overline{(\wh h(\omega^-,t)-\wh h(\omega^-,s))}}{d(t,s)^{\kappa}}dm(t)dm(s)d\mu(\omega^-) .
\end{align*}
%where $\wh p=\int \wh p_{\omega^-}dp(\omega^-)$ is the disintegration of $\wh p$ w.r.t. the measurable partition $\{\{\omega^-\}\times X\}_{\omega^-\in \Sigma^-}$.

The {\em associated Hilbert space} $\HH_\kappa(\wh m)$ is defined as the space of all functions in $L^2(\wh m)$ whose $\langle \cdot,\cdot\rangle$-induced norm is finite:
\begin{align*}
&\|\wh g\|_{\HH_\kappa(\wh m)}:=\\
&\sqrt{\int |\wh g|^2d\wh m+\int\int \int \frac{|\wh g(\omega^-,t)-\wh g(\omega^-,s)|^2}{d(t,s)^{\kappa}} dm(t)dm(s)d\mu(\omega^-)}.
\end{align*}
%is defined as the unique (up to an isometry of Hilbert spaces) completion of $\Big(\mathrm{H\ddot{o}l}_\theta(\Sigma^-\times X:\mathbb{C}),\langle\cdot,\cdot\rangle\Big)$, denoted by $\HH_\theta(\wh m)$.
\end{definition}

\begin{claim}[Extension of fiber H\"older functions]\label{HolderExt}
Assume that there exists $\theta>0$ s.t. $\frac{1}{d(t,s)^{\kappa-2\theta}}\in L^1(m\times m)$, then $\mathrm{H\ddot{o}l}_\theta(\Sigma^-\times X)\subseteq \HH_\kappa(\wh m)$.
\end{claim}
\begin{proof} 
Let $\wh g\in \mathrm{H\ddot{o}l}_\alpha(\Sigma^-\times X) $, then 
$$\|\wh g\|_{\HH_\kappa(\wh m)}^2\leq \|\wh g\|_\infty^2+\|\wh g\|_{\mathrm{H\ddot{o}l}_\theta}^2\cdot \int \int \frac{1}{d(t,s)^{\kappa-2\theta}} dm dm <\infty.$$
\end{proof}

\begin{prop}[Bounded operator on the adapted Hilbert space]\label{bddOnSmooth}
Assume that $F_\omega=F_{(\omega_i)_{i\leq0}}$%, and that $\{F_\omega\}_{\omega\in \Sigma}$ is an equi-bi-H\"older family
. Let $\wh\phi\in \mathrm{H\ddot{o}l}(\Sigma^-\times X)$, and let $\wh p$ be a $\wh\phi$-conformal measure. Assume that there exists $C>0$ s.t. for $t\in X$, $\int \frac{1}{d(t,s)^{\kappa-2\theta}}dm\leq C$, where $\wh\phi$ is $\theta$-H\"older. Then, $\L_{\wh\phi}$ defines a bounded linear operator on $\HH_\kappa(\wh m)$.
\end{prop}
\begin{proof}
Let $\wh g\in \HH_\kappa(\wh m)$. Lemma \ref{bddOnL2m} (particularly \eqref{alsoCoolEq2}) bounds the $L^2$-norm term of $\|\L_{\wh\phi}\wh g\|_{\HH_\kappa(\wh m)}^2$ by $e^{3\|\wh\varphi\|_\infty}\|\wh g\|_{L^2(\wh m)}$. We are left to bound the variation term. 

By the Jensen inequality, since $L_\psi 1=1$,
\begin{align*}
&	\sum_{T_R\sigma^-=\omega^-} | e^{\psi(\sigma^-)}(e^{\wh\varphi(\sigma^-,F_{\omega^-}(t))}\wh g(\sigma^-, F_{\omega^-}(t))-e^{\wh\varphi(\sigma^-,F_{\omega^-}(s))}\wh g(\sigma^-, F_{\omega^-}(s))) |^2\\
\leq&	\sum_{T_R\sigma^-=\omega^-} e^{\psi(\sigma^-)}\\
& |(e^{\wh\varphi(\sigma^-,F_{T_R\sigma^-}(t))}\wh g(\sigma^-, F_{T_R\sigma^-}(t))-e^{\wh\varphi(\sigma^-,F_{T_R\sigma^-}(s))}\wh g(\sigma^-, F_{T_R\sigma^-}(s))) |^2.
\end{align*}
Then, using the conformality of $\mu$, and inequality $|a+b|^2\leq 2|a|^2+2|b|^2$, we get,

\begin{align} \label{someEq}
	&\int\int\int \frac{|(\L_{\wh \phi}\wh g)(\omega^-,t)-(\L_{\wh \phi}\wh g)(\omega^-,s)|^2}{d(t,s)^\kappa}d\mu dmdm\nonumber\\
	\leq&2\int \int\int \frac{|e^{\wh\varphi(\omega^-,F_{T_R\omega^-}(t))}-e^{\wh\varphi(\omega^-,F_{T_R\omega^-}(s))}|^2\cdot|\wh g (\omega^-,F_{T_R\omega^-}(t)) |^2}{d(t,s)^\kappa} dmdmd\mu\nonumber\\
	+& 2\int \int\int \frac{e^{2\wh\varphi(\omega^-,F_{T_R\omega^-}(s))}|\wh g (\omega^-,F_{T_R\omega^-}(t))-\wh g (\omega^-,F_{T_R\omega^-}(s)) |^2}{d(t,s)^\kappa}dmdmd\mu\nonumber\\
	\leq &2\int \int\int \frac{|e^{\frac{1}{2}(\wh\varphi(\omega^-,t)-\wh\varphi(\omega^-,s))}-1|^2\cdot|\wh g (\omega^-,t) |^2}{d(t,s)^\kappa} dmdmd\mu\nonumber\\
	+& 2\int \int\int \frac{e^{\wh\varphi(\omega^-,s)-\wh\varphi(\omega^-,t)}|\wh g (\omega^-,t)-\wh g (\omega^-,s) |^2}{d(t,s)^\kappa}dmdmd\mu.
\end{align}

Since $\wh \varpi$ is $\theta$-H\"older continuous, there exists $D>0$ s.t. $\frac{|e^{\frac{1}{2}(\wh\varphi(\omega^-,t)-\wh\varphi(\omega^-,s))}-1|^2}{d(t,s)^{2\theta}}$ is globally bounded by $D$. Then the r.h.s. of \eqref{someEq} is bounded by
$$2DC \|\wh g\|_{L^2(\wh m)}+2e^{2\|\wh\varphi\|_\infty}\|\wh g\|_{\HH_\kappa(\wh m)},$$
where we use $\int   |\wh g(\omega^-,t)|^2\int \frac{1}{d(t,s)^{\kappa-2\theta}}d m(s) d\wh m(\omega^-,t)\leq C\|\wh g\|^2_{L^2(\wh m)}$. Thus, in total,
$$\|\L_{\wh \phi}\wh g\|_{\HH_\kappa(\wh m)}^2\leq (e^{3\|\wh\varphi\|_\infty} +2DC)\|\wh g\|_{L^2(\wh m)}^2+2e^{2\|\wh\varphi\|_\infty}\|\wh g\|^2_{\HH_\kappa(\wh m)}.$$
\end{proof}

\section{Dissipative Gibbs Random Smooth Dynamics: Effective Expansion on Average and Lasota-Yorke Inequality}\label{FinalLY}
In this section we work in the the setting of Smooth Dynamics (i.e. the fiber is $X$ is a smooth manifold, and the randomly composed maps are smooth diffeomorphims), and as in \textsection \ref{GibbsProcess}, we allow the random composition to be driven by a Gibbs measure on $\Sigma^-$ (i.e. a Gibbs process). In particular, the goal of this section is present general machinery which does not assume that the smooth random maps are conservative (i.e. we allow them to be dissipative). This additional structure allows us to study finer spectral properties of the semi-Ruelle operator. In particular, when satisfying the open condition of effective expansion on average, we introduce Lyapunov forms, and a suitably adapted Hilbert space on which the semi-Ruelle operator acts and admits a Lasota-Yorke inequality. We define it formally below.

We draw our motivation for study of Gibbs process, in oppose to i.i.d. random dynamics for example, from the fact that we view the setting of random dynamics as modeling the dependence of our dynamical system on some exterior unknown system. In that case, we may assume that this dependence is governed by a physical measure on the exterior system, e.g. a hyperbolic and ergodic SRB measure. Such measures can be coded by a TMS (of countable states) and lift to a Gibbs measures (see \cite{SBO}). Although, for now, we focus on studying the case where the hyperbolic SRB measure can be coded by a finite state TMS (such as in an Anosov or Axiom A system).

\begin{remark}[Forwards dynamics]\label{forwardsDyns}
	For the sake of compatibility with existing literature, in this section we work the semi Ruelle operator 
	$$\L_{\wh\phi}:C(\Sigma^+\times X)\to C(\Sigma^+\times X),$$
	where $\Sigma^+:=\{(\omega)_{i\geq0}:\omega\in \Sigma\}$. It is defined by
$$(\L_{\wh\phi}\wh g)(\omega^+,t)=\sum_{\wh F_L(\sigma^+,s)=(\omega^+,t)}e^{\wh \phi(\sigma^+,s)} \wh g(\sigma^+,s),$$
where $\wh F_L(\sigma^+,s)=(T_L\sigma^+,F_{\sigma^+}(t))$, and $T_L:\Sigma^+\to \Sigma^+$ is the left-shift. Similarly, $\psi\in \mathrm{H\ddot{o}l}(\Sigma^+)$ is in a proper form if 
$\sum_{T_L\sigma^+=\omega^+}e^{\psi(\sigma^+)}=1$ for all $\omega^+$. 

Similarly, the log-Jacobian of a smooth family for the forwards dynamics is
$$\wh\varphi(\omega^+,x):=-\log \frac{dm\circ f_{\omega^+}}{dm}.$$
\end{remark}

\subsection{Dissipative Gibbs Random Smooth Dynamics}\label{UEASect}

Let $m$ be the normalized Riemannian volume on $M$- a closed Riemannian manifold of dimension $d\geq2$ (i.e. a smooth measure), with its associated log-Jacobian function $\wh\varphi\in \mathrm{H\ddot{o}l}(\Sigma^+\times M)$ (for the forwards dynamics, i.e. $\wh\varphi(\omega^+,x)=-\log \Jac_x(f_{\omega^+})$). Assume that $f_\omega=f_{(\omega_i)_{i\geq0}}$, and that $\{f_\omega\}_{\omega\in \Sigma}\subseteq \mathrm{Diff}^{1+\alpha}(M)$, $\alpha>0$, where $\omega\mapsto f_\omega$ is a H\"older continuous map. Let $\psi$ be a potential in a proper form, and let $\mu$ be its associated Gibbs state. Let $\wh\phi:=\psi+\wh\varphi$. Finally, denote $M_f:=\sup_{x\in M,\omega\in \Sigma}\{\|d_xf_\omega\|,\|d_xf_\omega^{-1}\|\}$.

By Lemma \ref{bddOnL2m}, $\L_{\wh\phi}$ is a bounded linear operator on $L^2(\wh m)$. By Theorem \ref{coolThm2}, it is the adjoint operator of $\wh g\mapsto \wh g\circ \wh f_L$, i.e. the the Koopman operator, on $L^2(\wh m)$.

\subsection{Effective Expansion on Average, and the Similarity Dimension}\label{UEASect2}

\begin{definition}[Derivative cocycle extension]
We denote the {\em derivative cocycle extension} by $\wh F_L:\Sigma^+\times T^1M\to \Sigma^+ \times T^1M$, where
$$\wh F_L(\omega^+, x,\xi):=\Big(T_L\omega^+, f_{\omega^+},\frac{d_xf_{\omega^+}\xi}{|d_xf_{\omega^+}\xi|}\Big).$$
Define the derivative potential
$$\wh D(\omega,x,\xi):=-\log |d_{x} f_{\omega}^{-1}\xi|,$$
as a H\"older potential on the skew-product system.
\end{definition}

\begin{remark}\label{rmk93}\text{ }
\begin{enumerate}
\item 	 Note that
\begin{equation}\label{forConv}
    \wh P(\psi-\kappa\cdot  \wh D)=\limsup
\frac{1}{n}\log \max_{(x,\xi)\in T^1M}\int |d_xf_{\omega}^{n}\xi|^{-\kappa}d\mu.
\end{equation}
\item When $f_\omega=f_{(\omega_i)_{i\leq0}}$, we are in the setting of \cite[Theorem~6.7]{TDFOE_I} (the variational principle), and so it is easy to observe that for all $\kappa\in\mathbb{R}$, $\wh P(\psi-\kappa\cdot  \wh D)$ is continuous in the $C^1$-topology of $\{f_\omega\}_{\omega\in \Sigma}$ and in the $C^0$-topology of $\psi$.
\end{enumerate}

\end{remark}

\begin{definition}[Similarity dimension]
	The {\em similarity dimension} associated with $(\Sigma\times M,\wh f)$ and $\mu$ is
	$$d^*:=\sup\Big\{\beta>0: \wh P(\psi- \beta\cdot \wh D)<0\Big\},$$
	and is defined as $0$ if the supremum is over an empty set.
\end{definition}

\begin{remark}
 By \cite[Theorem~8.12]{TDFOE_I}, $\wh D$-expansion on average (a $C^1$-open condition, see \cite[Definition~8.1]{TDFOE_I}) implies that the supremum is not over an empty set, and in particular $d^*>0$.
\end{remark}

\begin{lemma}[Similarity dimension bound]\label{dStarDBound}
	$$d^* \leq d.$$
\end{lemma}
\begin{proof}
Let $\kappa\geq d$, then by the Jensen inequality, if $d^*$ were to be greater than $d$, then
\begin{equation}\label{forforJens}
	\int |d_xf_{\omega}^{n}\xi|^{-\kappa}d\mu \geq \Big(\int |d_xf_{\omega}^{n}\xi|^{-d}d\mu \Big)^\frac{\kappa}{d}.
\end{equation}

We prove that there exists  a constant $c_M$ depending only on $M$ s.t. for all $n$, $\max_{(x,\xi)\in T^1M}\int |d_xf_{\omega}^{n}\xi|^{-d}d\mu\geq c_M$. Since the r.h.s. is independent of the diffeomorphisms $f_\omega$, it is enough to assume $n=1$. By \eqref{forforJens}, for all $n$, for all $\kappa\geq d$, $\max_{(x,\xi)\in T^1M}\int |d_xf_{\omega}^{n}\xi|^{-\kappa}d\mu\geq c_M^\frac{\kappa}{d}$, and so $\wh P(\psi-\kappa\wh D)\geq0$, hence $d^*\leq d$.

We integrate over $T_xM$, where we denote by $\sigma_{d-1}$ the normalized volume of $\mathbb{S}^{d-1}\subseteq T_xM$. Then again by the Jensen inequality, 
\begin{align}\label{eq1000}
\int\int |d_xf_{\omega}\xi|^{-d}d\mu d\sigma_{d-1}(\xi) = &\int\int |d_xf_{\omega}\xi|^{-d} d\sigma_{d-1}(\xi)d\mu\\
\geq &\int\Big(\int |d_xf_{\omega}\xi|^{d} d\sigma_{d-1}(\xi) \Big)^{-1} d\mu \nonumber\\
=&  \int\Big(\int \int_0^{|d_xf_{\omega}\xi|} d\cdot r^{d-1}dr d\sigma_{d-1}(\xi) \Big)^{-1} d\mu \nonumber\\
=&d \int\Big(\Vol(d_xf_\omega[B_{T_xM}(1)])\Big)^{-1} d\mu \nonumber\\
=& \frac{d}{\Vol(B_{T_xM}(1))}\int\Jac_x(f_\omega)^{-1} d\mu, \nonumber
\end{align}	
where $\Vol$ is the Euclidean volume on $T_xM$, and $B_{T_xM}(1) $ is the unit ball. Then we integrate both sides of \eqref{eq1000} w.r.t. to $m$, and apply Jensen's inequality again, 
\begin{align*}
	& \max_{(x,\xi)\in T^1M}\int |d_xf_{\omega}\xi|^{-d}d\mu \geq \int \int\int |d_xf_{\omega}\xi|^{-d}d\mu d\sigma_{d-1}(\xi)dm(x)\\
	\geq &\frac{d}{\Vol(B_{T_xM}(1))}\int \int\Jac_x(f_\omega)^{-1} d\mu dm(x)\\
\geq& \frac{d}{\Vol(B_{T_xM}(1))}\Big(\int \int\Jac_x(f_\omega) d\mu dm(x) \Big)^{-1}\\
=& \frac{d}{\Vol(B_{T_xM}(1))}\Big(\int \int\Jac_x(f_\omega) dm(x) d\mu  \Big)^{-1}\geq \min_{x\in M}\frac{d}{\Vol(B_{T_xM}(1))}\equiv c_M.
\end{align*}
\end{proof}

\begin{prop}[Convexity, directional derivatives, and vanishing]\label{CDDV}
    The function $\Xi:\kappa\mapsto \wh P(\psi-\kappa \wh D)$ is convex on $\mathbb{R}$ (thus in particular it is continuous), vanishes at $\kappa=d^*$ and at $\kappa=0$. If $d^*>0$, then $\Xi$ admits positive directional derivative at $\kappa=d^*$.
\end{prop}
\begin{proof}
    Convexity follows from \eqref{forConv}, which implies continuity and existence of directional derivatives. By continuity,
$$\Xi(d^*) =0,$$
as Lemma \ref{dStarDBound} tell us that $d^*$ is bounded. Indeed, $\wh P(\psi-d^*\wh D) \geq0$ as it can be approximated from above with positive values, and it cannot be negative as then by continuity for $d'>d^*$ we would we would get $\wh P(\psi-d'\wh D) <0 $.

To see that $\Xi$ admits positive directional derivatives, let $\beta\in(0,d^*)$ with $\Xi(\beta)<0$ (by the assumption $d^*>0$), and note that for all $\kappa \in (\beta,d^*)$,
$$\Xi(\kappa)\leq \lambda \Xi(\beta)+(1-\lambda)\Xi(d^*)=\lambda \Xi(\beta),$$
where $\lambda \beta+(1-\lambda)d^*=\kappa$, $\lambda>0$.

Next, for $\kappa>d^*$, $$0=\Xi(d^*)\leq \lambda'\Xi(\beta)+(1-\lambda')\Xi(\kappa)\Rightarrow \Xi(\kappa)\geq \frac{\lambda'}{1-\lambda'}\Xi(\beta),$$
where $\lambda' \beta+(1-\lambda')\kappa=d^*$, $\lambda'\in (0,1)$.
\end{proof}

\begin{comment}
\begin{remark}
	In Lemma \ref{ContSimDim} below we show that in particular%, when $f_\omega=f_{(\omega_i)_{i\leq0}}$
	,
	$$\wh P(\psi-d^*\wh D)=0,$$
	as the case where $d^*=0$ is trivial.
\end{remark}
\end{comment}

\begin{lemma}[Continuity of the similarity dimension]\label{ContSimDim}
	When $f_\omega=f_{(\omega_i)_{i\leq0}}$, if $d^*>0$, then it is continuous in the $C^1$-topology of $\{f_\omega\}_{\omega\in \Sigma}$ and in the $C^0$-topology of $\psi$.
\end{lemma}
\begin{proof}
%	Recall Remark \ref{rmk93}. By continuity,
%$$\wh P(\psi-d^*\wh D) =0,$$
%as Lemma \ref{dStarDBound} tell us that $d^*$ is bounded. Indeed, $\wh P(\psi-d^*\wh D) \geq0$ as it can be approximated from above with positive values, and it cannot be negative as then by continuity for $d'>d^*$ we would we would get $\wh P(\psi-d'\wh D) <0 $.

By Proposition \ref{CDDV}, $\wh P(\psi-d^*\wh D)=0$. Fix $\beta_0<d^*$ s.t. $\wh P(\psi-\beta_0\wh D)<0$, then by convexity, for all $\kappa>d^*$, for all $\epsilon>0$,
$$0= \wh P(\psi-d^*\wh D)\leq \lambda \wh P(\psi-\kappa\wh D)+(1-\lambda) \wh P(\psi-\beta_0\wh D), $$
where $\lambda \kappa +(1-\lambda)\beta_0=d^*$ for some $\lambda>0$. Then, for all $\kappa>d^*$,
$$ \wh P(\psi-\kappa\wh D) >0.$$

Similarly, for any $\beta\in (\beta_0,d^*)$, by convexity, 
	$$\wh P(\psi-\beta\wh D)\leq \lambda' \wh P(\psi-\beta_0\wh D)+(1-\lambda') \wh P(\psi-d^*\wh D)<0,$$
where $\lambda' \beta_0+(1-\lambda')d^*=\beta$ for some $\lambda'>0$.	
	
Then for all $\epsilon>0$ small enough there exists $\delta_\epsilon>0$ s.t.
\begin{equation}\label{openEqForDim}
	\wh P(\psi-(d^*-\epsilon)\wh D)<-\delta_\epsilon\text{ and } \wh P(\psi-(d^*+\epsilon)\wh D)>\delta_\epsilon.
\end{equation}
The property of \eqref{openEqForDim} is open in the $C^1$-topology of $\{f_\omega\}_{\omega\in \Sigma}$ and in the $C^0$-topology of $\psi$. Convexity implies that the pressure function cannot become negative again beyond $d^*+\epsilon$. Thus the new similarity dimension must also lie between $d^*-\epsilon $ and $d^*+\epsilon $  for all small enough perturbations.
\end{proof}

\begin{definition}[Effective expansion on average]\label{UEA2}
     We say that $(\Sigma\times M,\wh f)$ admits {\em effective expansion on average w.r.t. $\mu$} if
    \begin{equation*}\label{UEAEq}
%      \wh P(\psi- d\cdot \wh D)<0.
d^*>d-2.
    \end{equation*}
We may also specify and say that  $(\Sigma\times M,\wh f)$ admits {\em effective expansion on average with a parameter $\chi>0$} if $\inf_{\kappa\in (d-2,d^*)}\wh P(\psi- \kappa\cdot \wh D)<-\chi$.
\end{definition}

\begin{remark}[$C^1$-open]
	 When $f_\omega=f_{(\omega_i)_{i\leq0}}$, we are in the setting of \cite[Theorem~6.7]{TDFOE_I} (the variational principle), and so it is easy to observe that the effective expansion on average is a $C^1$-open condition in $\{f_\omega\}_{\omega\in \Sigma}$ and $C^0$-open in $\psi$ (recall Lemma \ref{ContSimDim}).

\end{remark}

\begin{remark}[Effective expansion on average from expansion on average in small parameter]
Note, when $d=2$, by \cite[Theorem~8.12]{TDFOE_I} every system which admits {\em expansion on average}:
$$\inf_{\omega^-\in \Sigma^-,(x,\xi)\in T^1M}\int\log  |d_xf_\omega\xi| d\mu_{\omega^-}\geq \chi>0$$
admits $\beta>0$ s.t. $\wh P(\psi-\beta\wh D)<0$, and so satisfies effective expansion on average. Such systems have many example (even conservative) and are extensively studied. 
\end{remark}

\begin{remark}[Effective expansion on average with $d^*=d$ from co-expansion on average] In \textsection \ref{coExImpEffEx} we show that every conservative systems which satisfies co-expansion on average ($C^1$-open condition suggested in \cite{DeWittDolgopyat2}) is effectively expanding on average (in fact with an optimal rate $d^*=d$). This immediately provides many families where effective expansion holds.
\end{remark}

\begin{remark}[Effective expansion with $d^*<d$]
	To see examples which are effectively expanding but do not satisfy $d^*=d$, consider a random system on a $2$-dimensional torus which is expanding on average (thus it is effectively expanding on average), and which admits an attractor s.t. the volume contracts exponentially fast to some measure on the attractor almost surely, and such that the stationary measure is of dimension strictly less than $2$. By \textsection \ref{appliC2} (specifically Corollary \ref{CorHausDim}), $d^*<d$. %Thus by \textsection \ref{coExImpEffEx}, the system does not admit co-expansion on average.
\end{remark}

\begin{remark}[Families of effectively expanding systems with very dissipative elements and a Gibbs process]
	One can easily generate a family of concrete examples which are dissipative, in dimension greater than $2$ (including highly dissipative diffeomorphisms which are not perturbations of conservative). Example: Consider any two maps $f_0,f_1$ composed w.r.t. to the $(\frac{1}{2},\frac{1}{2})$-Bernoulli measure on $\{0,1\}^{\mathbb{N}}$ which are effectively expanding on average (for example you may start from conservative maps which are co-expanding). The corresponding potential is $\psi=-\log 2\mathbb{1}_{[0]}+ -\log 2\mathbb{1}_{[1]}$. We may now also view a potential on $\{0,1,2\}^{\mathbb{N}}$ with $f_2$ as dissipative as we like. We still have $\wh P(\psi_N-\kappa \wh D')<0$ for $\kappa\in (d-2,d)$, where $\wh D'(\omega,x,\xi)=\wh D(\omega,x,\xi)$ for $\omega\in [0],[1]$ and otherwise  $\wh D'(\omega,x,\xi)=-\log|d_xf_2^{-1}\xi|$, and $\psi_N= -\log 2\mathbb{1}_{[0]}+ -\log 2\mathbb{1}_{[1]}-N\cdot \mathbb{1}_{[2]}$ with $N$ sufficiently large as $\wh P _{f_0,f_1,f_2}(\psi_N-\kappa \wh D')\xrightarrow[N\to\infty]{} \wh P_{f_0,f_1}(\psi-\kappa \wh D)<0$ (note also $P_{0,1,2}(\psi_N)\xrightarrow[N\to\infty]{}P_{0,1}(\psi)=0$). Finally, after finding sufficiently large $N$, perturb also $\psi_N$ in $C^0$-topology amongst H\"older potentials. This yields an effectively expanding on average system, whose randomness is driven by a Gibbs measure, and which contains very dissipative elements (not perturbations of conservative maps).
\end{remark}

\subsection{Dual Sobolev Space}\label{SoboSect}

We wish to work with a space of functions on $M$ (recall $\mathrm{dim}M=d\geq2)$. In this section we introduce the space of functions and the sequence of operators which act on it.

\begin{definition}[Sobolev space]
	Given $s\in (0,1)$, the inner-product of two function $g,h\in L^2(m)$ is defined by
\begin{align*}
\langle  g, h\rangle_s &:= \int  g \overline{ h}dm+\int \int \frac{(g(x)- g(y))\overline{( h(x)-h(y))}}{d(x,y)^{d+2s}}dm(x)dm(y).
\end{align*}

The {\em Sobolev space of index $s$}, $\Hh_{s}(m)$, is defined as the space of all functions in $L^2(m)$ whose $\langle \cdot,\cdot\rangle_{s}$-induced norm is finite:
\begin{align*}
&\|\wh g\|_{\Hh_s( m)}:=\sqrt{\int | g|^2d m+\int\int  \frac{| g(x)- g(y)|^2}{d(x,y)^{d+2s}} dm(x)dm(y)}.
\end{align*}
\end{definition}

\begin{remark}\label{rmkRangeS}
	The constant $s$ has to be taken within the range $(0,1)$ for the following reason: If $s\geq 1$, it is a classical fact that then $H_s(m)$ contains only the constant functions. In addition, we would also like that the unit ball of $\Hh_s(m)$ embeds compactly into $L^2(m)$ for usefulness of the definition and later on quasi-compactness. This property only holds when $s>0$. These restrictions on the range of $s$ impose restrictions on the range of $-2s$ in the definition of dual Sobolev space which we study later as well. 
\end{remark}

\begin{definition}[Dual Sobolev space]
	Given $s\in (0,1)$, we associate the inner-product for two functions on $M$:
\begin{align*}
\langle  g, h\rangle_{-s} &:= \int \int \frac{g(x)\cdot \overline{h(y)}}{d(x,y)^{d-2s}}dm(x)dm(y).
\end{align*}

For any function $g\in L^2(m)$ we have, 
\begin{equation} \label{eqNice}
	\| g\|_{-s}^2:= \langle  g, g\rangle_{-s} =\int \int \frac{g(x)\cdot \overline{g(y)}}{d(x,y)^{d-2s}}dm(x)dm(y).
\end{equation}

The {\em dual Sobolev space of index $-s$}, $\Hh_{-s}(m)$, is defined as the metric completion of $L^2(m)$ w.r.t. $\|\cdot\|_{-s}$.
\end{definition}

\begin{remark}\text{ }
	\begin{enumerate}
	\item 	Note, whenever $s\in (0,1)$, we have $|g(x)|\cdot|g(y)|\leq \frac{1}{2}|g(x)|^2+ \frac{1}{2}|g(y)|^2$, and so the integral in \eqref{eqNice} is bounded by 
	$$\int |g(x)|^2\int \frac{1}{d(x,y)^{d-2s}}dm(y)dm(x),$$
	where $\int \frac{1}{d(x,y)^{d-2s}}dm(y) $ is bounded by a global constant.
	\item To see that $\| g\|_{-s} $ indeed defines a norm, we need to observe that $\langle  g, h\rangle_{-s} $ defines a positive bi-linear form. We note that $\langle  g, h\rangle_{-s}=\int Ag  \overline{h} dm$ where $(Ag)(y)=\int \frac{g(x)}{d(x,y)^{d-2s}}dm(x)$. This is a compact and self-adjoint kernel-operator, where the kernel is classically known to be the principal singularity of the inverse fractional Laplace-Beltrami operator $(I-\Delta_{\mathfrak g})^{-s}$ (where $\mathfrak g$ is the Riemannian metric of $M$). Since it is a positive operator, we get that all eigen-values of $A$ are positive, and so $\langle  \cdot , \cdot \rangle_{-s} $ is a positive form.
	\item We get the relationship $\Hh_{-s}(m)\subseteq L^2(m)\subseteq \Hh_s(m)$.
	\item Recall Remark \ref{rmkRangeS}. We focus on $\kappa=d-2s\in (d-2,d)$, which is the source for the condition of effective expansion on average (recall Definition \ref{UEA2}) to demand $d^*>d-2$ (where we know that $d^*\leq d$ by Lemma \ref{dStarDBound}).
	\item The Deny inequality says that for any $g,h\in L^2(m)$,
\begin{equation}\label{Deny}
	\Big|\int_M g(x) \overline{h(x)} dm(x) \Big|\leq C_{\mathrm{Deny}}\|g\|_{s} \|h\|_{-s}.
\end{equation}
Indeed, since $L^2(m)$ is dense in $\Hh_{-s}(m)$ (by definition), the inequality extends to elements $h\in \Hh_{-s}(m)$, where the functional action against $\Hh_{s}(m)$ is extended as in \eqref{Deny} via an $L^2(m)$-inner-product. %Note that the assumption $\int h dm=0 $ extends to general elements $h\in\Hh_s(m)$ by $h(1)=0$ (with $1\in \Hh_s(m)$).
\end{enumerate}
\end{remark}

\subsection{Averaged semi-Ruelle Operator and averaged Koopman operator%, and averaged weighted 
}\label{ASROSect}

Up until now we studied the semi-Ruelle operator, which acts on functions on $\Sigma^+\times M$, which was the dual of the Koopman operator $\cdot\circ \wh F_L$. We now study define the corresponding averaged versions, which act on functions on $M$, and act as duals.

\subsubsection{Averaged semi-Ruelle operator}

\begin{definition}[Random sequence of operators on the Sobolev space]
	Given $\omega^+\in \Sigma^+$ the {\em  random sequence of operators} acting on $g\in \Hh_s(m)$, for all $n\in \mathbb{N}$, is defined by 
	$$(\mathcal L_{\omega^+,n}g)(x):=(\L_{\wh \phi}^n \wh g)_{\omega^+}(x),$$
	where $\wh g(\omega^+,x):=g(x)$, and for any $\wh h:\Sigma^+\times M\to \mathbb R$, $(\wh h)_{\omega^+}(x):=\wh h(\omega^+,x)$.
\end{definition}

\begin{remark}
	It is easy to check that $$\int \|\mathcal L_{\omega^+,n} g\|_{\Hh_s(m)}^2 d\mu(\omega^+)= \|\L_{\wh \phi}^n \wh g\|_{\HH_{d+2s}(\wh m)}^2.$$
\end{remark}

In Definition \ref{ASO} we define an averaged version of (powers of) the semi-Ruelle operator, which acts on the Sobolev space.

\begin{definition}[Averaged semi-Ruelle operator]\label{ASO}
Given $n\in\mathbb{N}$	The {\em averaged ($n$-th power of the) semi-Ruelle operator} acting on $g\in L^2(m)$, is defined by 
	$$(\mathcal L_{n}g)(x):=\Big(\int \L_{\wh\phi}^nd\mu\Big)(g)(x)=\int (\mathcal{L}_{\omega^+,n}g)(x)d\mu(\omega^+).$$
	We denote $\mathcal L_{1}$ by simply $\mathcal{L}$ for short.
\end{definition}

\begin{comment}
\begin{lemma}[Averaged operator norm inequality]\label{AONI}
    For all $g\in L^2(m)$, for all $n\geq0$,
    $$\|\mathcal L_{n}g\|_{-s}^2\leq \int \|\mathcal{L}_{\omega^+,n}g\|_{-s}^2 d\mu=\|\|\mathcal{L}_{\omega^+,n}g\|_{-s} \|_{L^2(\mu)}^2.$$
\end{lemma}
\begin{proof} By Jensen's inequality, 
  \begin{align*}
      \|\mathcal L_{n}g\|_{-s}^2=&\int |\int \mathcal{L}_{\omega^+,n}gd\mu|^2  dm\\
      +&\int\int \frac{|\int (\mathcal{L}_{\omega^+,n}g)(x)-(\mathcal{L}_{\omega^+,n}g)(y) d\mu|^2}{d(x,y)^{d+2s}} dm(x)dm(y)\\
      \leq& \int \int |\mathcal{L}_{\omega^+,n}g|^2d\mu  dm\\
       +&\int\int \int \frac{|(\mathcal{L}_{\omega^+,n}g)(x)-(\mathcal{L}_{\omega^+,n}g)(y) |^2}{d(x,y)^{d+2s}}  dm(x)dm(y)d\mu(\omega^+)\\
       =&\int \Big(\int |\mathcal{L}_{\omega^+,n}g|^2d\mu  dm\\
       +&\int\int \int \frac{|(\mathcal{L}_{\omega^+,n}g)(x)-(\mathcal{L}_{\omega^+,n}g)(y) |^2}{d(x,y)^{d+2s}}  dm(x)dm(y)\Big)d\mu(\omega^+)\\
       =&\int \|\mathcal{L}_{\omega^+,n}g\|_{\Hh_s(m)}^2 d\mu.
  \end{align*}  
\end{proof}
\end{comment}

\begin{lemma}[i.i.d. averaged powers]\label{AvgPwr}
	Assume that $\Sigma$ is a full-shift, and that $\mu$ is a Bernoulli measure and $f_\omega=f_{\omega_0}$. Then for all $g\in L^2(m)$, for all $n\geq1$, for all $\omega^+\in \Sigma^+$, for $m$-a.e. $x\in M$,
	$$(\mathcal L_1^ng)(x)=(\mathcal L_ng)(x)=(\L_{\wh\phi}^ng)(\omega^+,x).$$
\end{lemma}
\begin{proof}
	It is enough to prove the right-most equality for all $\omega^+\in \Sigma^+$, for all $g\in C(M)$, for all $x\in M$. Indeed, 
	\begin{align}\label{SatSat}
		(\L_{\wh\phi}^ng)(\omega^+,x)=&\sum_{|\ul w|=n}e^{\psi^{(n)}(\ul w)}\cdot e^{\wh\phi^{(n)}(\ul w, (f_{\ul w}^n)^{-1}(x))}g((f_{\ul w}^n)^{-1}(x)),%%\\
%=&\int e^{\wh\phi^{(n)}(\sigma^+, (f_{\sigma^+}^n)^{-1}(x))}g((f_{\sigma^+}^n)^{-1}(x))d\mu(\sigma^+).
	\end{align}
	where we use the notation $\psi^{(n)}(\ul w)$, $\wh\phi^{(n)}(\ul w,\cdot)$, $f_{\ul w}^n$ to imply that the expression is independent of the choice of the word $\sigma^+\in[\ul w]$. Indeed, the r.h.s. of \eqref{SatSat} is independent of $\omega^+$, and so when integrating both sides w.r.t. $\omega^+$, the r.h.s. does not change while the l.h.s. equals to $(\mathcal L_n g)(x)$ by definition. 
\end{proof}

\begin{remark}\label{RmkAvgPwr}
	In the setting of Lemma \ref{AvgPwr}, the operator $\mathcal L:\Hh_{-s}(m)\to \Hh_{-s}(m) $ is well-defined, where 
	$$\mathcal L:= \mathcal L_1\text{ and } \mathcal L^n= \mathcal L_n.$$
\end{remark}

\begin{comment}
\subsection{Averaged Koopman operator}

\begin{definition}[Koopman operator]
	Given $\wh g\in L^2(\wh m)$, the {\em Koopman operator} acting on it is defined by
	$$(\U \wh g)(\omega^+,x)=\wh g\circ \wh F_L(\omega^+,x).$$
\end{definition}

\begin{definition}[Random sequence of operators on the dual Sobolev space]
	Given $\omega^+\in \Sigma^+$ the {\em  random sequence of operators} acting on $g\in L^2(m)$, for all $n\in \mathbb{N}$, is defined by 
	$$(\mathcal U_{\omega^+,n}g)(x):=(\U^n \wh g)_{\omega^+}(x),$$
	where $\wh g(\omega^+,x):=g(x)$, and for any $\wh h:\Sigma^+\times M\to \mathbb R$, $(\wh h)_{\omega^+}(x):=\wh h(\omega^+,x)$.
\end{definition}

\begin{definition}[Averaged Koopman operator]\label{AKO}
Given $n\in\mathbb{N}$	The {\em averaged ($n$-th power of the) Koopman operator} acting on $g\in L^2(m)$, is defined by 
	$$(\mathcal U_{n}g)(x):=\int (\mathcal{U}_{\omega^+,n}g)(x)d\mu(\omega^+).$$
	We denote $\mathcal U_{1}$ by simply $\mathcal{U}$ for short.
\end{definition}

\begin{remark}
	The action of $\mathcal{U}$ extends to $\Hh_{-s}(m)$, as one can check that $\mathcal U$ is the $L^2(m)$-dual of $\mathcal L$, and $\mathcal L$ is a bounded linear operator on $\Hh_{s}(m)$ (recall \eqref{Deny}).
\end{remark}
\end{comment}

\subsubsection{Averaged %weighted 
Koopman operator}

%While the duality of $\mathcal L$ and $\mathcal U$ holds, we wish to use a weighted Koopman operator whose action we study later on.

\begin{definition}[Koopman operator]
	Given $\wh g\in L^2(\wh m)$, the {\em Koopman operator} acting on it is defined by
	$$(\wh \Uu\wh g)(\omega^+,x)=\wh g\circ \wh F_L(\omega^+,x).$$
\end{definition}

\begin{definition}[Averaged Koopman operator]\label{AWKO}
Given $n\in\mathbb{N}$	The {\em averaged ($n$-th power of the) Koopman operator} acting on $g\in L^2(m)$, is defined by 
	$$(\Uu_{n}g)(x):= \int 
	(\wh \Uu^n g)(\omega^+,x)d\mu(\omega^+)=\int 
	g\circ f_{\omega^+}^n(x)d\mu(\omega^+).$$
	We denote $\Uu_{1}$ by simply $\Uu$ for short.

%Similarly, we define
%$$(\Ll_{n}g)(x):=\int (\L_\psi^n g)(\omega^+,x)d\mu(\omega^+),$$
% and denote $\Uu_{1}$ by simply $\Uu$ for short.
\end{definition}

\begin{remark}\label{WeightAvgPwrs}%\text{ }
%\begin{enumerate}
%	\item 	It is easy to check that $\Uu$ is the $L^2(m)$-dual of $\Ll$, as 
%	\begin{align*}
%		\int g \Uu h  dm=&\int \int g e^{-\wh \varphi} h\circ \wh F_Ld\wh m= \int \int \L_{\wh \phi}(g e^{-\wh \varphi}) hd\wh m\\
%		 =& \int \int \L_{\psi}(g)hd\wh m =\int g \Ll h dm.
%	\end{align*}
%	\item 
As in Lemma \ref{AvgPwr}, if $\Sigma$ is a full-shift, $\mu$ is a Bernoulli measure, and $f_\omega=f_{\omega_0}$, then 
	$$\Uu_n=\Uu^n%\text{ and } \Ll_n=\Ll^n
	.$$
%\end{enumerate}
\end{remark}

\begin{lemma}[Averaged Koopman as dual]
	For all $g,h\in L^2(m)$,
	$$\int \mathcal{L} g h dm=\int g \Uu h dm.$$
\end{lemma}
\begin{proof}
Since $\mu$ is $\psi$-conformal,
	\begin{align*}
		\int \Ll g h dm=&\int \sum_{T_L\sigma^+=\omega^+}\int e^{\wh \varphi(\sigma^+,f_{\sigma^+}^{-1}(x))}g(f_{\sigma^+}^{-1}(x))h(x)dm(x)d\mu(\omega^+)\\
		=& \int \sum_{T_L\sigma^+=\omega^+}\int g(x)h(f_{\sigma^+}(x))dm(x)d\mu(\omega^+)\\
		=& \int \int g(x)h(f_{\omega^+}(x))dm(x)d\mu(\omega^+)= \int g \Uu h dm.
	\end{align*}
\end{proof}

\begin{remark}
	For conservative systems, $\Uu_{\wh F}=\Ll_{\wh F^{-1}}$. This property can become useful to study the action of an operator on its dual space, and for example $\Uu$ can admit a spectral gap on both $\Hh_{-s}(m)$ and  $\Hh_{s}(m)$ to deduce an $L^2(m)$-spectral gap.
\end{remark}

\subsection{Lasota-Yorke and Quasi-compactness}

\begin{theorem}[Lasota-Yorke inequality]\label{FinalLYthm}
Let $\mu$ be a Gibbs state in a proper form, and assume that $(\Sigma\times M,\wh f)$ admits effective expansion on average with a parameter $\chi>0$. Denote by $m$ the normalized Riemannian volume of $M$, where $\mathrm{dim}M=d$. Assume that $f_\omega=f_{(\omega_i)_{i\geq0}}$ is a family of $C^{1+\theta}$ diffeomorphisms of $M$. Then there exists $s\in (0,1)$ s.t. for all $\epsilon >0$ small enough, for all $g\in \Hh_{-s}(m)$, for all $n\geq 0$, 
\begin{align*}
\|\Ll_n g\|_{-s}\leq C_\epsilon D_f^n\|g\|_{-s-\epsilon}+ Ce^{-\frac{\chi}{2} n}\|g\|_{-s},
\end{align*}
where $D_f>0$ is a constant depending only on $f$, $C_\epsilon>0$ is a constant depending on $\epsilon>0$, and $C>0$ is a constant depending on $s$ (and $f$). 
\end{theorem}
\begin{proof}
Let $n\in\mathbb{N}$, and let $h\in \Hh_s(m)$ with $\|h\|_s=1$. Then by the Deny inequality (recall \eqref{Deny}),
\begin{align}\label{eqReducToNegS}
	&\Big|\int \Ll_n g \overline{h} dm\Big|\nonumber\\
	=&\Big|\int \sum_{T_L^n\sigma^+=\omega^+}e^{\psi^{(n)}(\sigma^+)}\int e^{\wh \varphi(\sigma^+, (f_{\sigma^+}^n)^{-1}(x))}g(( f_{\sigma^+}^n)^{-1}(x)) \overline{h}(x) dmd\mu\Big|\nonumber\\
	\leq& \int \sum_{T_L^n\sigma^+=\omega^+}e^{\psi^{(n)}(\sigma^+)} \Big|\int e^{\wh \varphi(\sigma^+, (f_{\sigma^+}^n)^{-1}(x))}g(( f_{\sigma^+}^n)^{-1}(x)) \overline{h}(x) dm \Big| d\mu\nonumber\\
	\leq& C_\mathrm{Deny}\int \sum_{T_L^n\sigma^+=\omega^+}e^{\psi^{(n)}(\sigma^+)}\|h\|_s\cdot \|e^{\wh \varphi(\sigma^+, (f_{\sigma^+}^n)^{-1}(x))}g(( f_{\sigma^+}^n)^{-1}(x))\|_{-s}d\mu\nonumber\\
	=& C_\mathrm{Deny}\int \|e^{\wh \varphi(\omega^+, (f_{\omega^+}^n)^{-1}(x))}g(( f_{\omega^+}^n)^{-1}(x))\|_{-s}d\mu\nonumber\\
	\leq& C_\mathrm{Deny} \sqrt{\int \|e^{\wh \varphi(\omega^+, (f_{\omega^+}^n)^{-1}(x))}g(( f_{\omega^+}^n)^{-1}(x))\|_{-s}^2d\mu}.
\end{align}

We let $\epsilon\in (0,1)$ sufficiently small so $d-2s-2\epsilon>d-2$, where $s\in (0,1)$ is chosen so $\wh P(\psi-(d-2s)\wh D)<-\chi$ (by the effective expansion on average). Then we continue to bound the $\|e^{\wh \varphi(\omega^+, (f_{\omega^+}^n)^{-1}(x))}g(( f_{\omega^+}^n)^{-1}(x))\|_{-s}^2 $ term in the integrand of \eqref{eqReducToNegS} for all $\omega^+$ by a change of variables:
\begin{align}\label{Sat28}
	&\|e^{\wh \varphi(\omega^+, (f_{\omega^+}^n)^{-1}(x))}g(( f_{\omega^+}^n)^{-1}(x))\|_{-s}^2\nonumber\\
	 =&\int \int \frac{e^{\wh \varphi(\omega^+, (f_{\omega^+}^n)^{-1}(x))}g(( f_{\omega^+}^n)^{-1}(x)) \overline{e^{\wh \varphi(\omega^+, (f_{\omega^+}^n)^{-1}(y))}g(( f_{\omega^+}^n)^{-1}(y))}}{d(x,y)^{d-2s}}dmdm\nonumber\\
	 =& \int \int \frac{g(x) \overline{g( y)}}{d(f_{\omega^+}^n (x), f_{\omega^+}^n (y))^{d-2s}}dmdm\nonumber\\
	 =& \int \int \frac{g(x) \overline{g( y)}}{d(x,y)^{d-2s}}\Big(\frac{d(x,y)}{d(f_{\omega^+}^n (x), f_{\omega^+}^n (y))}\Big) ^{d-2s} dmdm.
\end{align}

We now integrate \eqref{Sat28}, to plug back in \eqref{eqReducToNegS}, and get 
\begin{equation}\label{Sat29}
\Big|\int \Ll_n g \overline{h} dm\Big|\leq C_\mathrm{Deny}\sqrt{ \int \int g(x) \overline{g( y)}\frac{Q_n(x,y)}{d(x,y)^{d-2s}}dm dm},
\end{equation}
where $$Q_n(x,y):=\int \Big(\frac{d(x,y)}{d(f_{\omega^+}^n (x), f_{\omega^+}^n (y))}\Big) ^{d-2s} d\mu.$$

In order to estimate the r.h.s. of \eqref{Sat29}, we take two different approaches in two different case. For a general function $g\in L^2(m)$, the integrand may not be positive, and so a naive bound via point-wise bounds for $Q_n(x,y)$ fails; In that case we have to take the more subtle approach of studying the (pseudodifferential) kernel operator
$$(T_n g)(x):=\int g(y) \frac{Q_n(x,y)}{d(x,y)^{d-2s}}dm(y),$$
where
$$\int \int g(x) \overline{g( y)}\frac{Q_n(x,y)}{d(x,y)^{d-2s}}dm dm=\langle g ,T_ng \rangle_{L^2(m)}.$$
The specific form of the operator $T_n$ allows us to prove estimates that are not true in general from a mere point-wise integrand bound.

\medskip
\noindent For simplicity, we present first a naive proof in the case where $g\geq 0$:

Assume that $g\geq 0$. We break the domain of integration in \eqref{Sat28} into $[d(x,y)\geq \epsilon^2 M_f^{-2n}]$ and its complement, denoting the two terms by I and II respectively, where $M_f:=\max_{\omega,x}\{\|d_xf_\omega\|, \|d_xf_\omega^{-1}\|\}$. To bound I:
\begin{align}\label{toBoundI}
	\mathrm{I}\leq& (M_f^n)^{d-2s}(\frac{1}{\epsilon^2}M_f^{3n})^{2\epsilon} \int \int \frac{g(x) \overline{g( y)}}{d(x,y)^{d-2s-2\epsilon}}dmdm\nonumber\\
	= &(M_f^n)^{d-2s}(\frac{1}{\epsilon^2}M_f^{3n})^{2\epsilon} \int \int \frac{g(x) \overline{g( y)}}{d(x,y)^{d-2s-2\epsilon}}dmdm\nonumber\\
	=& (M_f^n)^{d-2s}(\frac{1}{\epsilon^2}M_f^{3n})^{2\epsilon} \|g\|^2_{-s-\epsilon}.
\end{align}
We continue to bound II, by specifying further that we choose $\epsilon>0$ sufficiently small so if $0<d(x,y)<\epsilon^2 M_f^{-2n}$ then $|d_xf_{\omega^+}^n\frac{\xi_{xy}}{| \xi_{xy} |}|^{-1}=e^{\pm \epsilon}\frac{d(x,y)}{d(f_{\omega^+}^n (x), f_{\omega^+}^n (y))}$ for all $n\geq 0$, where $\xi_{xy}:=\exp_x^{-1}(y)$:
\begin{align}\label{toBoundII}
\mathrm{II}\leq& \int \int_{[0<d(x,y)<\epsilon^2 M_f^{-2n}]} \frac{g(x)\overline{g(y)}}{d(x,y)^{d-2s}}\Big(\frac{d(x,y)}{d(f_{\omega^+}^n(x), f_{\omega^+}^n(y))}\Big)^{d-2s}dm dm\nonumber\\
\leq& e^\epsilon\int \int \frac{g(x)\overline{g(y)}}{d(x,y)^{d-2s}}\frac{1}{|d_xf_{\omega^+}^n\frac{\xi_{xy}}{| \xi_{xy} |}|^{d-2s}}dm dm.
\end{align}

Plugging back the estimates of \eqref{toBoundI} and \eqref{toBoundII} into \eqref{eqReducToNegS}, we get
\begin{align}\label{Sat34}
	&\Big|\int \Ll_n g \overline{h} dm\Big|\nonumber\\
	\leq&C_\mathrm{Deny}\sqrt{\frac{1}{\epsilon^2}M_f^{4nd} \|g\|^2_{-s-\epsilon}+ e^\epsilon\int \int \int \frac{g(x)\overline{g(y)}}{d(x,y)^{d-2s}}\frac{1}{|d_xf_{\omega^+}^n\frac{\xi_{xy}}{| \xi_{xy} |}|^{d-2s}}dm dmd\mu }\nonumber\\
	=& C_\mathrm{Deny}\sqrt{\frac{1}{\epsilon^2}M_f^{4nd} \|g\|^2_{-s-\epsilon}+ e^\epsilon\int  \int \frac{g(x)\overline{g(y)}}{d(x,y)^{d-2s}}\int \frac{1}{|d_xf_{\omega^+}^n\frac{\xi_{xy}}{| \xi_{xy} |}|^{d-2s}}d\mu dm dm }\nonumber\\
	\leq& C_\mathrm{Deny}\sqrt{\frac{1}{\epsilon^2}M_f^{4nd} \|g\|^2_{-s-\epsilon}+ e^\epsilon\|g\|_{-s}^2Ce^{-\chi n} },
\end{align}
where $C$ is given by the effective expansion on average convergence rate (for $s$). Since \eqref{Sat34} holds for all $h\in \Hh_s(m)$ with $\|h\|_s=1$, we get that 
\begin{align*}
	\|\Ll_n g\|_{-s}\leq& C_\mathrm{Deny}\sqrt{\frac{1}{\epsilon^2}M_f^{4nd} \|g\|^2_{-s-\epsilon}+ e^\epsilon\|g\|_{-s}^2Ce^{-\chi n} }\\
	\leq& C_\mathrm{Deny}\frac{1}{\epsilon}M_f^{2nd} \|g\|^2_{-s-\epsilon}+ e^\epsilon\|g\|_{-s}\sqrt Ce^{-\frac{\chi}{2} n} .
\end{align*}

This concludes the proof in the case where $g\geq 0$. We continue to treat the general case where $g\in L^2(m)$ is not necessarily non-negative (nor real). This part of the proof is disjoint from the rest of the objects in this paper, and appears in \textsection \ref{appPseudoDiffKerOps}.
\end{proof}

\begin{cor}[Averaged semi-Ruelle is quasi-compact]\label{QCasr}
For all $n\in\mathbb{N}$ large enough, 
$$\mathcal L_n=K_n^*+R_n^*,$$
where the essential radium of $R_n^*:\Hh_{-s}(m)\to\Hh_{-s}(m)$ is bounded by $e^{-\frac{\chi}{2}+\frac{\epsilon}{n}}$, $K_n^*$ is compact (finite-rank), and $K_n^*R_n^*=R_n^* K_n^*=0$.
\end{cor}
\begin{proof}
By Theorem \ref{FinalLYthm}, there exists $B>1$ and $s>0$ s.t. for all $\epsilon>0$ small enough (with $d-2s-2\epsilon >d-2$), for all $g\in L^2(m)$, for all $n\geq 0$,
\begin{equation}\label{forQC}
    \|\mathcal{L}_ng\|_{-s}\leq B^{n+1}\|g\|_{-s-\epsilon}+e^\epsilon e^{-\frac{\chi}{2}n}\|g\|_{-s}.
\end{equation}

By the Rellich-Kondrachov Theorem, as $s,s+\epsilon\in (0,1)$, the embedding $i: \Hh_{-s}(m) \hookrightarrow \Hh_{-s-\epsilon}(m) $ is compact. By \eqref{forQC}, Hennion's theorem implies that $\mathcal L_n:\Hh_{-s}(m)\to \Hh_{-s}(m) $ is quasi-compact and can be written as 
$$\mathcal L_n=K_n^*+R_n^*,$$
where the essential radium of $R_n^*:\Hh_{-s}(m)\to\Hh_{-s}(m)$ is bounded by $e^{-\frac{\chi}{2}+\frac{\epsilon}{n}}$, $K_n^*$ is compact (finite-rank), and $K_n^*R_n^*=R_n^* K_n^*=0$.
\end{proof}

\begin{lemma}[Norm bound] For all $n\geq 0$,
	$$\|\mathcal L_n\|_{\Hh_{-s}(m)\to \Hh_{-s}(m)}\geq\frac{1}{\|1\|_{-s}}=\frac{1}{\sqrt{\int \int \frac{1}{d(x,y)^{d-2s}}dm(x)dm(y)}}.$$
	Therefore, in particular, the spectral radius of the sequence $\L_n$ is at least $1$. 
\end{lemma}
\begin{proof} For $\mu$-a.e. $\omega^+$, 
\begin{align}\label{forNormBoundBelow}
\int (\L_{\wh\phi}^n 1)_{\omega^+}dm=&L_\psi^n\int e^{\wh\varphi^{(n)}(\omega^+,x)}dm(x)= (L_\psi^n1)(\omega^+) =1.
\end{align}
Then, since $\| 1\|_{\Hh_{-s}(m)}=1$, by Jensen's inequality, and by \eqref{forNormBoundBelow},
	\begin{align*}
		\|\mathcal L_n\|_{\Hh_{-s}(m)\to \Hh_{-s}(m)}^2\geq& \frac{\|\mathcal L_n 1\|_{-s}^2}{\|1\|_{-s}}\geq \frac{1}{\|1\|_{-s}}\int\Big|\int (\L_{\wh\phi}^n 1)_{\omega^+}d\mu \Big|^2dm\\
		\geq & \frac{1}{\|1\|_{-s}}\Big(\int\int (\L_{\wh\phi}^n 1)_{\omega^+}d\mu dm \Big)^2\\
		=& \frac{1}{\|1\|_{-s}}\Big(\int\int (\L_{\wh\phi}^n 1)_{\omega^+}dm d\mu \Big)^2= \frac{1}{\|1\|_{-s}}.
	\end{align*}
\end{proof}

\begin{cor}[Sequential quasi-compactness]\label{SeqQuasCom}
For all $n\geq0 $, there exists a decomposition $\mathcal{L}_n = K_n^* + R_n^*$ where $K_n^* : \Hh_{-s}(m) \to \Hh_{-s}(m)$ is a finite-rank (and thus compact) operator, and $R_n^* :\Hh_{-s}(m) \to\Hh_{-s}(m)$ satisfies 
$$\|R_n^*\|_{\Hh_{-s}(m) \to\Hh_{-s}(m)} \leq 2B e^{-\frac{\chi}{2} n}.$$
\end{cor}
\begin{proof}
By the Rellich-Kondrachov Theorem, as $s,s+\epsilon \in (0,1)$, the embedding $i: \Hh_{-s}(m) \hookrightarrow \Hh_{-s-\epsilon}(m) $ is compact.

Since $\Hh_{-s}(m)$ is a Hilbert space, the compact embedding $i$ can be approximated in the operator norm by finite-rank operators. Specifically, we can choose a sequence of orthogonal projections $\{P_N\}_{N=1}^\infty$ onto $N$-dimensional subspaces of $\Hh_{-s}(m)$ such that the approximation error decays to zero:

\begin{equation}\label{approxAbove}
    \lim_{N \to \infty} \|I - P_N\|_{\Hh_{-s}(m) \to \Hh_{-s-\epsilon}(m)} = 0.
\end{equation}
For each fixed $n \geq 0$, let $\epsilon_n = B^{-n} e^{-\frac{\chi}{2} n}$. By \eqref{approxAbove} there exists a sufficiently large integer $N_n$ s.t. for all $g \in \Hh_{-s}(m)$,
$$\|(I - P_{N_n})g\|_{-s-\epsilon} \leq \epsilon_n \|g\|_{-s} = B^{-n} e^{-\frac{\chi}{2} n} \|g\|_{-s}.$$
We now define the components of our operator:
\begin{align*}
    K_n^* &:= \mathcal{L}_n P_{N_n}, \\
    R_n^* &:= \mathcal{L}_n (I - P_{N_n}).
\end{align*}
Because $P_{N_n}$ projects onto a finite-dimensional subspace of rank $N_n$, $K_n^*$ is a finite-rank operator.

We now bound the remainder $R_n$. Let $g \in\Hh_{-s}(m)$. By \eqref{forQC},
\begin{align}\label{intoTheIneq}
    \|R_n^* g\|_{-s}= &\|\mathcal{L}_n (I - P_{N_n})g\|_{-s} \\
    \leq& B^{n+1} \|(I - P_{N_n})g\|_{-s-\epsilon} + Be^{-\frac{\chi \kappa }{2}n} \|(I - P_{N_n})g\|_{-s}.\nonumber
\end{align}

Using our choice of $N_n$ to bound the $\Hh_{-s-\epsilon}$-norm, we have:
$$B^{n+1} \|(I - P_{N_n})g\|_{-s-\epsilon}\leq B^{n+1} \left( B^{-n} e^{-\frac{\chi}{2} n} \|g\|_{-s} \right) = Be^{-\frac{\chi }{2} n} \|g\|_{-s}.$$
Furthermore, because $I - P_{N_n}$ is an orthogonal projection on the Hilbert space $\Hh_{-s}(m)$, its operator norm is bounded by $1$, meaning $\|(I - P_{N_n})g\|_{-s} \leq \|g\|_{-s}$. Substituting these into \eqref{intoTheIneq} yields
$$\|R_n g\|_{-s} \leq Be^{-\frac{\chi}{2} n} \|g\|_{-s} + Be^{-\frac{\chi}{2} n} \|g\|_{-s} = 2  Be^{-\frac{\chi}{2} n} \|g\|_{-s}.$$
Taking the supremum over all $g \in \Hh_{-s}(m)$ with $\|g\|_{-s} = 1$ yields the operator norm bound:
$$\|R_n\|_{\Hh_{-s}(m) \to \Hh_{-s}(m)} \leq 2B e^{-\frac{\chi}{2} n}.$$
\end{proof}

\begin{remark}
	It is clear that the sequential quasi-compactness of Corollary \ref{SeqQuasCom} becomes much more useful when the rank of the compact operators $K_n$ is uniformly bounded. Indeed, this property is most useful when for all $n$, the rank of $K_n$ equals $1$% (see for example Theorem \ref{ExpConv} below)
	. To obtain this property, one has to establish {\em ergodicity} of the random dynamics, which is an independent property, which often requires separate assumptions, and is proved geometrically rather than via spectral methods. Combining ergodicity together with the sequential quasi-compactness of Corollary \ref{SeqQuasCom} yields many powerful statistical properties for the random dynamics.
\end{remark}

\section{Applications}\label{appsSect}
In \cite{TDFOE_I} we prove a POE variational principle, under the assumption of $\wh\phi$-holonomies. In \textsection \ref{VarPrinceSubSect} we provide another proof for the POE variational principle, using conformal measures and harmonic functions%, under the stronger assumption of mild unstable holonomies (see \textsection \ref{HolsHierarcs})
. The advantage of the proof of the variational principle in this paper is that using the methods of this paper and the semi-Ruelle operator, we are able to show additional properties for the measures which obtain the maximum in the POE variational principle. For example, positive entropy out of equilibrium, as in \textsection \ref{appliA} below, and the semi-Gibbs estimates as in Proposition \ref{newSemiGibbsProp}%, and the equi-conformality property as in Theorem \ref{equiConf}
. We also characterize cases when our system admits a smooth random measure in \textsection \ref{appliB}. In addition, we study more subtle behavior (such as a spectral gap for the averaged semi-Ruelle operators) in \textsection \ref{appliC2}.

\subsection{Positive Entropy Out of Equilibrium for POE Maximizers}\label{appliA}

\begin{theorem}[Positive entropy out of equilibrium]\label{posEOE}
Assume that $(\Sigma\times X,\wh F)$ admits $\wh\phi$-holonomies, and that $F_\omega=F_{(\omega_i)_{i\leq0}}$%, and that $\wh\phi\in \mathrm{H\ddot{o}l}(\Sigma^-\times X)$
. Let $\wh\nu$ be an measure as in Lemma \ref{InvExists}. Then, if $(\Sigma,T)$ admits positive topological entropy, then
$\wh\nu$ admits positive entropy out of equilibrium.
%Then, POE maximizers admit positive entropy out of equilibrium.
\end{theorem}
In the proof of Theorem \ref{POEupperBound} we showed that $\wh\nu$ is indeed a maximizer of the POE variational principle.
\begin{proof}
Assume w.l.o.g. that $\wh P(\wh\phi)=0$. Then 
    $$\wh h_{\wh\nu}(\wh F)+\int \wh \phi d\wh \nu=0.$$

By Lemma \ref{ConfOfErg} we may assume further w.l.o.g. that $\wh \nu$ is ergodic ($\wh p$ being ergodic implies that $\wh \rho\cdot \wh p$ is ergodic), as the metric entropy out of equilibrium is affine w.r.t. the ergodic decomposition.

If $\wh h_{\wh\nu}(\wh F)=0$, then we get that 
$$\int \wh \phi d\wh \nu= \int -\wh \phi d\wh \nu =0.$$
By the semi-Gibbs estimates (Proposition \ref{newSemiGibbsProp}), we have that for $\widehat\nu$-a.e. $(\omega,t)$,
\begin{equation}\label{forPosEOEeq}
	\nu_{(\omega^-,t)}([\omega]_n\times X)=e^{\pm C_{\wh\rho}} e^{\wh \phi^{(n)}(\omega,t)}.
\end{equation}

However, by the measurable version of the Gottschalk-Hedlund theorem (see \cite[Corollary~2]{MeasurableGH}), either (a) $\wh\phi$ is cohomologous to $0$ via  a bounded coboundary function, or (b) for on a positive $\wh\nu$-measure,  $$\liminf_{n\to\infty} -\wh\phi^{(n)}=-\infty,\text{ that is, } \limsup_{n\to\infty} \wh\phi^{(n)}=\infty.$$
In case (a), we get that $0=\wh P(\wh \phi)=\wh h (\wh F)=h_\mathrm{top}(T)$ by Corollary \ref{rmkEOE}, which is a contradiction. In case (b), we get that from \eqref{forPosEOEeq} that 
for a positive-measure set of points, 
\begin{align*}
	1\geq& \limsup\nu_{(\omega^-,t)}([\omega]_n\times X) \\
	\geq &  \limsup e^{- C_{\wh\rho}} e^{\wh \phi^{(n)}(\omega,t)}= e^{- C_{\wh\rho}}  e^{\limsup\wh \phi^{(n)}(\omega,t)} =\infty,
\end{align*}
a contradiction! Thus $\wh h_{\wh\nu}(\wh F)>0$.
\end{proof}

\subsection{Smooth Random Measure}\label{appliB}

We work in the setting where $F_\omega=F_{(\omega_i)_{i\leq0}}$, $\psi$ is a potential in a proper form (Definition \ref{properFormPot}), where $\mu$ be its associated Gibbs state, and $m$ is a smooth measure on $X$ (recall Definition \ref{defLogJac}), with its associated log-Jacobian function $\wh\varphi\in \mathrm{H\ddot{o}l}(\Sigma^-\times X)$.

\begin{definition}[Smooth random measure]
	We say that $(\Sigma^-\times X,\wh F_R)$ and $\mu$ admit a {\em smooth random measure} if there exists an $\wh F_R$-invariant Borel probability measure $\wh \mu$, and $C>1$, such that
	$$C^{-1}\wh m\leq \wh \mu\leq C \wh m,$$
	where $\wh m:=\mu\times m$.
\end{definition}

%Recall Definition \ref{mildHOls} of mild unstable holonomies.
\begin{theorem}[Existence of smooth random measures]\label{smoothIFF}
	Assume that $m$ is fully-supported on $X$. Let $\wh\phi:=\psi+\wh\varphi$. Assume that $(\Sigma,\wh F)$ admits %mild unstable 
	$\wh\phi$-holonomies. Then,
	\begin{align*}
		&(\Sigma^-\times X,\wh F_R)\text{ and }\mu\text{ admit a smooth random measure} \\
		\Leftrightarrow &\wh P(\wh\phi)=0\text{ and }\sup_k\|\log \L_{\wh\phi}^k1\|_{L^\infty(\wh m)}<\infty.
	\end{align*}
\end{theorem}
\begin{proof}\text{ }

\noindent\underline{$\Rightarrow$:} Write $\wh m:=\mu\times m$. Assume that $\wh P(\wh \phi)=0$. By Theorem \ref{coolThm2} $\wh m$ is a $\wh\phi$-conformal measure. Since we assume %mild unstable 
$\wh\phi$-holonomies, by Theorem \ref{harmExist}, there exists a harmonic function $\wh \rho$ with $\log\wh \rho\in %\B_{m}^{\theta,\infty}
L^\infty(\wh m)$% for some $\theta\in (0,1]$
. By Lemma \ref{InvExists}, $\wh\mu:=\wh \rho\cdot \wh m$ is $\wh F_R$-invariant. Since $\log \wh\rho$ is essentially bounded, we get that $\wh \mu$ is a smooth random measure. 

\medskip
\noindent\underline{$\Leftarrow$:} Assume that we have a smooth measure $\wh\mu$, and write 
$$\wh\rho :=\frac{d\wh \mu}{d\wh m},$$ 
which is well-defined $\wh m$-a.e. $\mu$ is a Gibbs state, and so fully-supported on $\Sigma^-$. $m$ is fully-supported on $X$ by assumption, and so $\wh m$ is fully-supported on $\Sigma^-\times X$. We have that $\wh m$-a.e. 
$$C^{-1}\leq \wh \rho\leq C.$$ 
In addition, for all $\wh g\in C(\Sigma^-\times X)$,
\begin{align*}
	\int \wh g\L_{\wh\phi}\wh\rho d\wh m=& \int \L_{\wh\phi}(\wh g\circ \wh F_R\cdot \wh\rho) d\wh m = \int \wh g\circ \wh F_R\cdot \wh\rho d\wh m \\
	=&\int \wh g\circ \wh F_R d\wh \mu= \int \wh g d\wh \mu =\int \wh g \wh \rho d\wh m.
\end{align*}
Thus, 
$$\L_{\wh\phi}\wh\rho=\wh\rho \ \text{ }\wh m\text{-a.e.}$$
Then, for all $n\geq0$, and since $\wh m$ is fully-supported and $\L_{\wh\phi^n}1$ is continuous,
\begin{align*}
	\|\L_{\wh\phi}^n1\|_\infty= \|\L_{\wh\phi}^n1\|_{L^\infty(\wh m)}\leq \|\L_{\wh\phi}^n C\wh \rho\|_{L^\infty(\wh m)}= C\cdot  \|\wh \rho\|_{L^\infty(\wh m)}\leq C^2.
\end{align*}
Thus, by Lemma \ref{specRad}, $\wh P(\wh\phi)\leq 0$. Together with Corollary \ref{PBound}, we conclude $\wh P(\wh \phi)=0$. Finally, it follows that $\|\wh\rho\|_{L^\infty(\wh m)}<\infty$ where $\wh m$ is indeed $\wh \phi$-harmonic, and by Theorem \ref{harmExist}, $\sup_k\|\log \L_{\wh\phi}^k1\|_{L^\infty(\wh m)}<\infty $. 
\end{proof}

\begin{remark}\label{noContHarm}
	Note, in addition to Theorem \ref{smoothIFF}, if $\wh P(\wh\phi)>0$, there cannot be a $\wh \phi$-harmonic function in $L^1(\wh m)$. If there were such $\wh\rho$, then for all $n\geq0$,
	$$\int \wh \rho d\wh m=\int \L^n_{\wh\phi}\wh\rho d\wh m=e^{-n \wh P(\wh\phi)} \int \wh \rho d\wh m \to 0.$$
\end{remark}

\subsection{Volume Decay of Correlations and the Spectral Gap}\label{appliC2}
 Recall the setting of \textsection \ref{FinalLY}: Let $m$ be the normalized Riemannian volume on $M$- a closed Riemannian manifold of dimension $d\geq2$ (i.e. a smooth measure), with its associated log-Jacobian function $\wh\varphi\in \mathrm{H\ddot{o}l}(\Sigma^+\times M)$.  Assume that $f_\omega=f_{(\omega_i)_{i\geq0}}$, and that $\{f_\omega\}_{\omega\in \Sigma}\subseteq \mathrm{Diff}^{1+\gamma}(M)$, $\gamma>0$, where $\omega\mapsto f_\omega$ is a H\"older continuous map. Let $\psi$ be a potential in a proper form, and let $\mu$ be its associated Gibbs state. Let $\wh\phi:=\psi+\wh\varphi$. %Finally, denote $M_f:=\sup_{x\in M,\omega\in \Sigma}\{\|d_xf_\omega\|,\|d_xf_\omega^{-1}\|\}$.

In Theorem \ref{FinalLYthm} we prove that under the assumption of effective expansion on average %, unstable holonomies and $\wh\phi$-holonomies 
(recall Definition \ref{UEA2}% and Definition \ref{StableHolonomies}, and see Remark \ref{HolsAreOpen}
), % or alternatively if $f_\omega=f_{\omega_0}$, then 
$\mathcal{L}_{n}$ admits a Lasota-Yorke inequality when acting on the Sobolev space $\Hh_{\kappa}(m)$, when $\kappa-d>0$ is small enough. In Corollary \ref{QCasr} we conclude quasi-compactness for the averaged semi-Ruelle operators (recall Definition \ref{ASO}).

\begin{definition}[Volume Decay of Correlations for H\"older functions]\label{HolExpConv}
		We say that $\wh m$ admits {\em volume decay of correlation  on H\"older functions} with $\wh \mu\in \mathbb{P}(\Sigma^+\times M)$ if for all $\theta>0$ small enough there exist $C,\tau>0$ s.t. 
 for all $g,h\in \mathrm{H\ddot{o}l}_\theta(M)$, for all $n\geq0$,
\begin{equation}\label{forSobLater}
    \Big|\int g\circ \wh F^n hd\wh m-\int g d\wh \mu \cdot\int h dm\Big|\leq %Ce^{-\tau n}
    a_n\|g\|_{\mathrm{H\ddot{o}l}_\theta} \|h\|_{\mathrm{H\ddot{o}l}_\theta},
\end{equation}
%	$$\int g\circ \wh F^n hd\wh m-\int g d\wh \mu \cdot\int h dm\xrightarrow[]{n\to\infty}0,$$
	where we view $g$ and $h$ as functions on $\Sigma^+\times M$ by the natural extension and $a_n\to 0$.
\end{definition}

\begin{remark}
    Note, assuming that there exist $C,\tau>0$ s.t. for all $g,h\in \mathrm{H\ddot{o}l}_\theta(M)$ with $\int h dm=0$, for all $n\geq0$,
\begin{equation*}
    \Big|\int g\circ \wh F^n hd\wh m\Big|\leq Ce^{-\tau n}
    \|g\|_{\mathrm{H\ddot{o}l}_\theta} \|h\|_{\mathrm{H\ddot{o}l}_\theta},
\end{equation*}
Implies volume decay of correlations with some $\wh \mu\in \mathbb{P}(\Sigma^+\times M)$. See \cite{ExpFastVolLimits}.
\end{remark}

Recall Lemma \ref{AvgPwr} and Remark \ref{RmkAvgPwr} for the notation $\mathcal L$.

\begin{theorem}[i.i.d. spectral gap]\label{SpecGapIID}
		In the setting of Theorem \ref{FinalLYthm}, assume that $\wh m$ admits volume decay of correlations with a measure $\wh \mu$, and assume further that that $\mu$ is a Bernoulli measure with $f_\omega=f_{\omega_0}$. Then,
			$$\Ll= K^*+R^*,$$
	where $K^* h=\int h dm\cdot \wh \mu$ with  $\wh\mu \in \Hh_{-s}(m)$,  and the spectral radius of $R^*: \Hh_{-s}(m)\to \Hh_{-s}(m) $ is less than $1$.
    In particular, $\wh \mu \in \Hh_{-s}(m)$, and 
    	$$\Uu= K+R,$$
        where $K g=1\cdot \int g d\wh \mu$,  and the spectral radius of $R: \Hh_{s}(m)\to \Hh_{s}(m) $ is less than $1$.
\end{theorem}
\begin{proof}
Given $h,g\in \mathrm{H\ddot{o}l}_\theta(m)$, with $\theta\in (0,1)$ given by the volume decay of correlations assumption, for all $n\geq0$, 
\begin{align}\label{toRefEq30}
\int g \mathcal L_{n} h d m= &\int g\int \L_{\wh\phi}^{n}hd\mu d m=\int g \L_{\wh\phi}^{n}h d\wh m= \int g\circ \wh F_L^{n} h d\wh m\nonumber\\
=& \int %P_{\wh\mu }g\cdot 
gd\wh \mu\cdot\int h  dm%\pm Ce^{-\tau nN}\|g\|_{\Hh_\kappa(m)} \|h\|_{\Hh_\kappa(m)}
\pm a_n \|g\|_{\mathrm{H\ddot{o}l}_\theta} \|h\|_{\mathrm{H\ddot{o}l}_\theta}. 
\end{align}

In addition, by Corollary \ref{QCasr} and Lemma \ref{AvgPwr}, for all $n\geq 0$
\begin{align}\label{toRefEq31}
	\int g \mathcal L_{n} h d  m=&\int g (K^*)^n h  dm\pm \|(R^*)^n\|_{\Hh_{-s}(m)\to\Hh_{-s}(m)}\cdot \|g\|_{-s} \|h\|_{-s},
\end{align}
where $\Ll=K^*+R^*$. We continue to compare the l.h.s.'s of \eqref{toRefEq30} and \eqref{toRefEq31}:
\begin{equation}\label{eq35}
	\int %P_{\wh\mu }
	gd\wh\mu\cdot\int   h  dm = \int g (K^*)^n h  dm+o(1).
\end{equation}

We recall that $K^*$ is obtained in Corollary \ref{QCasr} via the spectral decomposition of the Hennion theorem, and so it is of the form $K^*=\sum_{j=1}^L(\lambda_j \Pi_j+N_j)$ where $\Pi_j$'s are finite-rank projections, $\lambda_j$'s are scalars, and $N_j$'s are nilpotent operators, and moreover $(\lambda_j \Pi_j+N_j) (\lambda_i \Pi_i+N_i)=0 $ for all $i\neq j$; And $\Pi_j N_j= N_j \Pi_j=N_j$ for all $j\leq L$. In addition, the images of the $N_j$'s are contained in Riesz subspaces which are mutually orthogonal. Assume that the $\lambda_j$'s are given by an increasing order of their moduli. We can also assume w.l.o.g. that the moduli are greater or equal to one, as the ones smaller than one can be absorbed into $R$.

We claim first by \eqref{eq35} %and \eqref{eq36} 
that $K_1= \lambda_1 \Pi_1$. Indeed, if $N_1\neq 0$, then $K^n=\sum_{j\leq L} (\lambda_j \Pi_j+N_j)^n=\sum_{j\leq L}\sum_{i=0}^{r_j}\binom{n}{i}\lambda_j^{n-i} N_j^i$, where $r_j$ is the nilpotency rank of $N_j$ and we set $N_j^0=\Pi_j$. Therefore, if $N_1\neq 0$, we would have a block with a polynomial growth in $n$ for $\int g K_1^n h dm$, which contradicts our estimate of the r.h.s. of \eqref{eq35} %and of \eqref{eq36} 
being bounded up to an exponential error term. While we do not know a-priori that $\mathrm{Im}(N_1)$ contains H\"older functions, we can find a sufficiently close H\"older function, and since $|\lambda_1|$ is maximal, it is enough to have a mere non-trivial component of the H\"older function in $\mathrm{Im}(N_1)$. While we assume that the $\lambda_j$'s admit no modulus smaller than $1$, in fact they cannot be greater than $1$ either, as we also would have growth. 
%Once concluding that $K_1= P_{\wh \mu}^*$ (and in fact for all $j$'s of maximal modulus for $\lambda_j$), we are done as all other $\lambda_j$'s must have modulus less than one. 
Therefore $K=\sum_{j\leq L}\lambda_j\Pi_j$ where the $P_j$'s are finite-rank projections and the $\lambda_j$'s have modulus 1; And finally 
\begin{equation}\label{eq37}
	(K^*) ^n=\sum_{j\leq L}\lambda_j^n\Pi_j.
\end{equation}

Putting \eqref{eq37} back with \eqref{eq35}, we get that for all $g,h\in \mathrm{H\ddot{o}l}_\theta(M)$,
$$\int %P_{\wh\mu }
gd\wh\mu\cdot \int   h  dm = \int g \sum_{j\leq L}\lambda_j^n\Pi_j h  dm+o(1).$$

It follows that $\lambda_j=1$ for all $j\leq L$, and so $K^*=\Pi$ for some finite-rank projection $\Pi$, and in particular, for all $g,h\in \mathrm{H\ddot{o}l}_\theta(M)$,
\begin{equation}\label{eq38}\int %P_{\wh\mu }
gd\wh \mu\cdot\int  h  dm = \int g \Pi h  dm.
\end{equation}	
Then, by fixing any $h_0$ with $\int h_0 dm=1$, we get for all $g\in \mathrm{H\ddot{o}l_\theta}(M)$,
\begin{equation}\label{eq39}\int %P_{\wh\mu }
gd\wh \mu = \int g \Pi h_0  dm.
\end{equation}
From \eqref{eq39} it follows that $\wh \mu= \Pi h_0 \in \Hh_{-s}(m)$ when tested on functions on $M$. Since this is independent of the choice of $h_0$, we get that $\Pi$ is of rank one, and $\Pi h=\int h dm \cdot \wh\mu$.

Note, $\int h dm$ is well-defined for elements of $\Hh_{-s}(m)$ as it can be expressed as $h(1)$, $1\in \Hh_s(m)$.

Similarly, we study $\Uu=\Ll^*=K+R$, and conclude
$$\int Kg h dm=\int g d\wh  \mu \int hdm,$$
and deduce $Kg=1\cdot \int gd\wh\mu$.
\end{proof}

\begin{cor}[Dimension bounds for the stationary measure]\label{CorHausDim}
$$\mathrm{dim}_\mathrm{Hausdorff}(\wh \mu)\geq d-2s.$$
\end{cor}
\begin{proof}
 By Theorem \ref{SpecGapIID}, $\wh \mu\in \Hh_{-s}(m)$ is a bounded linear functional on $\Hh_s(m)$, and so $\mathrm{dim}_\mathrm{Hausdorff}(\wh \mu)\geq d-2s$.
\end{proof}

\appendix

\section{Example with no integrable harmonic function}\label{AppCount}

\begin{lemma}
	There exists $\psi\in \mathrm{H\ddot{o}l}(\mathbb{S}^1)$ which satisfies $\int \psi d\lambda=0$, where $\lambda$ is the normalized Lebesgue measure on $\mathbb{S}^1$, and s.t. $\psi$ is not a measurable coboundary for $R_\alpha^{-1}$, where $R_\alpha$ is the rotation by $\alpha\in (0,1)\setminus\mathbb{Q}$ on $\mathbb{S}^1$.
\end{lemma}
\begin{proof}
	By \cite{Hillel}, there exists $\psi\in \mathrm{H\ddot{o}l}(\mathbb{S}^1)$ which satisfies $\int \psi d\lambda=0$ s.t. 
	$$T_\psi(t,s):=(t-\alpha, s-\psi(t)) \ \pmod 1, \ T_\psi:\mathbb{T}^2\to \mathbb{T}^2,$$
	is strictly ergodic w.r.t. $\lambda\times \lambda$.
	
	Assume for contradiction that there exists $u:\mathbb{S}^1\to\mathbb{R}$ measurable s.t. $$\psi=u-u\circ R_\alpha^{-1}.$$
	Set $$H(t,s):=e^{2\pi i (s-u(t))}.$$
$H$ is a measurable non-constant function on $\mathbb{T}^2$. Note,
\begin{align*}
	H\circ T_\psi(t,s)=& e^{2\pi i (s-\psi(t)-u(t-\alpha))}= e^{2\pi i (s-\psi(t)+\psi(t)-u(t))}\\
	=& e^{2\pi i (s-u(t))}=H(t,s).
\end{align*}	
	Then $H$ is $T_\psi$-invariant. This is a contradiction to $T_\psi$ being strictly ergodic and $H$ being non-constant. Therefore $\psi$ is not a measurable coboundary.
\end{proof}

\begin{prop}
Let $\Sigma=\{0,1\}^{\mathbb{Z}}$ and $X:=\mathbb{S}^1$. Let $\alpha\in (0,1)\setminus \mathbb{Q}$, and set $F_\omega(t):=t+\alpha \ \mathrm{mod} \ 1$ (i.e. product dynamics). Set $\wh \phi(\omega,t):=\psi(t)$, where $\psi\in \mathrm{H\ddot{o}l}(\mathbb{S}^1)$ satisfies $\int \psi d\lambda=0$ where $\lambda$ is the normalized Lebesgue measure on $\mathbb{S}^1$. Assume further that $\psi$ is not a measurable coboundary for $R_\alpha^{-1}$, where $R_\alpha$ is the rotation by $\alpha$ on $\mathbb{S}^1$. Then for any $\wh \phi$-conformal measure $\wh p$, there exists no $\wh\phi$-harmonic function $\wh\rho\in L^1(\wh p)$.
\end{prop}
\begin{proof}
	First, we note that $\wh\phi$-holonomies are satisfied trivially. We then continue to compute that POE:
\begin{align*}
	\wh P(\wh\phi)=&\limsup\frac{1}{n}\log \|\L_{\wh\phi}^n1\|_\infty=\limsup\frac{1}{n}\sup_t \sum_{|\ul w|=n}e^{\psi^{(n)}(t)}\\
	=&\log2+ \limsup \sup_t \frac{1}{n}\psi^{(n)}(t).
\end{align*}
Since $(\mathbb{S}^1,\lambda,R_\alpha)$ is uniquely ergodic, $\frac{1}{n}\psi^{(n)}(t) \to\int \psi d\lambda=0$ uniformly in $t$.

Let $\wh p$ be an $\wh \phi$-conformal probability measure, i.e. $\L_{\wh\phi}^*\wh p=2\wh p$. Project $\wh p$ to $\tau\in \mathbb{P}(\mathbb{S}^1)$. Given any $g\in C(\mathbb{S}^1)$, write $\wh g(\omega,t):=g(t)$, and so
\begin{align}\label{justForhPos}
	\int g d\tau=&\int \wh g d\wh p= \int \frac{1}{2}\L_{\wh\phi}\wh g d\wh p=\int \frac{1}{2}\sum_{a\in\{0,1\}}e^{\psi\circ R_\alpha(t)}g(R_\alpha (t))d\wh p\nonumber\\
	=&	\int e^{\psi\circ R_\alpha}g \circ R_\alpha d\tau.
\end{align}
Then $\tau=e^\psi\cdot \tau\circ R_\alpha^{-1}$. If $\tau$ were not singular to $\lambda$, then by the Lebesgue decomposition theorem there would exist a density $0\leq h\in L^1(\lambda)$ s.t. $\tau\gg h\cdot\lambda$. Then, $h\cdot\lambda=(e^\psi \cdot R_\alpha^{-1} )\cdot \lambda\circ R_\alpha^{-1}= (e^\psi \cdot R_\alpha^{-1} )\cdot \lambda$. Hence, $\lambda$-a.e.
$$\psi=\log h-\log h\circ R_\alpha^{-1}.$$
We note that $[h>0]$ is an $R_\alpha$-invariant set by \eqref{justForhPos}, and so by the ergodicity of $\lambda$, indeed $\log h$ is well-defined $\lambda$-a.e. This is a contradiction to our assumption, and so $\tau$ is singular to $\lambda$.

Indeed, if there existed $\wh\rho\in L^1(\wh p)$ which is $\wh\phi$-harmonic, then $\wh \nu:=\wh\rho\cdot \wh p$ would be $\wh F_R$-invariant (recall Lemma \ref{InvExists}). In this case, $\wh \nu$ projects to an $R_\alpha$-invariant probability on $\mathbb{S}^1$, but unique ergodicity guaranties that it must be $\lambda$. Thus, for $\pi_X:\Sigma^-\times X\to X$ being the projection,
$$\tau=\wh p\circ \wh \pi_X^{-1}\gg(\wh\rho\cdot \wh p)\circ \wh \pi_X^{-1}=\lambda,$$
a contradiction!
\end{proof}

\section{Co-expansion on average implies effective expansion on average with $d^*=d$}\label{coExImpEffEx}

\begin{lemma}[Integral duality]\label{BeautifulFourier}
	For all $\kappa\in (0,d)$ there exists $C_{d,\kappa}>0$ s.t. for all $A\in\mathrm{GL}_d(\mathbb{Z})$, 
	$$\int |A\xi|^{-(d-\kappa)}d\sigma_{d-1}(\xi)= C_{d,\kappa}\int |\mathrm{det}A|^{-1}\cdot|A^{-t}\eta|^{-\kappa}d\sigma_{d-1}(\eta),$$
where $\sigma_{d-1}$ is the normalized Lebesgue measure on $\mathbb{S}^{d-1}$, $A^{-t}$ is the inverse transposed matrix.
\end{lemma}
\begin{proof}
Write $g(\xi):=|\xi|^{-(d-\kappa)}$. Then the Fourier transform is $\wh g (\eta)= C'_{d,\kappa}|\eta|^{-\kappa}$. $g_A(\xi):=g\circ A(\xi)=|A\xi|^{-(d-\kappa)}$, and so
\begin{equation}\label{theFourier}
\wh g_A (\eta)= C'_{d,\kappa} |\mathrm{det}A|^{-1} \cdot|A^{-t}\eta|^{-\kappa}.
\end{equation}
 By the Parseval identity, with $h(\xi)=e^{-\frac{|\xi|^2}{2}}$ which satisfies $\wh h(\eta)=e^{-\frac{|\eta|^2}{2}}$, we have
$$\int g_A(\xi)  \wh h(\xi) d \xi= \int \wh g_A(\eta)  h(\eta) d \eta,$$
which yields,
$$\int |A\xi|^{-(d-\kappa)} e^{-\frac{|\xi|^2}{2}}  d \xi= \int C'_{d,\kappa} |\mathrm{det}A|^{-1} \cdot |A^{-t}\eta|^{-\kappa}  e^{-\frac{|\eta|^2}{2}} d \eta.$$
Passing to polar coordinates, we get 
$$\int |A\xi|^{-(d-\kappa)} d\sigma_{d-1}(\xi)= C_{d,\kappa}\int |\mathrm{det}A|^{-1} \cdot |A^{-t}\eta|^{-\kappa}  d\sigma_{d-1}(\eta).$$
\end{proof}

\begin{prop}\label{FourierBounds}
If $\{f_\omega\}_{\omega\in \Sigma}$ are %volume preserving 
diffeomorphisms on $M$ where $\mathrm{dim}M=d\geq 2$, and $\kappa\in (0,d)$, then $\exists \wh C_{\kappa,d}>1$ s.t. for all $x\in M$ and for all $n\geq 0$,
if 
$$G_n(x,\xi):=\int  |d_xf_\omega\xi|^{-(d-\kappa)}d\mu,$$
then 
$$\max_{\xi\in T_xM,|\xi|=1} G_n(x,\xi)\leq C_0 e^{-\chi n}\Rightarrow \max_{\eta\in T_x^*M,|\eta|=1} \wh G_n(x,\eta)\leq\wh C_{d,\kappa} C_0 e^{-\chi n}.$$
\end{prop}
\begin{proof}
Fix $x\in M$. Then for all $\eta\in \mathbb{S}^{d-1}$ (recall \eqref{theFourier} in the proof of Lemma \ref{BeautifulFourier}, together with the affinity of the Fourier transform operator),
\begin{equation}\label{formOfFourier}
	\wh G_n(\eta)= C_{d,\kappa} \int \Jac_x(f_\omega^n)^{-1}|(d_xf_\omega^n)^{-t}\eta|^{-\kappa}d\mu,
\end{equation}
which is positive, and so this proposition is proper. Set,
$$\epsilon_n= \max_{\xi\in T_xM,|\xi|=1} G_n(x,\xi). $$

We expand on the proof of Lemma \ref{BeautifulFourier} by choosing $h_{\xi_0,R}$, an approximate unit centered at $\eta_0$, for $\eta_0\in \mathbb{S}^{d-1}\simeq T_x^*M$ with $|\eta|=1$.

Choose a smooth, non-negative bump function $h \in C_c^\infty(\mathbb{R}^d)$ that is radially symmetric, supported entirely within the unit ball $B(0, 1)$, and normalized such that $\int h(\eta) d\eta = 1$. Given $\eta_0\in \mathbb{S}^{d-1}$ set for $R=\frac{1}{2}$,
$$h_{\eta_0,\epsilon}(\eta) = \frac{1}{R^d} h\Big(\frac{\eta - \eta_0}{R}\Big)$$
By construction, $h_{\eta_0,R}$ is supported in $B(\eta_0, R)$ and integrates to 1.

Note, $t \mapsto |t|^{-\kappa}$, and so, since linear composition preserves convexity, the map $\eta \mapsto \Jac_x(f_\omega^n)|(d_xf_\omega^n)^{-t} \eta|^{-\kappa}$ is convex (in $\eta$). Since $\wh G_n(\eta)$ is an integral of convex functions, $\wh G_n(\eta)$ is also convex.

Therefore, as $h_{\eta_0,R}$ is symmetric around $\xi_0$ and integrates to 1, we can apply Jensen's inequality: 
\begin{align}\label{CoolJenConvPrime}
\wh G_n(\eta_0)=\wh G_n\Big(\int \eta \cdot h_{\eta_0,R}(\eta) d\eta\Big)\leq\int \wh G_n(\eta) h_{\eta_0,R}(\eta) d\eta.
\end{align}

Next, we compute $\wh h_{\eta_0,R}(\xi)$:
\begin{align*}
 \wh h _{\eta_0,R}(\xi) =& \int\frac{1}{R^d} h\left(\frac{\eta - \eta_0}{R}\right) e^{-2\pi i \xi \cdot \eta}  d\eta=\Big[\zeta=\frac{\eta-\eta_0}{R}\Big]\\ 	=&\int h(\zeta) e^{-2\pi i  \xi\cdot (\eta_0 + R \zeta) } d\zeta = e^{-2\pi i \xi\eta_0} \widehat{h}(R \xi), 
 \end{align*}
and so,
\begin{equation}\label{CoolJenConv2Prime}
	|\wh h_{\eta_0,R}(\xi)| = |\wh h(R \xi)|.
\end{equation}

Hence by \eqref{CoolJenConvPrime}, Parseval's identity, and \eqref{CoolJenConv2Prime},
\begin{align}\label{whyIFF}
0\leq \wh G_n(\eta_0)\leq &\int \wh G_n(\eta) h(\eta)  d\eta = \int G_n(\xi) \wh h_{\eta_0,R}(\xi)  d\xi\\
\leq &\max_{\xi'\in \mathbb{S}^{d-1}} G_n(\xi')\cdot  \int|\xi|^{-(d-\kappa)} |\widehat{h}(R \xi)| d\xi=\Big[\zeta=R\eta\Big]\nonumber\\
\leq &	\max_{\xi'\in \mathbb{S}^{d-1}} G_n(\xi')\cdot  \int R^{d-\kappa}|\zeta|^{-(d-\kappa)} |\widehat{h}(\zeta)| \frac{d\zeta}{R^d}\nonumber\\
=&R^{-\kappa} \max_{\xi'\in \mathbb{S}^{d-1}}\wh G_n(\xi')\cdot \int |\zeta|^{-(d-\kappa)} |\widehat{h}(\zeta)| d\zeta.\nonumber
\end{align}
Since we assumed $h\in C^\infty_c(\mathbb{R})$, $\widehat{h}(\zeta) $ is bounded near $0$, where $|\zeta|^{-(d-\kappa)}$ is integrable; And in addition $\widehat{h}(\zeta) $ decays super-polynomially at infinity, and so 
$$C_h:=\int |\zeta|^{-(d-\kappa)} |\widehat{h}(\zeta)| d\zeta \in (0,\infty).$$

Thus in total, for all $x\in M$, for all $\xi_0\in \mathbb{S}^{d-1}$, for all $n\geq0 $,
$$\int \Jac_x(f_\omega^n)^{-1}|(d_xf_\omega^n)^{-t}\eta_0|^{-\kappa}d\mu=\wh G_n(\eta_0)\leq 2^{\kappa}C_h C_{d,\kappa} \cdot \epsilon_n.$$
Set $\wh C_{d,\kappa}:= 2^{d-\kappa}C_h C_{d,\kappa} $.
\end{proof}

\begin{remark}\label{forExpLater}
	In the statement of Proposition \ref{FourierBounds}, the ``$\Rightarrow$" can in fact be improved to ``$\Leftrightarrow$", as the roles of $G_n$ and $\wh G_n$ can be replaced in \eqref{whyIFF}.
\end{remark}

\begin{cor}[Co-expansion on average implies effective expansion on average with $d^*=d$]\label{ddstarcor}
	If a Bernoulli measure $\mu$ is co-expanding on average:
	$$\min_{(x,\xi)\in T^1M}\int\log |(d_xf_\omega)^{-t}\xi|d\mu>0,$$
	then $$d^*=d.$$
\end{cor}
\begin{proof}
	Follows from Remark \ref{forExpLater}, and substituting $\Jac_x(f_\omega)=1$ in \eqref{formOfFourier}, for all $\kappa>0$ small enough.
\end{proof}

\begin{remark}
	If $\mu$ is a Gibbs measure, by \cite[Theorem~8.12]{TDFOE_I}, 
	$$\min_{\omega^-\in\Sigma^-,(x,\xi)\in T^1M}\int\log |(d_xf_\omega)^{-t}\xi|d\mu_{\omega^-}>0$$
also implies $d^*=d$.
\end{remark}

\section{Proof completion of Theorem \ref{FinalLYthm} via pseudodifferential operator theory}\label{appPseudoDiffKerOps}

We assume that there exist constants $\chi,C_0> 0$ and $D_f > 1$ such that for all $n \geq 1$ and all $x \neq y$ satisfying $d(x,y) \leq D^{-n}_f$,
$$Q_n(x,y) \leq C_ 0e^{-\chi n},$$
where 
\begin{equation}\label{eq:Q_def}
Q_n(x,y):=\int\Big(\frac{d(x,y)}{d(f_\omega^n(x),f_\omega^n(y))}\Big)^{d-2s}d\mu=M_f^{\pm n(d-2s)},	
\end{equation}
with $s\in (0,1)$.

We consider the operator $T_n$ defined on $C^\infty(M)$:
\begin{equation*}
( T_n g )(x) = \int_{M\setminus \Delta} g(y) \frac{Q_n(x,y)}{d(x,y)^{d-2s}} dm(y)
\end{equation*}
where $\Delta = \{(x,x) : x \in M\}$ is the diagonal set of measure zero.

we wish to prove that for all $\epsilon>0$ sufficiently small, for all $g\in \Hh_{-s}(m)$,
\begin{equation}\label{main_ly}
    \langle g, T_n g \rangle_{L^2(m)} \leq  B^{n+1} \|g\|_{-s-\epsilon}^2+ B e^{-\chi n} \|g\|_{-s}^2, 
\end{equation}
 where $\wh C_{d,s}$ is a global dimensional constant, and $B> 0$ is a global constant. 

\begin{proof}

Let $\psi \in C^\infty([0, \infty))$ be a smooth bump function which satisfies $\psi|_{[0,1]}= 1$, $\psi|_{[2,\infty)} = 0$ , and $0\leq \psi\leq 1$. Set $\Psi_n(x,y) = \psi(D_f^n d(x,y))$. We split $T_n = T_n^{\mathrm{near}} + T_n^{\mathrm{far}}$ using kernels:

$$K_n^{\mathrm{near}}(x,y) = \frac{\Psi_n(x,y) Q_n(x,y)}{d(x,y)^{d-2s}} \mathbb{1}_{M \times M \setminus \Delta},$$
    and 
    $$ \quad K_n^{\mathrm{far}}(x,y) = \frac{(1-\Psi_n(x,y)) Q_n(x,y)}{d(x,y)^{d-2s}} \mathbb{1}_{M \times M \setminus \Delta}.$$

Fix $x \in M$. Given $y$ with $0<d(x,y)<r_0:=$injective radius of $M$, set  $\xi_{xy}= \exp_x^{-1}(y)\in T_x M \setminus \{0\}$ with $|\xi|=d(x,y)$. 

Note, since $\{f_\omega\}_{\omega\in \Sigma}$ is a pre-compact subset of $\mathrm{Diff}^{1+\gamma}(M)$, $\gamma>0$, we have for $d(x,y)\leq r_0 M_f^{-n}$,
$$
    \exp_{f_\omega^n(x)}^{-1} \left( f_\omega^n(\exp_x(\xi_{xy})) \right) = d_xf_\omega^n \xi + R_{\omega, n}(x, \xi_{xy}),$$
where the remainder satisfies $|R_{\omega, n}(x, \xi_{xy})|\leq C_{2,n} |\xi_{xy}|^{1+\gamma}$ uniformly in $\omega \in \Sigma$. Then,
$$d(f_\omega^n(x), f_\omega^n(\exp_x(\xi_{xy}))) = |d_xf_\omega^n\xi_{xy}| \left( 1 + E_{\omega, n}(x, \xi_{xy}) \right),$$
where $|E_{\omega, n}(x, z)| \leq C_{3,n} |z|^\gamma$. Substituting this into \eqref{eq:Q_def}:
$$Q_n(x, \exp_x(\xi_{xy})) = \int \frac{1}{|d_xf_\omega^n\frac{\xi_{xy}}{| \xi_{xy} |}|^{d-2s}} d\mu + O\left(|\xi_{xy}|^\gamma\right).$$
Note, the big O notation absorbs a coefficient that may grow exponentially in $n$,

Next, for all $x\in M$, for all $\xi\in T_xM\setminus \{0\}$,
$$G_n(x, \xi) := \int \frac{1}{|d_xf_\omega^n\xi|^{d-2s}} d\mu.$$
Our effective expansion assumption implies that 
$$G_n(x, \xi) \leq C_0 e^{-\chi n} |\xi|^{-(d-2s)}.$$

\begin{lemma}[Principal symbol bound]\label{symbol_bound}
$T_n^{\mathrm{near}}$ is a pseudodifferential operator with $C^\gamma$ spatial symbol (i.e. Fourier transform\footnote{of $(x,\xi)\mapsto \Psi_n(x,\exp_x\xi)\frac{Q_n(x,\exp_x \xi)}{|\xi|^{d-2s}}\Jac_\xi(\exp_x)$.} which depends on the reference point $x$) $p_n(x, \eta) \in C^\gamma S^{-2s}(T^* M)$ (i.e. $C^\gamma$ in $x$, and decays like $|\eta|^{-2s}$ in $\eta$, defined on the cotangent space $T^*M$). Its principal symbol (that is, the first term in the expansion by terms ordered by homogeneity degree) given by $\sigma_0(T_n^{\mathrm{near}})(x, \eta) =\wh G_n(x, \eta)$ satisfies $\forall (x, \eta) \in T^* M \setminus \{0\} $,
\begin{equation*}
    0 \leq \sigma_0(T_n^{\mathrm{near}})(x, \eta) \leq \wh C_{d,s} C_0 e^{-\chi n} |\eta|^{-2s},
\end{equation*}
where $\wh C_{d,s} $ is a constant depending only $d$ and $s$.
\end{lemma}
\begin{proof}
The principal symbol of an operator defined by a homogeneous kernel on tangent space is given by its fiberwise Fourier transform:
\begin{equation*}
    \sigma_0(T_n^{\mathrm{near}})(x, \eta) = \int_{T_x M} e^{-i \langle \xi, \eta \rangle_{T_xM}} G_n(x, \xi) d\xi.
\end{equation*}
By Proposition \ref{FourierBounds}, there exists a global constant $\wh C_{d,s}>0$ s.t.
%
%We can bound $0 \leq G_n(x, \xi) \leq  C_ 0e^{-\chi n} |\xi|^{-(d-2s)}$ as positive distributions. The Fourier transform of the homogeneous radial distribution $|\xi|^{-(d-2s)}$ on $\mathbb{R}^d$ equals $\wh C_{d,s} |\eta|^{-2s}$ for some constant $\wh C_{d,s}$. Because the Fourier transform preserves order for positive-definite kernels, we obtain 
$$\sigma_0(T_n^{\mathrm{near}})(x, \eta) \leq \wh C_{d,s} C_0 e^{-\chi n} |\eta|^{-2s}.$$

Finally, since $x \mapsto d_xf_\omega^n\xi$ is $\gamma$-H\"older, we get $x \mapsto G_n(x, \xi)$ and hence $x \mapsto \sigma_0(T_n^{\mathrm{near}})(x, \eta)$ are $C^\gamma$ in $x$. Thus, $T_n^\mathrm{near}$ is a pseudodifferential operator with a $C^\gamma$ spatial symbol.
\end{proof}

\begin{theorem}[Sharp G\r{a}rding inequality for $C^\gamma$ symbols {\cite[Chapter~VII]{SharpGarding}}]\label{garding_nonsmooth}
Let $P$ be a pseudodifferential operator whose spatial symbol is $C^\gamma S^{-2s}(T^*M)$ (that is, of order $-2s$ with symbol having $C^\gamma$ spatial regularity). If its principal symbol satisfies $\mathrm{Re} \sigma_0(P)(x, \eta) \geq 0$ for all $(x, \eta) \in T^* M$, then there exists a constant $C_P > 0$ such that for all $g \in C^\infty(M)$:
\begin{equation*}
    \mathrm{Re} \langle g, P g \rangle_{L^2(m)} \geq - C_P \|g\|_{-s-\frac{\gamma}{2}}^2.
\end{equation*}
\end{theorem}

Set $\Lambda=(I-\Delta_M)^\frac{1}{2}$  (the Bessel potential operator), with $\Delta_M$ being the Laplace-Beltrami operator.

\begin{lemma}[Bound for near terms]\label{near_est}
For all $g \in \Hh_{-s}(m)$, for all $\epsilon\in(0,\frac{\gamma}{2}]$,
$$\langle g, T_n^{\mathrm{near}} g \rangle_{L^2(m)} \leq C_0\wh C_{d,s} e^{-\chi n} \|g\|_{-s}^2 + A_f^{n+1} \|g\|_{-s-\epsilon}^2,$$
for a constant $A_f>0$ depending on the family $\{f_\omega\}_{\omega\in \Sigma}$.
\end{lemma}
\begin{proof}
Consider the difference operator $P_n = C_0\wh C_{d,s} e^{-\chi n} \Lambda^{-2s} - T_n^{\mathrm{near}}$. The operator $\Lambda^{-2s}$ has order $-2s$ and a smooth symbol $|\xi|^{-2s}$. 
By Lemma \ref{symbol_bound}, the principal symbol of $P_n$ satisfies
$$\sigma_0(P_n)(x, \xi) = C_0\wh C_{d,s} e^{-\chi n} |\xi|^{-2s} - \sigma_0(T_n^{\mathrm{near}})(x, \xi) \geq 0.$$
Applying Theorem \ref{garding_nonsmooth} to $P_n$ (being $\gamma$-H\"older implies being $2\epsilon$-H\"older for all $\epsilon\in (0,\frac{\gamma}{2}]$) we get,
$$\langle g, \left( \wh C_{d,s} e^{-\chi n} \Lambda^{-2s} - T_n^{\mathrm{near}} \right) g \rangle_{L^2(m)} \geq - C_{P_n} \|g\|_{-s-\epsilon}^2.$$
Since $\langle g, \Lambda^{-2s} g \rangle_{L^2(m)} = \|\Lambda^{-s} g\|_{L^2(m)}^2 = \|g\|_{-s}^2$, rearranging terms yields,
$$\langle g, T_n^{\mathrm{near}} g \rangle_{L^2(m)} \leq C_0\wh C_{d,s} e^{-\chi n} \|g\|_{-s}^2 + C_n \|g\|_{-s-\epsilon}^2.$$
We are left to show that there exists a global constant $A$ s.t. $C_n\leq A^{n+1}$. Indeed, the proof in \cite[Chapter~VII]{SharpGarding} shows that $C_{P_n}\leq \wt C_{d,s}\|p_n\|_{C^\gamma S^{-2s}} $, where $p_n\in C^\gamma S^{-2s}(T^* M) $ is its spatial symbol (recall Lemma \ref{symbol_bound}). Then $C_{P_n}\leq A_f^{n+1}$ for some $A_f$ which depends on $\{f_\omega\}_{\omega\in \Sigma}$.
\end{proof}

\begin{lemma}[Bound for far terms]\label{far_est}
For all $\epsilon \in( 0,\frac{\gamma}{2}]$, there exists a constant $\widetilde{C} > 0$ such that for all $g \in \Hh_{-s-\epsilon}(m)$,
$$\langle g, T_n^{\mathrm{far}} g \rangle_{L^2(m)} \leq \widetilde{C}^{n+1} \|g\|_{-s-\epsilon}^2.
$$
\end{lemma}
\begin{proof}
The kernel $K_n^{\mathrm{far}}(x,y)$ vanishes identically on $\{(x,y) \in M \times M : d(x,y) \leq D_f^{-n}\}$. 
We bound $T_n^{\mathrm{far}}: \Hh_{-s-\epsilon}(m) \to \Hh_{s+\epsilon}(m)$ by computing the $L^2(M) \to L^2(M)$ operator norm of the conjugated operator:
$$\widetilde{T}_n^{\mathrm{far}} = \Lambda^{s+\epsilon} T_n^{\mathrm{far}} \Lambda^{s+\epsilon}.$$
The integral kernel of $\widetilde{T}_n^{\mathrm{far}}$ is $\widetilde{K}_n^{\mathrm{far}}(x,y) = (\Lambda_x^{s+\epsilon} \Lambda_y^{s+\epsilon} K_n^{\mathrm{far}})(x,y)$.

Since $\mathrm{Supp}(K_n^{\mathrm{far}}) \subseteq \{(x,y) : d(x,y) \geq D_f^{-n}\}$, $K_n^{\mathrm{far}}$ is $C^{1+\gamma}$ off the diagonal. The operator $\Lambda_x^{s+\epsilon} \Lambda_y^{s+\epsilon}$ acts as a fractional derivative of order $2(s+\epsilon)$.

Recall, $K_n^\mathrm{far}(x,y)=\frac{(1-\Psi_n(x,y))Q_n(x,y)}{d(x,y)^{d-2s}}\mathbb{1}_{[x\neq y]}$ (and $(m\times m)([x\neq y])=1$). Bounding the derivatives of the bump function term: $|\nabla_{x,y}^k (1 - \Psi_n(x, y))| \leq \check{C}_{\psi,k}(D_f^n)^k$.

Applying $2s+2\epsilon$ derivatives to $d(x,y)^{-(d-2s)}$ increases the singular order of the kernel from $d-2s$ to $(d-2s) + (2s+2\epsilon) = d+2\epsilon$.

Then, since the minimum distance on $\mathrm{Supp}(K_n^{\mathrm{far}})$ is bounded below by $D_f^{-n}$, the kernel $\widetilde{K}_n^{\mathrm{far}}$ is everywhere bounded in $L^\infty(m \times m)$,
\begin{align*}
	\|\widetilde{K}_n^{\mathrm{far}}\|_{L^\infty(m\times m)} \leq& C_1 D_f^{n(d+2\epsilon)} \sup_{x \neq y} Q_n(x,y) \\
	\leq &C_1 D_f^{n(d+2\epsilon )} M_f^{n(d-2s)}  =: \widetilde{C}^{n+1},
\end{align*}
where $C_1$ is a constant depending on the derivatives of $\psi$, $d$, $s$, and $\epsilon$. For any Hilbert-Schmidt kernel operator,
$$\|\widetilde{T}_n^{\mathrm{far}}\|_{L^2(m) \to L^2(m)} \leq \|\widetilde{K}_n^{\mathrm{far}}\|_{L^2(m\times m)} \leq \|\widetilde{K}_n^{\mathrm{far}}\|_{L^\infty(m\times m)} \leq \widetilde{C}^{n+1}.$$
Setting $g = \Lambda^{s+\epsilon} h$ for $h \in L^2(M)$, we have $\|h\|_{L^2(m)}=\|g\|_{-s-\epsilon}$, and so
\begin{align*}
	\langle g, T_n^{\mathrm{far}} g \rangle_{L^2(m)} =& \langle h, \widetilde{T}_n^{\mathrm{far}} h \rangle_{L^2(m)} \leq \|\widetilde{T}_n^{\mathrm{far}}\|_{L^2(m) \to L^2(m)} \|h\|_{L^2(m)}^2\\
	=& \|\widetilde{T}_n^{\mathrm{far}}\|_{L^2(m) \to L^2(m)} \|g\|_{-s-\epsilon}^2\leq \widetilde{C}^{n+1} \|g\|_{-s-\epsilon}^2.
\end{align*}
\end{proof}

Combining Lemma \ref{near_est} and Lemma \ref{far_est}:
\begin{align*}
	\langle g, T_n g \rangle_{L^2(m)} = &\langle g, T_n^{\mathrm{near}} g \rangle _{L^2(m)} + \langle g, T_n^{\mathrm{far}} g \rangle _{L^2(m)}  \\
	\leq& C_0\wh C_{d,s} e^{-\chi n} \|g\|_{-s}^2 + (A_f^{n+1} + \widetilde{C}^{n+1}) \|g\|_{-s-\frac{\gamma}{2}}^2.
\end{align*}
Setting $B= A_f + \widetilde{C}+C_0\wh C_{d,s}$ yields \eqref{main_ly}.

\end{proof}

\bibliographystyle{alpha}
\tocless\bibliography{Elphi}

\Addresses

\end{document}